\documentclass[11pt]{article}

\usepackage{amssymb,amsmath,amscd,amsthm,xy,enumitem,mathrsfs,caption,pstricks}
\xyoption{all}

\newtheorem{thm}{Theorem}[section]

\newtheorem{prp}[thm]{Proposition}
\newtheorem{lmm}[thm]{Lemma}

\newtheorem{crl}[thm]{Corollary}

\theoremstyle{definition}
\newtheorem{dfn}[thm]{Definition}

\theoremstyle{remark}
\newtheorem{rmk}[thm]{Remark}

\numberwithin{equation}{section}

\def\lra{\longrightarrow}

\def\xlra#1{\xrightarrow{#1}}
\def\xlla#1{\xleftarrow{#1}}

\def\BE#1{\begin{equation}\label{#1}}
\def\EE{\end{equation}}
\def\lr#1{\langle#1\rangle}
\def\flr#1{\left\lfloor{#1}\right\rfloor}
\def\blr#1{\big\langle#1\big\rangle}

\def\wt#1{\widetilde{#1}}
\def\wh#1{\widehat{#1}}
\def\ov#1{\overline{#1}}
\def\eref#1{(\ref{#1})}
\def\tn#1{\textnormal{#1}}
\def\sf#1{\textsf{#1}}

\def\Ga{\Gamma}
\def\La{\Lambda}

\def\Si{\Sigma}

\def\al{\alpha}

\def\ga{\gamma}

\def\eps{\epsilon}
\def\io{\iota}

\def\na{\nabla}
\def\om{\omega}
\def\si{\sigma}
\def\th{\theta}
\def\ve{\varepsilon}
\def\vph{\varphi}

\def\vt{\vartheta}
\def\ze{\zeta}

\def\C{\mathbb C}
\def\cC{\mathcal C}  

\def\D{\mathbb D}
\def\fD{\mathfrak D}

\def\nd{\tn{d}}

\def\nE{\tn{E}}
\def\ne{\textnormal{e}}

\def\fI{\mathfrak i}
\def\cJ{\mathcal J}
\def\fJ{\mathfrak j}

\def\cL{\mathcal L}

\def\cM{\mathcal M}

\def\fM{\mathfrak M}

\def\cN{\mathcal N}
\def\cO{\mathcal O}

\def\P{\mathbb P}

\def\fP{\mathfrak P}
\def\R{\mathbb{R}}

\def\bT{\mathbb{T}}

\def\nV{\tn{V}}

\def\Z{\mathbb{Z}}

\def\Q{\mathbb Q}

\def\GL{\tn{GL}}
\def\GW{\textnormal{GW}}

\def\a{\mathbf a}
\def\fa{\mathfrak a}
\def\fd{\mathfrak d}

\def\fg{\mathfrak g}

\def\u{\mathbf u}

\def\fc{\mathfrak c}

\def\fj{\mathfrak j}
\def\ff{\mathfrak f}
\def\fo{\mathfrak o}
\def\fs{\mathfrak s}

\def\Cntr{\tn{Cntr}}

\def\tnd{\tn{d}}
\def\Edg{\tn{Edg}}
\def\ev{\tn{ev}}

\def\id{\tn{id}}

\def\ind{\tn{ind}}

\def\pt{\tn{pt}}
\def\rk{\tn{rk}}

\def\top{\tn{top}}
\def\val{\tn{val}}
\def\Ver{\tn{Ver}}
\def\vir{\tn{vir}}

\def\dbar{\bar\partial}
\def\prt{\partial}
\def\eset{\emptyset}
\def\i{\infty}

\def\w{\wedge}

\def\bu{\bullet}

\def\0{\mathbf 0}

\begin{document}

\title{Geometric Properties of Real Gromov-Witten Invariants}
\author{Penka Georgieva\thanks{Partially supported by ERC Grant ROGW-864919} 
$~$and 
Aleksey Zinger\thanks{Partially supported by NSF grant DMS 2301493}}
\date{\today}
\maketitle

\begin{abstract}
\noindent
We collect geometric properties of the
all-genus real Gromov-Witten theory and provide updates on its development
since its introduction in~2015.
We bring attention to a modification of the original construction
of this theory which is applicable in a broader setting,
obtain properties of the orientations of the moduli spaces of real maps
under this modification which had previously been obtained for the original construction,
and compare such orientations obtained via the two constructions when
both constructions are applicable.'
In particular, the orientation of the moduli space of real maps from
a disjoint union of domains may not be the product orientation
of the  moduli spaces of real maps from its components,
a possibility overlooked in our past~work.
\end{abstract}

\tableofcontents

\section{Introduction}
\label{intro_sec}

\noindent
The foundations of {\it complex} Gromov-Witten (or GW-) theory, i.e.~of counts 
of $J$-holomorphic curves in symplectic manifolds, were established in the 1990s 
and have been spectacularly applied  in symplectic topology,
algebraic geometry, and string theory since then.
On the other hand, the progress in {\it real} GW-theory, 
i.e.~of counts of $J$-holomorphic curves in symplectic manifolds 
preserved by anti-symplectic involutions, has been much slower,
as its mathematical foundations in positive genera did not even exist 
until the mid 2010s.
Since the introduction of real positive-genera GW-theory in~\cite{RealGWsI},
its has been extended to broader settings in~\cite{GI} by modifying
the approach of~\cite{RealGWsI} to orienting the moduli spaces of stable real maps. 
The present paper obtains geometric properties of the resulting orientations
analogous to the geometric properties obtained in~\cite{RealGWsII} for
the orientations constructed in~\cite{RealGWsI} and
compares the orientations obtained via the approaches of~\cite{RealGWsI}
and~\cite{GI} when both are applicable.
We also update on developments in real GW-theory since its introduction
and correct some oversight in past work.
We hope that the present paper and its more algebraic companion~\cite{RealGWsAlg}
will facilitate further development of real GW-theory.\\

\noindent
A \sf{real manifold} is a pair $(X,\phi)$ consisting of a (smooth) manifold~$X$
and an involution~$\phi$ on~$X$, i.e.~$\phi$ is a diffeomorphism of~$X$ such that~$\phi^2\!=\!\id_X$.
In such a case, we denote by $X^{\phi}\!\subset\!X$ the fixed locus of~$\phi$.
A \sf{real bundle pair} $(V,\vph)$ over a real manifold~$(X,\phi)$ consists 
of a complex vector bundle $V\!\lra\!X$ and a conjugation~$\vph$ on $V$ lifting~$\phi$,
i.e.~$\vph^2\!=\!\id_V$ and
$\vph\!:V_x\!\lra\!\ov{V}_{\phi(x)}$ is a $\C$-linear isomorphism for every $x\!\in\!X$.
In such a case, \hbox{$V^{\vph}\!\lra\!X^{\phi}$} is a real vector bundle.
For a real bundle pair~$(V,\vph)$ over~$(X,\phi)$, we  denote~by
$$\La_{\C}^{\top}(V,\vph)=\big(\La_{\C}^{\top}V,\La_{\C}^{\top}\vph\big)$$
the top exterior power of $V$ over $\C$ with the induced conjugation.
Direct sums, duals, and tensor products over~$\C$ of real bundle pairs over~$(X,\phi)$
are again real bundle pairs over~$(X,\phi)$.\\

\noindent
If $(X,\phi)$ is a real manifold and $L\!\lra\!X$ is a complex vector bundle, the homomorphisms 
\begin{alignat*}{2}
\phi_L^{\oplus}\!:L\!\oplus\!\phi^*\ov{L}&\lra L\!\oplus\!\phi^*\ov{L}, 
&\quad\phi_L^{\oplus}(v,w)&=(w,v),  \qquad\hbox{and}\\
\phi_L^{\otimes}\!:L\!\otimes_{\C}\!\phi^*\ov{L}&\lra L\!\otimes_{\C}\!\phi^*\ov{L},
&\quad\phi_L^{\otimes}(v\!\otimes\!w)&=w\!\otimes\!v,
\end{alignat*}
are conjugations covering~$\phi$.
The homomorphism
$$\big(L\!\oplus\!\phi^*\ov{L}\big)^{\phi_L^{\oplus}}\lra L|_{X^{\phi}}, \quad
(v,w)\lra v,$$
of real vector bundles over~$X^{\phi}$ is then an isomorphism.
We will call the orientation on the domain of this isomorphism induced
from the complex orientation of the target \sf{the projection orientation}
of~$(L\!\oplus\!\phi^*\ov{L})^{\phi_L^{\oplus}}$.
If $\rk_{\C}\!L\!=\!1$, then 
\BE{Lasumisom_e} \big(L\!\otimes_{\C}\!\phi^*\ov{L},\phi_L^{\otimes}\big)
\lra \La_{\C}^{\top}\big(L\!\oplus\!\phi^*\ov{L},\phi_L^{\oplus}\big), \qquad
v\!\otimes\! w\lra -\fI(v,0)\!\w\!(0,w),\EE
is an isomorphism of real bundle pairs over~$(X,\phi)$.
The induced isomorphism
\begin{equation*}\begin{split}
&\big\{rv\!\otimes_{\C}\!v\!\in\!L\!\otimes_{\C}\!\phi^*\ov{L}\!:
v\!\in\!L|_{X^{\phi}},\,r\!\in\!\R\big\}
\!=\!\big(L\!\otimes_{\C}\!\phi^*\ov{L}\big)^{\phi_L^{\otimes}}\\
&\hspace{2in}
\lra \big(\La_{\C}^{\top}\big(L\!\oplus\!\phi^*\ov{L}\big)\!\big)^{\La_{\C}^{\top}\phi_L^{\oplus}}
\!=\!\La_{\R}^{\top}\big(\!\big(L\!\oplus\!\phi^*\ov{L}\big)^{\phi_L^{\oplus}}\big)
\end{split}\end{equation*}
of real line bundles over~$X^{\phi}$ intertwines the canonical orientation on the domain
(given by the standard orientation of~$\R$) and the projection orientation on 
the~target.

\begin{dfn}\label{realorient_dfn0}
Let $(X,\phi)$ be a real manifold.
A \sf{real orientation} $(L,[\psi],\fs)$ on a real bundle pair $(V,\vph)$ over~$(X,\phi)$
consists~of 
\begin{enumerate}[label=(RO\arabic*),leftmargin=*]

\item\label{LBP_it0} a complex line bundle $L\!\lra\!X$ such that 
\BE{realorient_e0} w_2(V^{\vph})=w_2(L)|_{X^{\phi}} \qquad\hbox{and}\qquad
\La_{\C}^{\top}(V,\vph)\approx\big(L\!\otimes_{\C}\!\phi^*\ov{L},\phi_L^{\otimes}\big),\EE

\item\label{isom_it0} a homotopy class~$[\psi]$ of isomorphisms 
of real bundle pairs in~\eref{realorient_e0}, and

\item\label{spin_it0} a spin structure~$\fs$ on the real vector bundle
$V^{\vph}\!\oplus\!L^*|_{X^{\phi}}$ over~$X^{\phi}$
compatible with the orientation induced by~$[\psi]$.

\end{enumerate}
\end{dfn}

\vspace{.1in}

\noindent
By~\ref{isom_it0}, a real orientation on $(V,\vph)$ determines a homotopy class of isomorphisms
\BE{Vorientisom_e}\La_{\R}^{\top}V^{\vph}=\big(\La_{\C}^{\top}V\big)^{\La_{\C}^{\top}\vph}
\approx \big(L\!\otimes_{\C}\!\phi^*\ov{L}\big)^{\phi_L^{\otimes}}
=\big\{rv\!\otimes_{\C}\!v\!\in\!L\!\otimes_{\C}\!\phi^*\ov{L}\!:
v\!\in\!L|_{X^{\phi}},\,r\!\in\!\R\big\}\EE
and thus an orientation on~$V^{\vph}$.
In particular, the real vector bundle $V^{\vph}\!\oplus\!L^*|_{X^{\phi}}$ over~$X^{\phi}$
in~\ref{spin_it0} is oriented.
By the first condition in~\eref{realorient_e0}, it admits a spin structure.
By \cite[Theorems~1.1,1.2]{RBP}, a real orientation~$(L,[\psi],\fs)$ 
on a  rank~$k$ real bundle pair~$(V,\vph)$ over a symmetric surface~$(\Si,\si)$,
possibly nodal and disconnected, determines a homotopy class of isomorphisms
\BE{RBPisom_e2}\big(V\!\oplus\!(L^*\!\oplus\!\si^*\ov{L}^*),\vph\!\oplus\!\si_{L^*}^{\oplus}\big)
\approx \big(\Si\!\times\!\C^{k+2},\si\!\times\!\fc\big)\EE
of real bundle pairs over~$(\Si,\si)$,
where $\fc$ is the standard conjugation on~$\C^{k+2}$; 
see Section~\ref{nota_sec} for the terminology.
This homotopy class of isomorphisms is one of the key ingredients 
in the construction of real GW-invariants in~\cite{RealGWsI};
see also \cite{RealGWsSumm} and Section~\ref{MSorient_subs1}.\\

\noindent
A \sf{real symplectic manifold} is a triple $(X,\om,\phi)$ consisting 
of a symplectic manifold~$(X,\om)$ and an anti-symplectic involution~$\phi$ on~$X$,
i.e.~\hbox{$\phi^*\om\!=\!-\om$}.
The fixed locus~$X^{\phi}$ of~$\phi$ is then a Lagrangian submanifold of~$(X,\om)$ and
$$TX^{\phi}=\big\{\dot{x}\!\in\!TX|_{X^{\phi}}\!:\nd\phi(\dot{x})\!=\!\dot{x}\big\}.$$
A \sf{real orientation} on a real symplectic manifold~$(X,\om,\phi)$
is a real orientation on the real bundle pair~$(TX,\nd\phi)$ over~$(X,\phi)$
in the sense of Definition~\ref{realorient_dfn0}.
We call a real symplectic manifold admitting a real orientation \sf{real-orientable}.
The anti-holomorphic involutions 
\begin{equation*}\begin{aligned}
\tau_n\!:\P^{n-1}&\lra\P^{n-1}, &
\tau_n\big([[Z_1,\ldots,Z_n]\big)&=\begin{cases}
[\ov{Z}_2,\ov{Z}_1,\ldots,\ov{Z}_n,\ov{Z}_{n-1}],&\hbox{if}~n\!\in\!2\Z;\\
[\ov{Z}_2,\ov{Z}_1,\ldots,\ov{Z}_{n-1},\ov{Z}_{n-2},\ov{Z}_n],&\hbox{if}~n\!\not\in\!2\Z;
\end{cases}\\
\eta_{2m}\!: \P^{2m-1}& \lra\P^{2m-1},&
\eta_{2m}\big([Z_1,\ldots,Z_{2m}]\big)&= 
\big[\ov{Z}_2,-\ov{Z}_1,\ldots,\ov{Z}_{2m},-\ov{Z}_{2m-1}\big],
\end{aligned}\end{equation*}
are anti-symplectic with respect to the Fubini-Study symplectic
forms~$\om_n$ and~$\om_{2m}$ on the complex projective spaces~$\P^{n-1}$ 
and~$\P^{2m-1}$, respectively.
The real symplectic manifolds $(\P^{n-1},\om_n,\tau_n)$ with $n\!\in\!2\Z$ and 
$(\P^{2m-1},\om_{2m},\eta_{2m})$ admit real orientations~$(L,[\psi],\fs)$ 
with \hbox{$L\!=\!\cO_{\P^{n-1}}(n/2)$} and \hbox{$L\!=\!\cO_{\P^{2m-1}}(m)$}, respectively.
Many other examples of compact real symplectic manifolds with real orientations
are described in \cite[Section~1.1]{RealGWsIII}; 
see also Section~\ref{RealGWth_subs} in the present paper.
These include many projective complete intersections,
such as the real quintic threefolds,
i.e.~the smooth hypersurfaces in~$\P^4$ cut out by degree~5 homogeneous polynomials
on~$\C^5$ with real coefficients;
they play a prominent role in the interactions of symplectic topology 
with string theory and algebraic geometry.\\

\noindent
For a symplectic manifold~$(X,\om)$, we denote by~$\cJ_{\om}$ 
the contractible space of $\om$-tamed almost complex structures on~$X$ and by 
\hbox{$c_1(X,\om)\!\in\!H^2(X;\Z)$} the first Chern class of the complex vector bundle~$(TX,J)$
over~$X$.
For a real symplectic manifold~$(X,\om,\phi)$, let
$$\cJ_{\om}^{\phi}=\big\{J\!\in\!\cJ_{\om}\!:\phi^*J\!=\!-J\big\}.$$
For $g\!\in\!\Z$, $\ell\!\in\!\Z^{\ge0}$, \hbox{$B\!\in\!H_2(X;\Z)$}, and $J\!\in\!\cJ_{\om}^{\phi}$,
we denote by $\ov\fM_{g,\ell}^{\phi;\bu}(B;J)$
the moduli space of stable real $J$-holomorphic degree~$B$ maps 
from closed, possibly nodal and disconnected, symmetric Riemann surfaces of arithmetic genus~$g$
with $\ell$~conjugate pairs of marked points.

\begin{thm}\label{main_thm}
Suppose $(X,\om,\phi)$ is a real-orientable $2n$-manifold with $n\!\not\in\!2\Z$
and $J\!\in\!\cJ_{\om}^{\phi}$.
\begin{enumerate}[label=(\arabic*),leftmargin=*]

\item A real orientation $(L,[\psi],\fs)$ on~$(X,\om,\phi)$ induces orientations  
on the moduli spaces $\ov\fM_{g,\ell}^{\phi;\bu}(B;J)$
for all \hbox{$g,\ell\!\in\!\Z^{\ge0}$} and \hbox{$B\!\in\!H_2(X;\Z)$}.

\item The induced orientations satisfy the properties stated in
Propositions~\ref{unionorient_prp}-\ref{RnodisomE_prp} and~\ref{RelSpinOrient_prp}.

\end{enumerate}

\end{thm}

\noindent
Under the assumptions of Theorem~\ref{main_thm}, 
the (virtual) dimensions of all moduli spaces $\ov\fM_{g,\ell}^{\phi;\bu}(B;J)$,
\BE{RfMdim_e}
\dim^{\vir}\ov\fM_{g,\ell}^{\phi;\bu}(B;J)
\equiv(1\!-\!g)(n\!-\!3)\!+\!2\ell\!+\!\blr{c_1(X,\om),B},\EE
are even. In particular, the product orientation on the Cartesian product of two such spaces
is independent of their ordering.\\

\noindent
The notion of real orientation provided by \cite[Definition~5.1]{RealGWsI} 
and used in~\cite{RealGWsI} to orient moduli spaces of stable real maps
replaces $(L\!\otimes_{\C}\!\phi^*\ov{L},\phi_L^{\otimes})$ and 
$(L^*\!\oplus\!\phi^*\ov{L}^*,\phi_{L^*}^{\oplus})$ in~\eref{realorient_e0} and~\eref{RBPisom_e2}
by $(L,\wt\phi)^{\otimes2}$ and $2(L,\wt\phi)^*$, respectively, 
for a real line bundle pair~$(L,\wt\phi)$ over~$(X,\phi)$.
The real vector bundle \hbox{$(2L)^{2\wt\phi}\!=\!2L^{\wt\phi}$} over~$X^{\phi}$
is then canonically oriented by taking the same orientation on both copies of~$L^{\wt\phi}$.
The homomorphism
\BE{GIgen_e} \Phi_L\!:2(L,\wt\phi)\lra\big(L\!\oplus\!\phi^*\ov{L},\phi_L^{\oplus}\big),
\quad \Phi_L(v,w)=\big(v\!+\!\fI w,\wt\phi(v\!-\!\fI w)\!\big),\EE
is an isomorphism of real bundle pairs over~$(X,\phi)$.
The restriction
\BE{GIgen_e2} \Phi_L\!:(2L)^{2\wt\phi}\lra \big(L\!\oplus\!\ov{L}\big)^{\phi_L^{\oplus}}\EE
intertwines the canonical orientation on its domain
and the projection orientation on its target.
Definition~\ref{realorient_dfn0}, 
which is a slightly reworded version of the notion of \sf{twisted (real) orientation data}
introduced in \cite[Definition~A.1]{GI},
thus weakens (broadens) the notion of real orientation of \cite[Definition~5.1]{RealGWsI}.
In particular, $(\P^{2m-1},\om_{2m},\eta_{2m})$ with $m\!\not\in\!2\Z$ admits 
a real orientation in the sense of Definition~\ref{realorient_dfn0} in the present paper,
but not in the sense of \cite[Definition~5.1]{RealGWsI}.
If $(X,\phi)\!=\!(\Si,\si)$ is a symmetric surface,
the compositions of the isomorphisms \hbox{$\id\!\oplus\!\Phi_{L^*}$}
and~\eref{RBPisom_e2} yields a homotopy class of isomorphisms
\BE{RBPisom_e3}\big(V\!\oplus\!2L^*,\vph\!\oplus\!2\wt\si^*\big)
\approx \big(\Si\!\times\!\C^{k+2},\si\!\times\!\fc\big)\EE
of real bundle pairs over~$(\Si,\si)$.\\

\noindent
As noted in~\cite{GI}, the construction of orientations on the moduli spaces of real maps
in~\cite{RealGWsI} goes through almost verbatim with the weaker notion of real orientation 
of Definition~\ref{realorient_dfn0}; see also Section~\ref{MSorient_subs1}.
However, the resulting orientations of the moduli spaces are generally different
when both constructions apply;
the orientations on $\ov\fM_{g,\ell}^{\phi;\bu}(B;J)$ arising from 
the two constructions in such cases are compared~via Lemma~\ref{CvsCanorient_lmm} with
$$\rk_{\C}L\!=1 \qquad\hbox{and}\qquad
\deg L=-\blr{c_1(X,\om),B}/2.$$
Propositions~\ref{unionorient_prp}, \ref{Dblorient_prp}, and~\ref{RelSpinOrient_prp}
compare the orientations of various moduli spaces under the orientating procedure
described in Sections~\ref{MSorient_subs1}.
Propositions~\ref{Rnodisom_prp} and~\ref{RnodisomE_prp} do the same for 
the codimension~2 strata of maps with a conjugate pair of nodes in the domain
and the codimension~1 strata of maps with an isolated real node, respectively.\\

\noindent
Propositions~\ref{Dblorient_prp}, \ref{Rnodisom_prp}, and~\ref{RelSpinOrient_prp}
are analogues of Theorems~1.4, 1.2, and~1.5, respectively, in~\cite{RealGWsII},
which apply to the orientating procedure used in~\cite{RealGWsI}
(the statements of these theorems are also incorporated into the statements of 
Propositions~\ref{Dblorient_prp}, \ref{Rnodisom_prp}, and~\ref{RelSpinOrient_prp}).
On the other hand, Proposition~\ref{unionorient_prp} concerns the behavior of
the orientations of the moduli spaces of real maps under 
the disjoint union of the domains of the~maps.
Contrary to the expectation, the orientation at the disjoint union of two maps
may not be the product of the orientations at the two~maps.
This possibility was overlooked in our past work, resulting in oversights in some proofs,
but fortunately these oversights do not impact any material statements; 
see Section~\ref{updates_sec}.
Proposition~\ref{RnodisomE_prp} likewise has no analogue in our past work;
it had not been needed for the previously completed work in real positive-genus GW-theory,
but is used in~\cite{RealGWsAlg}.\\

\noindent
Similarly to the GW-invariants of a compact symplectic manifold~$(X,\om)$,
the real GW-invariants of a compact real symplectic manifold~$(X,\om,\phi)$
with a real orientation~$(L,[\psi],\fs)$ are rational numbers,
but an invertible transformation determined by the topology of~$(X,\phi)$ is 
expected to convert these numbers into integers.
If the (real) dimension of~$X$ is~6, such a transformation is expected to be linear.
This is confirmed in~\cite{RealGWvsEnum} for ``Fano" classes
in a compact real symplectic sixfold with a real orientation,
adapting the result of \cite[Theorem~1.5]{FanoGV} from the complex GW-theory 
to the real GW-theory; see also Section~\ref{RealGWvsEnum_sub}.
The structure of such a transformation for the complex GW-invariants of 
a compact Calabi-Yau symplectic sixfold is predicted 
in~\cite{GV} based on string theory considerations,
mathematically motivated in~\cite{BryanP1b,BryanP2}, and confirmed in~\cite{IPgv}. 
The analogue of this transformation for the real GW-invariants of 
a compact Calabi-Yau symplectic sixfold with an anti-symplectic involution
is predicted in \cite{BFM,Walcher2} based on string theory considerations
and mathematically motivated in~\cite{GI};
its mathematical confirmation is the subject of~\cite{GI3}.
Such a transformation for the complex genus~1 GW-invariants of 
a compact Calabi-Yau symplectic tenfold is mathematically motivated in~\cite{CY5},
but its adaptation to the real setting is yet to be obtained.
Other important geometric problems in real GW-theory include the mirror symmetry predictions
of~\cite{Walcher2}.
More algebraic directions for future research are indicated in~\cite{RealGWsAlg}.\\

\noindent
We gather the most frequently used notation and terminology in Section~\ref{nota_sec}.
In Section~\ref{DistOrient_subs}, we describe the behavior of
natural orientations on the determinants of real Cauchy-Riemann operators
on real bundle pairs of the form $(L\!\oplus\!\si^*\ov{L},\si_L^{\oplus})$
over symmetric Riemann surfaces~$(\Si,\si)$ under various operations.
We use this to study the behavior of the orientations 
on the determinants of real Cauchy-Riemann operators induced by real orientations
in Section~\ref{InducedOrien_subs}
and the orientations on the moduli space of real maps in 
Sections~\ref{MSorient_subs1} and~\ref{MSorient_subs2},
establishing Propositions~\ref{unionorient_prp}-\ref{RnodisomE_prp} in the process.
In Section~\ref{MSorient0_subs}, 
we specialize to genus~0 and establish Proposition~\ref{RelSpinOrient_prp}.
In Section~\ref{updates_sec}, we extend other results from our past work to the broader setting
of Definition~\ref{realorient_dfn0} and correct some oversights.

\section{Notation and terminology}
\label{nota_sec}

\noindent
For $\ell\!\in\!\Z^{\ge0}$, let $[\ell]\!=\!\{1,2,\ldots,\ell\}$.
We define the \sf{arithmetic genus}~$g$ of a closed, possibly nodal and disconnected, 
Riemann surface~$(\Si,\fj)$ by 
$$g\!-\!1=\sum_{i=1}^m(g_i\!-\!1)$$
if $g_1,\ldots,g_m$ are the arithmetic genera of the topological components of~$\Si$.
A \sf{symmetric Riemann surface} $(\Si,\si,\fj)$  is a closed, possibly nodal and disconnected,
Riemann surface~$(\Si,\fj)$ with an anti-holomorphic involution~$\si$. 
For example, there are two topological types of such involutions on~$\P^1$:
\BE{tauetadfn_e}\tau,\eta\!:\P^1\lra\P^1, \qquad \tau(z)=\frac{1}{\bar{z}}, \quad
\eta(z)=-\frac{1}{\bar{z}}\,.\EE
If $\phi$ is an involution on another space~$X$,
a map $u\!:\!\Si\!\lra\!X$ is \sf{$(\phi,\si)$-real} 
(or \sf{$\phi$-real}, or just \sf{real}) if $u\!\circ\!\si\!=\!\phi\!\circ\!u$.
For $g\!\in\!\Z$ and $\ell\!\in\!\Z^{\ge0}$, we denote by $\R\ov\cM_{g,\ell}^{\bu}$ 
the Deligne-Mumford moduli space of stable closed, possibly nodal and disconnected, 
symmetric Riemann surfaces~$(\Si,\si,\fj)$ of arithmetic genus~$g$
with $\ell$~conjugate pairs $(z_i^+,z_i^-)$ of marked points.
This space is a smooth compact orbifold of dimension $3(g\!-\!1)\!+\!2\ell$;
it is empty if $g\!+\!\ell\!\le\!1$.
Let \hbox{$\R\ov\cM_{g,\ell}\!\subset\!\R\ov\cM_{g,\ell}^{\bu}$} be 
the topological component parametrizing connected marked curves. 
If in addition $(X,\om,\phi)$, $J$, and $B$ are as in Theorem~\ref{main_thm}, we denote~by
$$\ov\fM_{g,\ell}^{\phi}(B;J)\subset\ov\fM_{g,\ell}^{\phi;\bu}(B;J)$$
the subspace of maps from connected domains.\\

\noindent
Let $(\Si,\fJ)$ be a closed, but possibly nodal and disconnected, Riemann surface.
For $k\!\in\!\Z^{\ge0}$, we denote by $k\dbar_{\Si}$ the standard CR-operator 
on the complex vector bundle~$\Si\!\times\!\C^k$ over~$\Si$.
If $V$ is any complex vector bundle over~$\Si$ and
$$D_V\!:  \Ga(\Si;V)\lra
\Ga_{\fJ}^{0,1}(\Si;V)\!\equiv\!\Ga\big(\Si;(T^*\Si,\fJ)^{0,1}\!\otimes_{\C}\!V\big)$$
is a \textsf{generalized Cauchy-Riemann} (or \sf{CR-}) \sf{operator} on~$V$  
as in \cite[Section~2.2]{FanoGV},
the \sf{determinant line}
$$\det D_V\equiv\La_{\R}^{\top}(\ker D_V) \otimes \big(\La^{\top}_{\R}(\text{cok}\,D_V)\big)^*$$
of $D_V$ has a canonical \sf{complex} orientation;
see \cite[Appendix~A.2]{MS} and Proposition~5.9 in the 5th arXiv version of~\cite{detLB}.
The space of generalized CR-operators on a complex vector bundle~$V$ is affine.
By~\cite{detLB}, the determinant lines of these operators form a line bundle
over this space;
its topology depends on the coherent system of determinant line bundles used.\\

\noindent
If $(\Si,\fj)$ is a disjoint union of closed Riemann surfaces~$(\Si_1,\fj_1)$ and~$(\Si_2,\fJ_2)$ 
and $V_i\!=\!V|_{\Si_i}$ for $i\!=\!1,2$, then \hbox{$D_V\!=\!D_{V_1}\!\oplus\!D_{V_2}$}
for some generalized  CR-operators~$D_{V_1}$ and~$D_{V_2}$ on~$V_1$ and~$V_2$, respectively.
By the {\it Direct Sum} property in Section~2.2 of the 5th arXiv version of~\cite{detLB},
this decomposition induces an isomorphism
\BE{Cunionorient_e0}
\det D_V\approx\big(\!\det D_{V_1}\big)\!\otimes\!\big(\!\det D_{V_2}\big)\EE
between the corresponding determinant lines which depends continuously 
on~$D_{V_1}$ and~$D_{V_2}$. 
If instead \hbox{$V\!=\!V_1\!\oplus\!V_2$} for some complex vector bundles~$V_1,V_2$
over~$\Si$ and \hbox{$D_V\!=\!D_{V_1}\!\oplus\!D_{V_2}$}
for some generalized  CR-operators~$D_{V_1}$ and~$D_{V_2}$ on~$V_1$ and~$V_2$, respectively,
we similarly obtain an isomorphism~\eref{Cunionorient_e0} which depends continuously 
on~$D_{V_1}$ and~$D_{V_2}$. 
However, the isomorphism in~\eref{Cunionorient_e0} in either case depends 
on the coherent system of determinant line bundles used,
unless $D_{V_1}$ and~$D_{V_2}$ are either ($\C$-linear) CR-operators or surjective. 
By the continuous dependence, there is a canonical homotopy class 
of isomorphisms~\eref{Cunionorient_e0} in both cases even if \hbox{$D_V\!\neq\!D_{V_1}\!\oplus\!D_{V_2}$}.
Since the isomorphism~\eref{Cunionorient_e0} respects the complex orientations
if \hbox{$D_V\!=\!D_{V_1}\!\oplus\!D_{V_2}$} and $D_{V_1},D_{V_2}$ are CR-operators, 
this is also the case for all generalized CR-operators $D_V,D_{V_1},D_{V_2}$
on $V,V_1,V_2$, respectively.\\

\noindent
Let $(\Si,\si,\fj)$ be a symmetric Riemann surface.
For $k\!\in\!\Z^{\ge0}$, we denote by $k\dbar_{(\Si,\si)}$ the standard CR-operator 
on the real bundle pair $(\Si\!\times\!\C^k,\si\!\times\!\fc)$;
this is the restriction of~$k\dbar_{\Si}$ to the subspace of $\si$-invariant sections
with the range restricted to the subspace of $\si$-invariant $(0,1)$-forms.
If $(V,\vph)$ is a real bundle pair over a symmetric Riemann surface~$(\Si,\si)$ and
\begin{equation*}\begin{split}
D_{(V,\vph)}\!:
\Ga(\Si;V)^{\vph}
\equiv&\big\{\xi\!\in\!\Ga(\Si;V)\!:\,\xi\!\circ\!\si\!=\!\vph\!\circ\!\xi\big\}\\
&\qquad\lra
\Ga_{\fJ}^{0,1}(\Si;V)^{\vph}\equiv
\big\{\ze\!\in\!\Ga_{\fJ}^{0,1}(\Si;V)\!:\,
\ze\!\circ\!\tnd\si=\vph\!\circ\!\ze\big\}
\end{split}\end{equation*}
is a \textsf{real CR-operator} 
on~$(V,\vph)$ as in \cite[Section~2.2]{RealGWsII}, let
$$\det D_{(V,\vph)}\equiv\La_{\R}^{\top}(\ker D_{(V,\vph)}) 
\otimes \big(\La^{\top}_{\R}(\text{cok}\,D_{(V,\vph)})\big)^*$$
be the \sf{determinant} line of~$D_{(V,\vph)}$.
The space of real CR-operators on a real bundle pair~$(V,\vph)$ is affine.
By~\cite{detLB}, the determinant lines of these operators form a line bundle
over this space.
Thus, an orientation on the determinant line of one real CR-operator on~$(V,\vph)$ determines
an orientation of the determinant lines for all other such operators.\\

\noindent
If $(\Si,\si,\fj)$ is a disjoint union of symmetric Riemann surfaces~$(\Si_1,\si_1,\fj_1)$ 
and~$(\Si_2,\si_2,\fJ_2)$ and \hbox{$(V_i,\vph_i)\!=\!(V,\vph)|_{\Si_i}$} 
for $i\!=\!1,2$, then 
\BE{RDVsplit_e} D_{(V,\vph)}=D_{(V_1,\vph_1)}\!\oplus\!D_{(V_2,\vph_2)}\EE
for some real CR-operators~$D_{(V_1,\vph_1)}$ and~$D_{(V_2,\vph_2)}$ on~$(V_1,\vph_1)$ 
and~$(V_2,\vph_2)$, respectively.
By the {\it Direct Sum} property for determinant line bundles, this decomposition induces an isomorphism
\BE{Runionorient_e0}
\det D_{(V,\vph)}\approx\big(\!\det D_{(V_1,\vph_1)}\big)\!\otimes\!
\big(\!\det D_{(V_2,\vph_2)}\big)\EE
between the corresponding determinant lines which depends continuously 
on~$D_{(V_1,\vph_1)}$ and~$D_{(V_2,\vph_2)}$.
If instead 
\BE{RVsplit_e} (V,\vph)=(V_1,\vph_1)\!\oplus\!(V_2,\vph_2)\EE
for some real bundle pairs~$(V_1,\vph_1),(V_2,\vph_2)$ over~$(\Si,\si)$ and 
\eref{RDVsplit_e} holds for some real CR-operators $D_{(V_1,\vph_1)}$ 
and~$D_{(V_2,\vph_2)}$ on~$(V_1,\vph_1)$ and~$(V_2,\vph_2)$, respectively,
we similarly obtain an isomorphism~\eref{Runionorient_e0} which depends continuously 
on~$D_{(V_1,\vph_1)}$ and~$D_{(V_2,\vph_2)}$. 
There is again a canonical homotopy class of isomorphisms~\eref{Runionorient_e0} in both cases 
even if~\eref{RDVsplit_e} does not hold. 

\section{Comparisons of orientations}
\label{orient_sec}

\subsection{Distinguished orientations}
\label{DistOrient_subs}

\noindent
Suppose $(\Si,\si,\fJ)$ is a symmetric Riemann surface of arithmetic genus~$g$
and $L$ is a complex vector bundle  over~$\Si$.
If $D_L$ is a generalized CR-operator on~$L$, then
\begin{gather}
\label{DovLdfn_e}
D_{\si^*\ov{L}}\!:\Ga(\Si;\si^*\ov{L})\lra\Ga_{\fJ}^{0,1}\big(\Si;\si^*\ov{L}\big),
\quad D_{\si^*\ov{L}}(\xi)=D\xi\!\circ\!\nd\si, \quad\hbox{and}\\
\label{DLovLdfn_e}
D_{(L\oplus\si^*\ov{L},\si_L^{\oplus})}\!\equiv\!D_L\!\oplus\!D_{\si^*\ov{L}}\!:
\Ga\big(\Si;L\!\oplus\!\si^*\ov{L}\big)^{\si_L^{\oplus}}\lra
\Ga_{\fJ}^{0,1}\big(\Si;L\!\oplus\!\si^*\ov{L}\big)^{\si_L^{\oplus}}
\end{gather}
are a generalized CR-operator on~$\si^*\ov{L}$ 
and a real CR-operator on $(L\!\oplus\!\si^*\ov{L},\si_L^{\oplus})$, respectively.
The homomorphisms
$$\Ga(\Si;L)\lra\Ga(\Si;\si^*\ov{L}),~~\xi\lra\xi\!\circ\!\si, \quad
\Ga_{\fJ}^{0,1}\big(\Si;L)\lra\Ga_{\fJ}^{0,1}(\Si;\si^*\ov{L}),
~~\ze\lra\ze\!\circ\!\nd\si,$$
are isomorphisms and determine an isomorphism
\BE{ConjCorient_e}\det D_L\approx \det D_{\si^*\ov{L}}\EE
which depends continuously on~$D_L$.
By the continuous dependence, there is a canonical homotopy class 
of isomorphisms~\eref{ConjCorient_e} even if~$D_L$ and~$D_{\si^*\ov{L}}$
are not related by~\eref{DovLdfn_e}.
By the stabilization to the surjective case as in the proof of Lemma~\ref{CvsCanorient_lmm} below, 
this homotopy class respects the complex orientations of the two determinant lines 
if and only~if 
\BE{ConjCorient_e2}(1\!-\!g)\rk_{\C}L\!+\!\deg L\in2\Z\,,\EE
as is immediate if $D_L$ is a ($\C$-linear) CR-operator;
the expression on the left-hand side above is the complex index
of a ($\C$-linear) CR-operator on~$L$ or~$\si^*\ov{L}$.\\

\noindent
Furthermore, the projections
\BE{LovLtoLdfn_e}\begin{aligned}
\Ga\big(\Si;L\!\oplus\!\si^*\ov{L}\big)^{\si_L^{\oplus}}&\lra\Ga(\Si;L),
&\quad(\xi,\xi')&\lra\xi,\\
\Ga_{\fJ}^{0,1}\big(\Si;L\!\oplus\!\si^*\ov{L}\big)^{\si_L^{\oplus}}
&\lra \Ga_{\fJ}^{0,1}(\Si;L), &\quad (\ze,\ze')&\lra\ze,
\end{aligned}\EE
are isomorphisms and determine an isomorphism
\BE{COrientIsom_e} \det D_{(L\oplus\si^*\ov{L},\si_L^{\oplus})}\approx\det D_L\EE
which depends continuously on~$D_L$.
There is a canonical homotopy class of isomorphisms~\eref{COrientIsom_e} even 
if~$D_L$ and~$D_{(L\oplus\si^*\ov{L},\si_L^{\oplus})}$
are not related by~\eref{DLovLdfn_e}.
Via this homotopy class, the complex orientation on $\det D_L$ induces 
an orientation on $\det D_{(L\oplus\si^*\ov{L},\si_L^{\oplus})}$,
which we call the \sf{projection orientation} of the latter;
it varies continuously with~$D_{(L\oplus\si^*\ov{L},\si_L^{\oplus})}$.\\

\noindent
For any real bundle pair~$(L,\wt\si)$ over~$(\Si,\si,\fJ)$,
\eref{Runionorient_e0} implies that the real~line
$$\det D_{2(L,\wt\si)}\approx
\big(\!\det D_{(L,\wt\si)}\big)\!\otimes\!\big(\!\det D_{(L,\wt\si)}\big)$$
has a \sf{canonical orientation};
it is obtained by taking the same orientation on each of the two factors
on the right-hand side above.

\begin{lmm}\label{CvsCanorient_lmm}
Let $(L,\wt\si)$ be a real bundle pair over a symmetric Riemann surface~$(\Si,\si,\fJ)$
of arithmetic genus~$g$.
The homotopy class of isomorphisms
$$\det D_{2(L,\wt\si)}\approx \det D_{(L\oplus\si^*\ov{L},\si_L^{\oplus})}$$
induced by the isomorphism~$\Phi_L$ of real bundle pairs  as in~\eref{GIgen_e}
respects the canonical orientation on the left-hand side and
projection orientation on the right-hand side if and only~if 
$$ \frac{\big(\!(1\!-\!g)\rk_{\C}L\!+\!\deg L\big)\!\big(\!(1\!-\!g)\rk_{\C}L\!+\!\deg L\!-\!1\big)}2
 \in2\Z.$$
\end{lmm}

\begin{proof}
By the continuity of the canonical and projection orientations,
it is sufficient to establish the claim for any pair of real CR-operators on~$2(L,\wt\si)$ 
and~$(L\oplus\si^*\ov{L},\si_L^{\oplus})$.
We assume that~$D_{2(L,\wt\si)}$ is two copies of
the restriction~$D_{(L,\wt\si)}$ of a CR-operator~$D_L$ on~$L$
to the subspaces of the domain and target invariant under the conjugation induced by~$\wt\si$
and~$D_{(L\oplus\si^*\ov{L},\si_L^{\oplus})}$ is given by~\eref{DLovLdfn_e}.\\

\noindent
Suppose first that the CR-operator~$D_L$ on~$L$ is surjective and
thus
$$k\equiv\dim\ker D_{(L,\wt\si)}
=\dim_{\C}\ker D_L=\ind_{\C}D_L=(1\!-\!g)\rk_{\C}L\!+\!\deg L\,.$$
If $\xi_1,\ldots,\xi_k$ is a basis for $\ker D_{(L,\wt\si)}$,
then 
$$(\xi_1,0),\ldots,(\xi_k,0),(0,\xi_1),\ldots,(0,\xi_k)
\quad\hbox{and}\quad
\Phi_L(\xi_1,0),\Phi_L(0,\xi_1),\ldots,\Phi_L(\xi_k,0),\Phi_L(0,\xi_k)$$
are a basis for $\ker D_{2(L,\wt\si)}$ determining the canonical orientation
and a basis for $\ker D_{(L\oplus\si^*\ov{L},\si_L^{\oplus})}$ determining 
the projection orientation.
This establishes the claim if~$D_L$ is surjective.\\

\noindent
If $(L,\wt\si)$ is the direct sum of real bundle pairs $(L_1,\wt\si_1)$ and $(L_2,\wt\si_2)$
over~$(\Si,\si)$ and \hbox{$D_L\!=\!D_{L_1}\!\oplus\!D_{L_2}$},
the isomorphisms $\Phi_L,\Phi_{L_1},\Phi_{L_2}$ as in~\eref{GIgen_e} induce an isomorphism
\BE{CvsCanorient_e5}\begin{split}
\xymatrix{0\ar[r]& D_{2(L_1,\wt\si_1)} \ar[r]\ar[d]^{\approx}& 
D_{2(L,\wt\si)}\ar[r]\ar[d]^{\approx}& D_{2(L_2,\wt\si_2)}\ar[r]\ar[d]^{\approx}& 0\\
0\ar[r]& D_{(L_1\oplus\si^*\ov{L_1},\si_{L_1}^{\oplus})}  \ar[r]& 
D_{(L\oplus\si^*\ov{L},\si_L^{\oplus})}\ar[r]& 
D_{(L_2\oplus\si^*\ov{L_2},\si_{L_2}^{\oplus})} \ar[r]& 0}
\end{split}\EE
of exact triples of real CR-operators.
Since direct sums of complex vector bundles commute with the projections as in~\eref{LovLtoLdfn_e}
and the corresponding isomorphisms as in~\eref{Cunionorient_e0} respect the complex orientations,
the isomorphism
\BE{unionorient_e0a}
\det D_{(L\oplus\si^*\ov{L},\si_L^{\oplus})}\approx
\big(\!\det D_{(L_1\oplus\si^*\ov{L_1},\si_{L_1}^{\oplus})}\big)\!\otimes\!
\big(\!\det D_{(L_2\oplus\si^*\ov{L_2},\si_{L_2}^{\oplus})}\big)\EE
induced by the bottom row in~\eref{CvsCanorient_e5} respects the projection orientations.
By the {\it Exact Squares} property of determinant line bundles as in Section~2.2 in 
the 5th arXiv version of~\cite{detLB},
the isomorphism
\BE{unionorient_e0b}
\det D_{2(L,\wt\si)}\approx
\big(\!\det D_{2(L_1,\wt\si_1)}\big)\!\otimes\!\big(\!\det D_{2(L_2,\wt\si_2)}\big)\EE
induced by the top row in~\eref{CvsCanorient_e5} respects the canonical orientations
if and only~if
$$\big(\ind\,D_{(L_1,\wt\si_1)}\big)\big(\ind\,D_{(L_2,\wt\si_2)}\big)\in2\Z.$$
Since $\ind\,D_{(L,\wt\si)}\!=\!\ind\,D_{(L_1,\wt\si_1)}\!+\!\ind\,D_{(L_2,\wt\si_2)}$,
the claim holds for $(L_1,\wt\si_1)$ if it holds for $(L,\wt\si)$ and~$(L_2,\wt\si_2)$.\\

\noindent
For any real bundle pair $(L_1,\wt\si_1)$ on $(\Si,\si)$, 
we can find a real bundle pair~$(L_2,\wt\si_2)$ on $(\Si,\si)$
so that the complex vector bundles~$L_2$ and $L\!\equiv\!L_1\!\oplus\!L_2$
admit surjective CR-operators.
The claim of this lemma then holds for the real bundle pairs~$(L_2,\wt\si_2)$
and $(L,\wt\si_1\!\oplus\!\wt\si_2)$ and thus for~$(L_1,\wt\si_1)$.
\end{proof}

\begin{lmm}\label{unionorient_lmm}
Suppose $(\Si_1,\si_1,\fJ_1)$ and $(\Si_2,\si_2,\fJ_2)$ are symmetric Riemann surfaces
of arithmetic genera~$g_1$ and~$g_2$, respectively,
$(\Si,\si,\fJ)$ is the disjoint union of~$(\Si_1,\si_1,\fJ_1)$ and~$(\Si_2,\si_2,\fJ_2)$,
and $L$ is a complex vector bundle over~$\Si$.
If $L_i\!=\!L|_{\Si_i}$ for $i\!=\!1,2$, then 
the associated homotopy class of isomorphisms~\eref{unionorient_e0a}
respects the projection orientations on the three determinant lines.
If in addition $\wt\si$ is a conjugation on~$L$ and \hbox{$\wt\si_i\!=\!\wt\si|_{L_i}$} 
for $i\!=\!1,2$, then the associated homotopy class of isomorphisms~\eref{unionorient_e0b}
respects the canonical orientations on the three determinant lines if and only~if
$$\big(\!(1\!-\!g_1)\rk_{\C}L\!+\!\deg L_1\big)\!
\big(\!(1\!-\!g_2)\rk_{\C}L\!+\!\deg L_2\big) \in2\Z.$$
\end{lmm}

\begin{proof}
The first claim follows from the homotopy class of isomorphisms~\eref{Cunionorient_e0}
with $V\!=\!L$ respecting the complex orientation.
The second claim follows from the {\it Exact Squares} property of determinant line bundles
as in Section~2.2  in the 5th arXiv version of~\cite{detLB}, since
$$\ind\, D_{(V_i,\wt\si_i)}=(1\!-\!g_i)\rk_{\C}L\!+\!\deg L_i$$
for $i\!=\!1,2$.
This claim also follows from the first claim and Lemma~\ref{CvsCanorient_lmm}.
\end{proof}

\noindent
Suppose $\si$ is an anti-holomorphic diffeomorphism between 
closed Riemann surfaces~$(\Si_1,\fj_1)$ and $(\Si_2,\fJ_2)$ of arithmetic genus~$g$ each, so~that 
\BE{Dbldfn_e}(\Si,\si,\fj)\equiv\big(\Si_1\!\sqcup\!\Si_2,\si,\fj_1\!\sqcup\!\fj_2\big)\EE
is a symmetric Riemann surface of arithmetic genus~$2g\!-\!1$,
and $(V,\vph)$ is a real bundle pair over~$(\Si,\si)$.
Let $V_i\!=\!V|_{\Si_i}$ for $i\!=\!1,2$.
The restrictions
\BE{Dblorient_e0}\Ga(\Si;V\big)^{\vph}\lra\Ga(\Si_1;V_1),~~\xi\lra\xi|_{\Si_1}, \quad
\Ga_{\fJ}^{0,1}(\Si;V)^{\vph}\lra\Ga_{\fJ}^{0,1}(\Si_1;V_1),~~\ze\lra\ze|_{\Si_1},\EE
then induce isomorphisms between real CR-operators on~$(V,\vph)$
and generalized CR-operators on~$V_1$ and thus an isomorphism
\BE{Dblorient_e} \det D_{(V,\vph)}\approx\det\!\big(D_{(V,\vph)}|_{\Si_1}\big)\EE
which depends continuously on~$D_{(V,\vph)}$.
The $\rk_{\C}L\!=\!1$ case of the first claim below is \cite[Lemma~3.2]{GI}.

\begin{lmm}\label{Dblorient_lmm}
Suppose $(\Si_1,\fj_1)$, $(\Si_2,\fJ_2)$, $g$, and $(\Si,\si,\fj)$ are as above 
and $L$ is a complex line bundle over~$\Si$.
Let $\si_i\!=\!\si|_{\Si_i}$ and $L_i\!=\!L|_{\Si_i}$ for $i\!=\!1,2$.
The homotopy class of isomorphisms
\BE{Dblorient_e1b}
\det D_{(L\oplus\si^*\ov{L},\si_L^{\oplus})}\approx
\det D_{L_1\oplus\si_1^*\ov{L_2}}\EE
as in~\eref{Dblorient_e}
respects the projection orientation on the left-hand side and 
the complex orientation on the right-hand side if and only~if
$$(1\!-\!g)\rk_{\C}L\!+\!\deg L_2\in2\Z\,.$$
If in addition $\wt\si$ is a conjugation on~$L$ lifting~$\si$, 
then the homotopy class of isomorphisms
\BE{Dblorient_e2b}
\det D_{2(L,\wt\si)}\approx \det D_{2L_1}\EE
as in~\eref{Dblorient_e} respects the canonical orientation on the left-hand side and 
the complex orientation on the right-hand side.
\end{lmm}

\begin{proof}
Since the restrictions to $\Si_1$ as in~\eref{Dblorient_e0} commute with direct sums, 
the second claim follows immediately from the definition of the canonical orientation
on the left-hand side of~\eref{Dblorient_e2b}.
The first claim follows from the sentence containing~\eref{ConjCorient_e2} with~$L$ replaced by~$L_2$, 
since \hbox{$D_L\!=\!D_{L_1}\!\oplus\!D_{L_2}$} for some generalized CR-operators
$D_{L_1},D_{L_2}$ on $L_1,L_2$, respectively.
If $L$ admits a conjugation~$\wt\si$ lifting~$\si$, 
then this claim also follows from the second claim and Lemma~\ref{CvsCanorient_lmm}.
\end{proof}

\noindent
Suppose $(\Si,\si,\fj)$ is a symmetric Riemann surface with a conjugate pair 
of nodes~$\{x^+,x^-\}$.
Let $(\wh\Si,\wh\si,\wh\fj)$ be the symmetric Riemann surface obtained from~$(\Si,\si,\fj)$
by replacing $x^+$ and~$x^-$ by 2~smooth points each.
We denote by \hbox{$q\!:\wh\Si\!\lra\!\Si$}
the holomorphic projection and by $\wh{x}_1^{\pm},\wh{x}_2^{\pm}\!\in\!\wh\Si$ 
the preimages of~$x^{\pm}$ so that $\wh\si(\wh{x}_i^{\pm})\!=\!\wh{x}_i^{\mp}$
for $i\!=\!1,2$.
We will then call the tuple
\BE{Sinormdfn_e} 
\big(q\!:\wh\Si\!\lra\!\Si,\wh\si,\wh\fj,(\wh{x}_1^+,\wh{x}_1^-),(\wh{x}_2^+,\wh{x}_2^-)\!\big)\EE
\sf{the normalization} of~$(\Si,\si,\fj)$ at~$\{x^+,x^-\}$.\\

\noindent
Let $(V,\vph)$ be a real bundle pair over~$(\Si,\si)$. Define
\begin{gather*}
\ev_x^+\!:\Ga(\wh\Si,q^*V)^{q^*\vph}\lra V_{x^+},~~
\ev_x^+(\xi)=\xi(\wh{x}_1^+)\!-\!\xi(\wh{x}_2^+), \\
\ev_x^+\!:\Ga^{0,1}_{\wh{\fJ}}(\wh\Si,q^*V)^{q^*\vph}\lra\{0\},\qquad
\0_{V_{x^+}}\!:V_{x^+}\lra\{0\}\,.
\end{gather*}
A real CR-operator $D_{(V,\vph)}$ on~$(V,\vph)$ lifts to 
a real CR-operator $q^*D_{(V,\vph)}$ on~$q^*(V,\vph)$ so that 
the triple
$$0\lra D_{(V,\vph)}\lra q^*D_{(V,\vph)} \stackrel{\ev_x^+}{\lra}\0_{V_{x^+}}\lra0$$
of operators is exact and thus induces an isomorphism
\BE{Rnodisom_e} \det q^*D_{(V,\vph)} \approx 
\big(\!\det D_{(V,\vph)}\big)\!\otimes\!\La_{\R}^{\top}\big(V_{x^+}\big)\EE
which depends continuously on~$D_{(V,\vph)}$;
see the {\it Exact Triples} property in Section~2.2 in the 5th arXiv version of~\cite{detLB}. 
The $\rk_{\C}L\!=\!1$ case of the first claim below is \cite[Lemma~7.5]{GI2}.

\begin{lmm}\label{Rnodisom_lmm}
Suppose $L$ is a complex vector bundle over a symmetric Riemann surface~$(\Si,\si,\fJ)$
with a conjugate pair of nodes~$\{x^+,x^-\}$ and~\eref{Sinormdfn_e}
is the normalization of~$(\Si,\si,\fj)$ at~$\{x^+,x^-\}$.
If $D_{(L\oplus\si^*\ov{L},\si_L^{\oplus})}$ is a real CR-operator on
the real bundle pair $(L\!\oplus\!\si^*\ov{L},\si_L^{\oplus})$ over~$(\Si,\si)$,
then the associated isomorphism~\eref{Rnodisom_e} respects 
the projection orientations
on the two determinant lines and the complex orientation on the last factor
 if and only if \hbox{$\rk_{\C}\!L\!\in\!2\Z$}.
If in addition $\wt\si$ is a conjugation on~$L$ lifting~$\si$ and 
$D_{2(L,\wt\si)}$ is a real CR-operator on the real bundle pair $2(L,\wt\si)$ over~$(\Si,\si)$,
then the associated isomorphism~\eref{Rnodisom_e} respects the canonical orientations
on the two determinant lines and the complex orientation on the last factor.
\end{lmm}

\begin{proof}
By the continuous dependence of~\eref{Rnodisom_e} on~$D_{(V,\vph)}$, 
it is sufficient to establish each claim for 
a single real CR-operator on the relevant real bundle pair. 
Since the real dimension of $L_{x^+}$ is even,
the second claim for \hbox{$D_{2(L,\wt\si)}\!=\!D_{(L,\wt\si)}\!\oplus\!D_{(L,\wt\si)}$}
follows from the {\it Exact Squares} property applied 
to the commutative square of operators in Figure~\ref{Rnodisom_fig}.\\

\begin{figure}
$$\xymatrix{&0\ar[d]&0\ar[d]&0\ar[d]&\\
0\ar[r]& D_{(L,\wt\si)}\ar[r]\ar[d]&
q^*D_{(L,\wt\si)}\ar[r]\ar[d]& \0_{L_{x^+}}\ar[r]\ar[d]& 0\\
0\ar[r]& D_{2(L,\wt\si)}\ar[r]\ar[d]&
q^*D_{(2L,\wt\si)}\ar[r]\ar[d]& \0_{2L_{x^+}}\ar[r]\ar[d]& 0\\
0\ar[r]& D_{(L,\wt\si)}\ar[r]\ar[d]&
q^*D_{(L,\wt\si)}\ar[r]\ar[d]& \0_{L_{x^+}}\ar[r]\ar[d]& 0\\
&0&0&0&}$$
\caption{Exact square of operators for the proof of Lemma~\ref{Rnodisom_lmm}} 
\label{Rnodisom_fig}
\end{figure}

\noindent
Let $D_L$ be a CR-operator on the complex vector bundle~$L$
and~$D_{(L\oplus\si^*\ov{L},\si_L^{\oplus})}$ be as in~\eref{DLovLdfn_e}.
Define
\begin{gather*}
\ev_x\!:\Ga(\wh\Si,q^*L)\lra L_{x^+}\!\oplus\!L_{x^-},~~
\ev_x(\xi)=\big(\xi(\wh{x}_1^+)\!-\!\xi(\wh{x}_2^+),\xi(\wh{x}_1^-)\!-\!\xi(\wh{x}_2^-)\!\big),\\
\ev_x\!:\Ga^{0,1}_{\wh{\fJ}}(\wh\Si,q^*L)\lra\{0\}.
\end{gather*}
The CR-operator $D_L$ lifts to a CR-operator $q^*D_L$ on~$q^*L$ so that the diagram
\BE{Rnodisom_e5}\begin{split}
\xymatrix{0\ar[r]& D_{(L\oplus\si^*\ov{L},\si_L^{\oplus})}\ar[r]\ar[d]^{\approx}& 
q^*D_{(L\oplus\si^*\ov{L},\si_L^{\oplus})}\ar[r]^{\ev_x^+}\ar[d]^{\approx}&
\0_{(L\oplus\si^*\ov{L})_{x^+}}\ar[r]\ar[d]^{=}& 0\\
0\ar[r]& D_L\ar[r]& q^*D_L\ar[r]^{\ev_x}&\0_{L_{x^+}\oplus L_{x^-}}\ar[r]&0}
\end{split}\EE
is an isomorphism of exact triples of Fredholm operators.
The bottom row respects the complex orientations of the three determinant lines,
while the left and middle vertical isomorphisms intertwine the projection orientations
on the domains with the complex orientations on the targets.
The right vertical isomorphism is orientation-preserving with respect to the complex orientations
on its domain and target if and only if the complex dimension of~$L_{x^-}$ is even.
This establishes the first claim.
\end{proof}

\noindent
Suppose now that $(\Si,\si,\fj)$ is a symmetric Riemann surface with an isolated real node~$x$,
called \sf{$E$-node} in~\cite{RealGWsI} and elsewhere.
Let $(\wh\Si,\wh\si,\wh\fj)$ be the symmetric Riemann surface obtained from~$(\Si,\si,\fj)$
by replacing~$x$ by 2~smooth points.
We denote by \hbox{$q\!:\wh\Si\!\lra\!\Si$}
the holomorphic projection and by $\wh{x}^+,\wh{x}^-\!\in\!\wh\Si$ 
the preimages of~$x$; they are interchanged by the involution~$\wh\si$.
We will then call the tuple
\BE{SinormdfnE_e} 
\big(q\!:\wh\Si\!\lra\!\Si,\wh\si,\wh\fj,(\wh{x}^+,\wh{x}^-)\big)\EE
\sf{the normalization} of~$(\Si,\si,\fj)$ at~$x$.\\

\noindent
Let $(V,\vph)$ be a real bundle pair over~$(\Si,\si)$. 
Since $x\!\in\!\Si^{\si}$, $V_x^{\vph}\!\subset\!V_x$ is a half-dimensional totally real subspace.
An orientation on~$V_x^{\vph}$ and the complex orientation on~$V_x$ induce
an orientation on~~$V_x/V_x^{\vph}$ via the exact sequence
$$0\lra V_x^{\vph}\lra V_x\lra V_x/V_x^{\vph}\lra0.$$
If $(V,\vph)\!=\!(L\!\oplus\!\si^*\ov{L},\si_L^{\oplus})$ for a complex vector bundle $L\!\lra\!X$, 
we call the orientation on~$V_x/V_x^{\vph}$ induced by the associated projection orientation
of~$V_x^{\vph}$  the \sf{projection orientation}.
If \hbox{$(V,\vph)\!=\!2(L,\wt\si)$} for some real bundle pair~$(L,\wt\si)$
over~$(\Si,\si)$, then \hbox{$V_x^{\vph}\!=\!2L_x^{\wt\si}$}.
We call the orientation on~$V_x/V_x^{\vph}$ induced by the associated canonical orientation
of~$2L_x^{\wt\si}$ the \sf{canonical orientation}.
The isomorphism~\eref{GIgen_e2} with $\phi\!=\!\si$
intertwines the canonical orientation on its domain
and the projection orientation on its target if and only if $\binom{\rk_{\C}\!L-1}2$ is even.
The same applies to the isomorphism induced by~\eref{GIgen_e}
on the corresponding quotients~$V_x/V_x^{\vph}$.
Thus, Lemma~\ref{RnodisomE_lmm} below is consistent with Lemma~\ref{CvsCanorient_lmm}.\\

\noindent
Define
\begin{gather*}
\ev_x\!:\Ga(\wh\Si,q^*V)^{q^*\vph}\lra V_x/V_x^{\vph},~~
\ev_x(\xi)=\xi(\wh{x}^+)\!+\!V_x^{\vph}, \\
\ev_x\!:\Ga^{0,1}_{\wh{\fJ}}(\wh\Si,q^*V)^{q^*\vph}\lra\{0\},\qquad
\0_{V_x/V_x^{\vph}}\!:V_x/V_x^{\vph}\lra\{0\}\,.
\end{gather*}
A real CR-operator $D_{(V,\vph)}$ on~$(V,\vph)$ lifts to 
a real CR-operator $q^*D_{(V,\vph)}$ on~$q^*(V,\vph)$ so that 
the triple
$$0\lra D_{(V,\vph)}\lra q^*D_{(V,\vph)} \stackrel{\ev_x}{\lra}\0_{V_x/V_x^{\vph}}\lra0$$
of operators is exact and thus induces an isomorphism
\BE{RnodisomE_e} \det q^*D_{(V,\vph)} \approx 
\big(\!\det D_{(V,\vph)}\big)\!\otimes\!\La_{\R}^{\top}\big(V_x/V_x^{\vph}\big)\EE
which depends continuously on~$D_{(V,\vph)}$.

\begin{lmm}\label{RnodisomE_lmm}
Suppose $L$ is a complex vector bundle over a symmetric Riemann surface~$(\Si,\si,\fJ)$
with an $E$-node~$x$ and arithmetic genus~$g$.
Let~\eref{SinormdfnE_e} be the normalization of~$(\Si,\si,\fj)$ at~$x$.
If $D_{(L\oplus\si^*\ov{L},\si_L^{\oplus})}$ is a real CR-operator on
the real bundle pair $(L\!\oplus\!\si^*\ov{L},\si_L^{\oplus})$ over~$(\Si,\si)$,
then the associated isomorphism~\eref{RnodisomE_e} respects the projection orientations
on the three lines if and only if \hbox{$\rk_{\C}\!L\!\in\!2\Z$}.
If in addition $\wt\si$ is a conjugation on~$L$ lifting~$\si$ and 
$D_{2(L,\wt\si)}$ is a real CR-operator on the real bundle pair $2(L,\wt\si)$ over~$(\Si,\si)$,
then the associated isomorphism~\eref{RnodisomE_e} respects the canonical orientations
on the three lines if and only~if $\rk_{\C}\!L(g\!+\!\deg L)$ is even.
\end{lmm}

\begin{proof} 
For the purposes of the second claim, we replace~$L_{x^+}$ by~$L_x/L_x^{\wt\si}$
in the exact square of Fredholm operators in Figure~\ref{Rnodisom_fig}.
The canonical orientation on~$(2L)/(2L)^{2\wt\si}$ as defined above is
the same as the canonical orientation with respect to the isomorphism 
\hbox{$(2L)/(2L)^{2\wt\si}\!\approx\!2(L/L)^{\wt\si}$} if and only if~$\rk_{\C}\!L$ is~even.
By the Exact Squares property, the middle row in the diagram in Figure~\ref{Rnodisom_fig}
respects the canonical orientations on the two determinant lines and 
the latter orientation on~$(2L)/(2L)^{2\wt\si}$ if and only if 
$$\big(\ind\,\0_{L_x/L_x^{\vph}}\big)\big(\ind,D_{(L,\wt\si)}\big)
=\rk_{\C}\!L\big((1\!-\!g)\rk_{\C}\!L\!+\!\deg L\big)$$
is even. Combining the two statements, we obtain the second claim.\\

\noindent
For the purposes of the first claim, we define
$$\ev_x\!:\Ga(\wh\Si,q^*L)\lra L_x,~~ \ev_x(\xi)=\xi(\wh{x}^+)\!-\!\xi(\wh{x}^-)
\quad \ev_x\!:\Ga^{0,1}_{\wh{\fJ}}(\wh\Si,q^*L)\lra\{0\}.$$
With the assumptions as in~\eref{Rnodisom_e5}, the diagram
\begin{equation*}\begin{split}
\xymatrix{0\ar[r]& D_{(L\oplus\si^*\ov{L},\si_L^{\oplus})}\ar[r]\ar[d]^{\approx}& 
q^*D_{(L\oplus\si^*\ov{L},\si_L^{\oplus})}\ar[r]^<<<<{\ev_x}\ar[d]^{\approx}&
\0_{(L\oplus\si^*\ov{L})_x/(L\oplus\si^*\ov{L})_x^{\si_L^{\oplus}}}\ar[r]\ar[d]^{\approx}& 0\\
0\ar[r]& D_L\ar[r]& q^*D_L\ar[r]^{\ev_x}&\0_{L_x}\ar[r]&0}
\end{split}\end{equation*}
is an isomorphism of exact triples of Fredholm operators.
The bottom row respects the complex orientations of the three determinant lines, as before,
while the left and middle vertical isomorphisms intertwine the projection orientations
on the domains with the complex orientations on the targets.
The right vertical isomorphism, which is given by taking the difference of the two components 
of $L_x\!\oplus\!L_x$, is orientation-preserving with respect to the projection orientation
on its domain and the complex orientation on its target
if and only if the complex dimension of~$L_x$ is even.
This establishes the first claim.
\end{proof}

\subsection{Induced orientations}
\label{InducedOrien_subs}

\noindent
Let $(\Si,\si)$ be a symmetric Riemann surface and
$D_{(V,\vph)}$ be a real CR-operator on a rank~$k$ real bundle pair~$(V,\vph)$ over~$(\Si,\si)$.
A real orientation~$(L,[\psi],\fs)$ for $(V,\vph)$ determines
a homotopy class of isomorphisms
\BE{ROrientCompl_e2}\begin{split}
 \big(\!\det D_{(V,\vph)}\big)\!\otimes\!\big(\!\det D_{(L^*\oplus\si^*\ov{L}^*,\si_{L^*}^{\oplus})}\big)
&\approx \det D_{(V\oplus L^*\oplus\si^*\ov{L}^*,\vph\oplus\si_{L^*}^{\oplus})}\\
&\approx \det\!\big(\!(k\!+\!2)\dbar_{(\Si,\si)}\big)
\approx\det\!\big(k\dbar_{(\Si,\si)}\big)\!\otimes\!\det\!\big(2\dbar_{(\Si,\si)}\big)
\end{split}\EE
via isomorphisms for direct sums of real bundle pairs as in~\eref{Runionorient_e0}
and the homotopy class of isomorphisms~\eref{RBPisom_e2}.
Combined with the projection orientation on the second determinant line above
and the canonical orientation on the last determinant line,
\eref{ROrientCompl_e2} determines a homotopy class of isomorphisms
\BE{reldetdfn_e0}\det D_{(V,\vph)}
\approx\det\!\big(k\dbar_{(\Si,\si)}\big)
\approx \big(\!\det\dbar_{(\Si,\si)}\big)^{\otimes k}\EE
of real lines.
If in addition $\wt\si$ is a conjugation on~$L$ lifting~$\si$, isomorphisms as in~\eref{Runionorient_e0} 
and the homotopy class of isomorphisms~\eref{RBPisom_e3}
determine a homotopy class of isomorphisms
\BE{ROrientCompl_e3}\begin{split}
\big(\!\det D_{(V,\vph)}\big)\!\otimes\!\big(\!\det\,D_{2(L^*,\wt\si^*)}\big)
&\approx \det D_{(V\oplus 2L^*,\vph\oplus2\wt\si^*)} \\
&\approx\det\!\big(\!(k\!+\!2)\dbar_{(\Si,\si)}\big)
\approx\det\!\big(k\dbar_{(\Si,\si)}\big)\!\otimes\!\det\!\big(2\dbar_{(\Si,\si)}\big)\,.
\end{split}\EE
Along with the canonical orientations on the second and last determinant lines above,
\eref{ROrientCompl_e3} determines another homotopy class of isomorphisms~\eref{reldetdfn_e0}.\\

\noindent
The next four corollaries follow readily from the corresponding lemmas.
The $\rk_{\C}V\!=\!1$ case of the first claim of Corollary~\ref{Rnodisom_crl} is \cite[Lemma~7.5]{GI2}.
The second claims of Corollaries~\ref{Dblorient_crl} and~\ref{Rnodisom_crl} 
are Lemma~3.1 and Corollary~4.5, respectively, in~\cite{RealGWsII}.

\begin{crl}[of Lemma~\ref{unionorient_lmm}]\label{unionorient_crl}
Suppose $(\Si_1,\si_1,\fJ_1)$, $(\Si_2,\si_2,\fJ_2)$, $(\Si,\si,\fJ)$,
and $g_1,g_2$ are as in Lemma~\ref{unionorient_lmm},
$(V,\vph)$ is a rank~$k$ real bundle pair over~$(\Si,\si)$,
and $(L,[\psi],\fs)$ is a real orientation for~$(V,\vph)$.
If 
$$(V_i,\vph_i)=(V,\vph)\big|_{\Si_i} \quad\hbox{and}\quad 
\big(L_i,[\psi_i],\fs_i\big)=\big(L,[\psi],\fs\big)\!\big|_{\Si_i}$$
for $i\!=\!1,2$,
then the diagram 
\BE{unionorient_e0}\begin{split}
\xymatrix{\det D_{(V,\vph)} \ar[rr]^{\eref{ROrientCompl_e2}}_{\approx}
\ar[d]^{\approx}_{\eref{Runionorient_e0}}&& 
\det\!\big(k\dbar_{(\Si,\si)}\big)\ar[d]_{\approx}^{\eref{Runionorient_e0}}\\
\big(\!\det D_{(V_1,\vph_1)}\big)\!\otimes\!\big(\!\det D_{(V_2,\vph_2)}\big)
\ar[rr]^{\eref{ROrientCompl_e2}}_{\approx}&&
\det\!\big(k\dbar_{(\Si_1,\si_1)}\big)\!\otimes\!
\det\!\big(k\dbar_{(\Si_2,\si_2)}\big)\,,}
\end{split}\EE
with the homotopy classes of the horizontal isomorphisms 
determined by the real orientations $(L,[\psi],\fs)$ and $(L_i,[\psi_i],\fs_i)$,
commutes if and only if \hbox{$(g_1\!-\!1)(g_2\!-\!1)\!\in\!2\Z$}.
If in addition $\wt\si$ is a conjugation on~$L$ lifting~$\si$
and \hbox{$\wt\si_i\!=\!\wt\si|_{L_i}$} for $i\!=\!1,2$, 
then the diagram~\eref{unionorient_e0} with~\eref{ROrientCompl_e2} replaced by~\eref{ROrientCompl_e3}
commutes if and only~if 
$$(g_1\!-\!1)(g_2\!-\!1)\!+\!
\big(g_1\!-\!1\!+\!\deg L_1\big)\!\big(g_2\!-\!1\!+\!\deg L_2\big) \in2\Z.$$
\end{crl}

\begin{proof} 
Since the indices of $D_{(L_i^*\oplus\si_i^*\ov{L}_i^*,\si_{L_i^*}^{\oplus})}$,
$D_{2(L_i^*,\wt\si_i^*)}$, and $2\dbar_{(\Si_i,\si_i)}$ are even,
the outer isomorphisms in~\eref{ROrientCompl_e2} and~\eref{ROrientCompl_e3}
commute with the isomorphisms~\eref{Runionorient_e0} associated with 
the disjoint union of $(\Si_1,\si_1,\fJ_1)$ and~$(\Si_2,\si_2,\fJ_2)$.
Since the restrictions of the homotopy classes of 
isomorphisms~\eref{RBPisom_e2} and~\eref{RBPisom_e3} determined by~$(L,[\psi],\fs)$ to~$\Si_i$
are determined by~$(L_i,[\psi_i],\fs_i)$,
the homotopy classes of the middle isomorphisms in~\eref{ROrientCompl_e2} and~\eref{ROrientCompl_e3}
also commute with the isomorphisms~\eref{Runionorient_e0} associated with 
the disjoint union of $(\Si_1,\si_1,\fJ_1)$ and~$(\Si_2,\si_2,\fJ_2)$.
The first claim of the corollary thus follows from the first statement of Lemma~\ref{unionorient_lmm}
with~$L$ replaced by~$L^*$ and the second statement of this lemma with $L\!=\!\Si\!\times\!\C$.
The second claim of the corollary similarly follows from the second statement of 
Lemma~\ref{unionorient_lmm} with~$L$ replaced by~$L^*$ and with $L\!=\!\Si\!\times\!\C$.
\end{proof}

\begin{crl}[of Lemma~\ref{Dblorient_lmm}]\label{Dblorient_crl}
Suppose $(\Si_1,\fj_1)$, $(\Si_2,\fJ_2)$, $g$, $(\Si,\si,\fj)$, and~$\si_1$
are as in Lemma~\ref{Dblorient_lmm}.
Let  $(V,\vph)$ and $(L,[\psi],\fs)$ be as in Corollary~\ref{unionorient_crl}.
If $L_i\!=\!L|_{\Si_i}$ for $i\!=\!1,2$,
then the composition
\BE{Dblorient_e0a} 
\det D_{(V,\vph)}|_{\Si_1}\underset{\approx}{\xlla{\eref{Dblorient_e}}}\det D_{(V,\vph)}
\underset{\approx}{\xlra{\eref{ROrientCompl_e2}}} 
\det\big(k\dbar_{(\Si,\si)}\big)
\underset{\approx}{\xlra{\eref{Dblorient_e}}}
\det\big(k\dbar_{\Si_1}\big)\,,\EE
with the middle isomorphism determined by $(L,[\psi],\fs)$, 
intertwines the complex orientations on the two ends if and only if
\hbox{$g\!-\!1\!+\!\deg L_2\!\in\!2\Z$}.
If in addition $\wt\si$ is a conjugation on~$L$ lifting~$\si$, 
then the composition~\eref{Dblorient_e0a} with~\eref{ROrientCompl_e2}
replaced by~\eref{ROrientCompl_e3}
intertwines the complex orientations on the two~ends.
\end{crl}

\begin{proof}
Isomorphisms for direct sums of complex vector bundles as in~\eref{Cunionorient_e0}
and the restrictions of the isomorphisms~\eref{RBPisom_e2} to the bundles over~$\Si_1$
determine a homotopy class of isomorphisms
\BE{ROrientCompl_e2c}\begin{split}
 \big(\!\det D_{(V,\vph)}|_{\Si_1}\big)\!\otimes\!
\big(\!\det D_{(L^*\oplus\si^*\ov{L}^*,\si_{L^*}^{\oplus})}|_{\Si_1}\big)
&\approx \det D_{(V\oplus L^*\oplus\si^*\ov{L}^*,\vph\oplus\si_{L^*}^{\oplus})}|_{\Si_1}\\
&\approx \det\!\big(\!(k\!+\!2)\dbar_{\Si}\big)
\approx\det\!\big(k\dbar_{\Si}\big)\!\otimes\!\det\!\big(2\dbar_{\Si}\big).
\end{split}\EE
Since the restrictions of isomorphisms~\eref{RBPisom_e2} are complex trivializations 
$V\!\oplus\!L^*\!\oplus\!\si^*\ov{L}^*|_{\Si_1}$,
\eref{ROrientCompl_e2c} is orientation-preserving with respect to the complex orientations
on all determinant lines above.
The isomorphisms~\eref{Dblorient_e} intertwine the isomorphisms in~\eref{ROrientCompl_e2} 
and~\eref{ROrientCompl_e2c}. 
The first claim of the corollary thus follows from the first statement of Lemma~\ref{Dblorient_lmm}
with~$L$ replaced by~$L^*$ and the second statement of this lemma with $L\!=\!\Si\!\times\!\C$.
The second claim of the corollary similarly follows from the second statement of 
Lemma~\ref{Dblorient_lmm} with~$L$ replaced by~$L^*$ and with $L\!=\!\Si\!\times\!\C$.
\end{proof}

\begin{crl}[of Lemma~\ref{Rnodisom_lmm}]\label{Rnodisom_crl}
Suppose $(\Si,\si,\fJ)$, $\{x^+,x^-\}$, $(\wh\Si,\wh\si)$, and $q$ 
are as in Lemma~\ref{Rnodisom_lmm}.
If $(V,\vph)$ and $(L,[\psi],\fs)$ are as in Corollary~\ref{unionorient_crl},
the diagram
\BE{Rnodisom_e0a}\begin{split}
\xymatrix{\det q^*D_{(V,\vph)}\ar[rr]_{\approx}^{\eref{ROrientCompl_e2}}
\ar[d]_{\eref{Rnodisom_e}}^{\approx}&& 
\det\!\big(k\dbar_{(\wh\Si,\wh\si)}\big)\ar[d]^{\eref{Rnodisom_e}}_{\approx}\\
\det D_{(V,\vph)}\!\otimes\!\La_{\R}^{\top}\big(V_{x^+}\big)
\ar[rr]_{\approx}^{\eref{ROrientCompl_e2}}&&
\det\!\big(k\dbar_{(\Si,\si)}\big)\!\otimes\!\La_{\R}^{\top}\big(\C^k\big)\,,}  
\end{split}\EE
with the homotopy classes of the horizontal isomorphisms determined 
by the real orientations $q^*(L,[\psi],\fs)$ and~$(L,[\psi],\fs)$ 
and by the complex orientations of~$V_{x^+}$ and~$\C^k$,
does not commute up to homotopy.
If in addition $\wt\si$ is a conjugation on~$L$ lifting~$\si$,
then the diagram~\eref{Rnodisom_e0a} with~\eref{ROrientCompl_e2}
replaced by~\eref{ROrientCompl_e3} commutes up to homotopy.
\end{crl}

\begin{proof} 
Since the indices of $D_{(L^*\oplus\si^*\ov{L}^*,\si_{L^*}^{\oplus})}$,
$D_{2(L^*,\wt\si^*)}$, and $2\dbar_{(\Si,\si)}$ are even,
the outer isomorphisms in~\eref{ROrientCompl_e2} and~\eref{ROrientCompl_e3}
commute with the corresponding isomorphisms~\eref{Rnodisom_e}.
Since the pullbacks of the homotopy classes of 
isomorphisms~\eref{RBPisom_e2} and~\eref{RBPisom_e3} determined by~$(L,[\psi],\fs)$ to~$\wh\Si$
are determined by~$q^*(L,[\psi],\fs)$,
the homotopy classes of the middle isomorphisms in~\eref{ROrientCompl_e2} and~\eref{ROrientCompl_e3}
also commute with the corresponding isomorphisms~\eref{Rnodisom_e}.
The first claim of the corollary thus follows from the first statement of Lemma~\ref{Rnodisom_lmm}
with~$L$ replaced by~$L^*$ and the second statement of this lemma with $L\!=\!\Si\!\times\!\C$.
The second claim of the corollary similarly follows from the second statement of 
Lemma~\ref{Rnodisom_lmm} with~$L$ replaced by~$L^*$ and with $L\!=\!\Si\!\times\!\C$.
\end{proof}

\begin{crl}[of Lemma~\ref{RnodisomE_lmm}]\label{RnodisomE_crl}
Suppose $(\Si,\si,\fJ)$, $x$, $g$, $(\wh\Si,\wh\si)$, and $q$ are as in Lemma~\ref{RnodisomE_lmm}.
If $(V,\vph)$ and $(L,[\psi],\fs)$ are as in Corollary~\ref{unionorient_crl},
the diagram
\BE{RnodisomE_e0a}\begin{split}
\xymatrix{\det q^*D_{(V,\vph)}\ar[rr]_{\approx}^{\eref{ROrientCompl_e2}}
\ar[d]_{\eref{RnodisomE_e}}^{\approx}&& 
\det\!\big(k\dbar_{(\wh\Si,\wh\si)}\big)\ar[d]^{\eref{RnodisomE_e}}_{\approx}\\
\det D_{(V,\vph)}\!\otimes\!\La_{\R}^{\top}\big(V_x/V_x^{\vph}\big)
\ar[rr]_{\approx}^{\eref{ROrientCompl_e2}}&&
\det\!\big(k\dbar_{(\Si,\si)}\big)\!\otimes\!\La_{\R}^{\top}\big(\C^k/\R^k\big)\,,}  
\end{split}\EE
with the homotopy classes of the horizontal isomorphisms determined 
by the real orientations $q^*(L,[\psi],\fs)$ and~$(L,[\psi],\fs)$,
the induced orientation of~$V_x^{\vph}$, and
the standard orientation of~$\C^k/\R^k$,
commutes up to homotopy if and only if \hbox{$g\!-\!1\!\in\!2\Z$}.
If in addition $\wt\si$ is a conjugation on~$L$ lifting~$\si$,
then the diagram~\eref{RnodisomE_e0a} with~\eref{ROrientCompl_e2}
replaced by~\eref{ROrientCompl_e3} commutes up to homotopy
if and only if \hbox{$\deg L\!\in\!2\Z$}.
\end{crl}

\begin{proof}
The proof is identical to that of Corollary~\ref{Rnodisom_crl},
with Lemma~\ref{Rnodisom_lmm} replaced by~\ref{RnodisomE_lmm}.
\end{proof}

\noindent
By Definition~\ref{realorient_dfn0}\ref{spin_it0},
a real orientation~$(L,[\psi],\fs)$  on~$(V,\vph)$ determines a \sf{relative spin structure}
on the real bundle~$V^{\vph}$ over~$X^{\phi}\!\subset\!X$ in
the spirit of \cite[Definition~8.1.2]{FOOO},
which we call the \sf{the associated relative spin structure on~$V^{\vph}$}.
If in addition~$\wt\phi$ is a conjugation on~$L$ lifting~$\phi$
and $L^{\wt\phi}\!\lra\!X^{\phi}$ is orientable, 
then $2(L^*)^{\wt\phi^*}$ has a canonical homotopy
class of trivializations as in the proof of \cite[Corollary~5.6]{RealGWsI}.
Such a real orientation on~$(V,\vph)$ thus determines a spin structure on~$V^{\vph}$,
which we call the \sf{the associated spin structure on~$V^{\vph}$}.\\

\noindent
Let $D_{(V,\vph)}$ be a real CR-operator on a real bundle pair~$(V,\vph)$ over~$(\P^1,\tau)$.
By \cite[Lemma~6.37]{Melissa}, a (relative) spin structure on 
the real vector bundle~$V^{\vph}$ over the $\tau$-fixed locus $S^1\!\subset\!\P^1$
and a choice of a disk $\D\!\subset\!\P^1$ determine an orientation on~$\det D_{(V,\vph)}$.
If $(V,\vph)$ admits a real orientation, then $\deg V\!\in\!2\Z$ and 
this orientation does not depend on the choice of~$\D$.
Since the kernel of the surjective real CR-operator~$\dbar_{(\P^1,\tau)}$ consists 
of the constant $\R$-valued functions on~$\P^1$,
$\det\dbar_{(\P^1,\tau)}$ is canonically oriented.
A homotopy class of isomorphisms as in~\eref{reldetdfn_e0} thus also determines
an orientation on~$\det D_{(V,\vph)}$.
Any complex vector bundle~$L$ over~$(\P^1,\tau)$ admits a conjugation~$\wt\tau$
lifting~$\tau$.
The real vector bundle $L^{\wt\tau}$ is orientable if and only if $\deg L\!\in\!2\Z$.

\begin{crl}\label{RelSpinOrient_crl}
Suppose $(V,\vph)$ is a real bundle pair over~$(\P^1,\tau)$,
$(L,[\psi],\fs)$ is a real orientation for~$(V,\vph)$, and $D_{(V,\vph)}$ is a real CR-operator.
The orientations on~$\det D_{(V,\vph)}$ induced by the associated relative spin structure 
on~$V^{\vph}$ and by~$(L,[\psi],\fs)$ via~\eref{ROrientCompl_e2} 
are the same if and only if $\deg V\!\in\!4\Z$.
If in addition $\wt\tau$ is a conjugation on~$L$ lifting~$\tau$,
the former orientation and the orientation induced  by~$(L,[\psi],\fs)$ 
via~\eref{ROrientCompl_e3} with $\wt\si\!=\!\wt\tau$
are the same if and only if \hbox{$\deg V\!\cong\!0,-2$}
\!\!\!$\mod8$.
If $\deg V\!\in\!4\Z$, the orientation on~$\det D_{(V,\vph)}$ induced 
by the associated spin structure and the latter orientation are the same.
\end{crl}

\begin{proof} The third claim follows from \cite[Lemma~3.3]{RealGWsII} with replaced by~$L^*$.
The  \hbox{$\deg V\!\in\!4\Z$} and \hbox{$\deg V\!-\!2\!\in\!4\Z$} cases
of the second claim follow from
\cite[Lemma~3.4]{RealGWsII} and \cite[Corollary~3.6]{RealGWsII}, respectively,
with $L$ replaced by~$L^*$ and thus with $\deg L$ replaced by $-(\deg V)/2$;
they are established via technical computations.
The first claim follows from the second and Lemma~\ref{CvsCanorient_lmm}
with~$L$ replaced by~$L^*$.
\end{proof}

\subsection{Orientations of moduli spaces}
\label{MSorient_subs1}

\noindent
Let $g,\ell\!\in\!\Z$ with $g\!+\!\ell\!\ge\!2$.
Denote~by $\R\cM_{g,\ell}^{\bu}\!\subset\!\R\ov\cM_{g,\ell}^{\bu}$ the subspace of smooth curves.
The determinants of the operators~$\dbar_{(\Si,\si)}$
on symmetric Riemann surfaces form a real line orbi-bundle~$\det\dbar_{\C}$
over~$\R\ov\cM_{g,\ell}^{\bu}$.
By \cite[Propositions~5.9,6.1]{RealGWsI}, the real line orbi-bundle
\BE{RcMisom_e} \La^{\top}_{\R}\big(T(\R\ov\cM_{g,\ell}^{\bu})\!\big)\otimes
\big(\!\det\dbar_{\C}\big)\lra\R\ov\cM_{g,\ell}^{\bu}\EE
has a natural orientation.
The construction of this orientation over the equivalence class
\hbox{$[\cC]\!\in\!\R\cM_{g,\ell}^{\bu}$}
of a smooth marked curve~$\cC$ 
with the underlying symmetric Riemann surface~$(\Si,\si,\fJ)$ involves:
\begin{enumerate}[label=($\cM\arabic*$),leftmargin=*]

\item\label{LSDolb_it} the Kodaira-Spencer and Dolbeault isomorphisms that 
identify $T_{[\cC]}(\R\ov\cM_{g,\ell}^{\bu})$ with the cokernel $H^1(\Si;T\cC)^{\vph}$ of 
an injective real CR-operator on a real line bundle pair~$(T\cC,\vph)$ over~$(\Si,\si)$;

\item\label{SD_it} the Serre Duality that identifies $H^1(\Si;T\cC)^{\vph}$
with the dual $(H^0(\Si;T\cC^*\!\otimes\!T^*\Si)^{\vph*\otimes\nd\si^*})^*$
of the kernel
of a surjective real CR-operator $\dbar_{(T\cC^*\otimes T^*\Si,\vph*\otimes\nd\si^*)}$
on the real bundle pair \hbox{$(T\cC^*\!\otimes\!T^*\Si,\vph^*\!\otimes\!\nd\si^*)$}
over~$(\Si,\si)$;

\item\label{ses_it} an exact triple of Fredholm operators that identifies
$\det\dbar_{(T\cC^*\otimes T^*\Si,\vph*\otimes\nd\si^*)}$ with 
the tensor product of $\det\dbar_{(T^*\Si^{\otimes2},\nd\si^{*\otimes2})}$
and the top exterior power of a complex vector space and thus 
determines an orientation on the tensor product of the two determinant lines;

\item\label{ROrientStab_it} a homotopy class of isomorphisms 
$\det\dbar_{(T^*\Si^{\otimes2},\nd\si^{*\otimes2})}\!\approx\!\det\dbar_{(\Si,\si)}$
obtained via~\eref{ROrientCompl_e3} from 
a canonical real orientation on $(T^*\Si^{\otimes2},\nd\si^{*\otimes2})$
with $L\!=\!T^*\Si$ and  involution~$\nd\si^*$;

\item\label{ROrientTwist_it} twisting the resulting homotopy class of isomorphisms 
$\La^{\top}_{\R}(T_{[\cC]}(\R\ov\cM_{g,\ell}^{\bu})\!)\!\approx\!\det\dbar_{(\Si,\si)}$
by $(-1)^{(g-1)+|\si|}$, where $|\si|$ is the number of connected components of~$\Si^{\si}$.

\end{enumerate}
The parity of each of the two summands in the twisting exponent in~\ref{ROrientTwist_it}
is invariant under disjoint unions of symmetric surfaces and under 
the first node-identifying immersion in~\eref{Riogldfn_e2} below.
This exponent is even if the domain is a doublet as in~\eref{Dbldfn_e} or~$(\P^1,\tau)$.
Thus, the twisting in~\ref{ROrientTwist_it} can be ignored for the purposes
of Lemma~\ref{unionorientDM_lmm} and 
Propositions~\ref{unionorient_prp}-\ref{Rnodisom_prp}  and~\ref{RelSpinOrient_prp} below.\\

\noindent
Let $(X,\om,\phi)$ be a real symplectic manifold of real dimension~$2n$ and 
$J\!\in\!\cJ_{\om}^{\phi}$.
We denote~by 
$$\fM_{g,\ell}^{\phi;\star}(B;J)\subset\ov\fM_{g,\ell}^{\phi;\bu}(B;J)$$
the subspace of maps from smooth stable domains and by 
$$\ff\!:\fM_{g,\ell}^{\phi;\star}(B;J)\lra\R\cM_{g,\ell}^{\bu}$$
the forgetful morphism dropping the map component.
The determinants of the linearizations of the $\dbar_J$-operators on 
the space of real maps from symmetric surfaces to~$(X,\phi)$   
form a real line orbi-bundle $\det D_{(TX,\nd\phi)}$ 
over $\fM_{g,\ell}^{\phi;\star}(B;J)$; see \cite[Section~4.3]{RealGWsI}.
The forgetful morphism~$\ff$ induces an isomorphism
\BE{RfMisom_e}\La_{\R}^{\top}\big(T\fM_{g,\ell}^{\phi;\star}(B;J)\!\big)
\approx \big(\!\det D_{(TX,\nd\phi)}\big)\!\otimes\!
\ff^*\!\big(\La^{\top}_{\R}(T\R\cM_{g,\ell}^{\bu})\!\big)\EE
of real line orbi-bundles over $\fM_{g,\ell}^{\phi;\star}(B;J)$;
the line orbi-bundle on the left-hand side is the top exterior power of 
the virtual tangent bundle for the moduli space $\fM_{g,\ell}^{\phi;\star}(B;J)$.
A real orientation~$(L,[\psi],\fs)$ on~$(X,\om,\phi)$ determines a homotopy class of isomorphisms
\BE{ROisom_e}  \det D_{(TX,\nd\phi)}\approx  
\ff^*\!\big(\!(\det\dbar_{\C})^{\otimes n}\big)\EE
of real line orbi-bundles over $\fM_{g,\ell}^{\phi;\star}(B;J)$ via~\eref{ROrientCompl_e2},
with $(L,[\psi],\fs)$ replaced by $u^*(L,[\psi],\fs)$ for each stable map~$u$ 
representing an element of $\fM_{g,\ell}^{\phi;\star}(B;J)$.
If in addition $\wt\phi$ is a conjugation on~$L$ lifting~$\phi$,
a homotopy class of isomorphisms~\eref{ROisom_e} is similarly determined 
via~\eref{ROrientCompl_e3}.\\

\noindent
Combining \eref{RcMisom_e}-\eref{ROisom_e}, we obtain a homotopy class of isomorphisms
\BE{RMBvfc_e} 
\La_{\R}^{\top}\big(T\fM_{g,\ell}^{\phi;\star}(B;J)\!\big)
\approx \ff^*\!\big(\!(\det\dbar_{\C})^{\otimes(n+1)}\big)\EE
of real line orbi-bundles over $\fM_{g,\ell}^{\phi;\star}(B;J)$.
If $n\!\not\in\!2\Z$, we can break the right-hand side into $(n\!+\!1)/2$ 
products of $\ff^*(\!(\det\dbar_{\C})^{\otimes2})$, each of which is canonically oriented,
and thus obtain an orientation on~$\fM_{g,\ell}^{\phi;\star}(B;J)$.
Via the stabilization of the domain as in the proof of \cite[Corollary~5.10]{RealGWsI},
this orientation extends over the subspace
$$\fM_{g,\ell}^{\phi;\bu}(B;J)\subset\ov\fM_{g,\ell}^{\phi;\bu}(B;J)$$
of all maps from smooth domains.
By the proof of \cite[Theorem~1.3]{RealGWsI}, the extended orientation
further extends over the rest of~$\ov\fM_{g,\ell}^{\phi;\bu}(B;J)$.\\

\noindent
For $g_1,g_2\!\in\!\Z$ and $\ell_1,\ell_2\!\in\!\Z^{\ge0}$, let
$$\io_{\sqcup}\!:\R\ov\cM_{g_1,\ell_1}^{\bu}\!\times\!\R\ov\cM_{g_2,\ell_2}^{\bu}
\lra \R\ov\cM_{g_1+g_2-1,\ell_1+\ell_2}^{\bu}$$
be the open embedding obtained by taking the disjoint union of two marked curves
and re-ordering their conjugate pairs of marked points in some fixed way.
If $[\u_i]\!\in\!\ov\fM_{g_i,\ell_i}^{\phi;\bu}(B_i;J)$ for $i\!=\!1,2$, we similarly define
$$\big[\u_1\!\sqcup\!\u_2\big]\in \ov\fM_{g_1+g_2-1,\ell_1+\ell_2}^{\phi;\bu}(B_1\!+\!B_2;J)$$
to be the equivalence class of the stable map obtained by taking the disjoint union of the two maps
and ordering the conjugate pairs of marked points of the two maps in some fixed way.

\begin{lmm}\label{unionorientDM_lmm}
Suppose $g_1,g_2\!\in\!\Z$ and $\ell_1,\ell_2\!\in\!\Z^{\ge0}$.
Let
$$\pi_1,\pi_2\!:\R\ov\cM_{g_1,\ell_1}^{\bu}\!\times\!\R\ov\cM_{g_2,\ell_2}^{\bu}
\lra\R\ov\cM_{g_1,\ell_1}^{\bu},\R\ov\cM_{g_2,\ell_2}^{\bu}$$
be the component projections.
The diagram 
\BE{unionorientDM_e0}\begin{split}
\xymatrix{\io_{\sqcup}^*\La^{\top}_{\R}\big(T(\R\ov\cM_{g_1+g_2-1,\ell_1+\ell_2}^{\bu})\!\big)
\ar[rr]^<<<<<<<<<<<<<<<<<<<<{\eref{RcMisom_e}}_<<<<<<<<<<<<<<<<<<<<{\approx} 
\ar[d]^{\approx}_{\La^{\top}_{\R}(\nd\pi_1,\nd\pi_2)}&& 
\io_{\sqcup}^*\det\dbar_{\C}\ar[d]_{\approx}^{\eref{Runionorient_e0}}\\
\pi_1^*\La^{\top}_{\R}\big(T(\R\ov\cM_{g_1,\ell_1}^{\bu})\!\big)
\!\otimes\!\pi_2^*\La^{\top}_{\R}\big(T(\R\ov\cM_{g_2,\ell_2}^{\bu})\!\big)
\ar[rr]^<<<<<<<<<<{\eref{RcMisom_e}}_<<<<<<<<<<{\approx}&&
\pi_1^*\det\!\big(\dbar_{\C}\big)\!\otimes\!\pi_2^*\det\!\big(\dbar_{\C}\big)\,,}
\end{split}\EE
with the homotopy classes of the horizontal isomorphisms determined by 
the orientations of the line orbi-bundles~\eref{RcMisom_e}, commutes.
\end{lmm}

\begin{proof}
Let $([\cC_1],[\cC_2])\!\in\!\R\ov\cM_{g_1,\ell_1}^{\bu}\!\times\!\R\ov\cM_{g_2,\ell_2}^{\bu}$
and $[\cC]\!=\!\io_{\sqcup}([\cC_1],[\cC_2])$.
It can be assumed that the unmarked domains $(\Si_1,\si_1,\fJ_1)$ and $(\Si_2,\si_2,\fJ_2)$
of~$\cC_1$ and~$\cC_2$, respectively, are smooth and thus so is their disjoint union~$(\Si,\si,\fJ)$.
We denote the top exterior powers of the relevant tangent spaces 
of~$[\cC_i]$ and~$[\cC]$ by~$\cM_i$ and~$\cM$, respectively, and 
the left vertical isomorphism in~\eref{unionorientDM_e0} by~$\La_{\R}^{\top}$.
The horizontal isomorphisms in~\eref{unionorientDM_e0} factor through the horizontal isomorphisms~in
the diagram
$$\xymatrix{\cM\ar[d]^{\approx}_{\La_{\R}^{\top}}
\ar[rr]^{\tn{\ref{LSDolb_it}-\ref{ses_it}}}_{\approx}&& 
\det\dbar_{2(T\Si,\nd\si)}\ar[d]_{\eref{Runionorient_e0}}^{\approx}\ar[r]^{\eref{reldetdfn_e0}}_{\approx}&
\big(\!\det\dbar_{(\Si,\si)}\big)\ar[d]^{\eref{Runionorient_e0}}_{\approx}\\
\cM_1\!\otimes\!\cM_2\ar[rr]^>>>>>>>>>>>{\tn{\ref{LSDolb_it}-\ref{ses_it}}}_>>>>>>>>>>>{\approx}&& 
\big(\!\det\dbar_{2(T\Si_1,\nd\si_1)}\big)\!\otimes\!\big(\!\det\dbar_{2(T\Si_2,\nd\si_2)}\big)
\ar[r]^>>>>>{\eref{reldetdfn_e0}}_>>>>>{\approx}&  
\big(\!\det\dbar_{(\Si_1,\si_1)}\big)\!\otimes\!\big(\!\det\dbar_{(\Si_2,\si_2)}\big)\,.}$$
Since each of Steps~\ref{LSDolb_it}-\ref{ses_it} commutes with disjoint unions,
the left rectangle in this diagram commutes.
By the second claim of Corollary~\ref{unionorient_crl}, the right rectangle 
in this diagram commutes up to homotopy because
$$(g_1\!-\!1)(g_2\!-\!1)\!+\!
\big(g_1\!-\!1\!+\!\deg T\Si_1\big)\!\big(g_2\!-\!1\!+\!\deg T\Si_2\big)\in2\Z.$$
Combining these two statements, we obtain the claim.
\end{proof}

\begin{prp}\label{unionorient_prp}
Suppose $(X,\om,\phi)$, $n$, $J$, and $(L,[\psi],\fs)$ are as in Theorem~\ref{main_thm}
and \hbox{$g_i,\ell_i\!\in\!\Z$},  \hbox{$B_i\!\in\!H_2(X;\Z)$},
and \hbox{$[\u_i]\!\in\!\ov\fM_{g_i,\ell_i}^{\phi;\bu}(B_i;J)$} for $i\!=\!1,2$.
The natural isomorphism
\BE{unionorient_e}T_{[\u_1\sqcup\u_2]}\big(\ov\fM_{g_1+g_2-1,\ell_1+\ell_2}^{\phi;\bu}(B_1\!+\!B_2;J)\!\big)
\approx T_{[\u_1]}\big(\ov\fM_{g_1,\ell_1}^{\phi;\bu}(B_1;J)\!\big)
\!\oplus\!T_{[\u_2]}\big(\ov\fM_{g_2,\ell_2}^{\phi;\bu}(B_2;J)\!\big)\EE
is orientation-preserving with respect to the orientations induced by~$(L,[\psi],\fs)$
if and only~if 
$$(n\!-\!1)(g_1\!-\!1)(g_2\!-\!1)/2\in2\Z.$$
If in addition $\wt\phi$ is a conjugation on~$L$ lifting~$\phi$ and 
the homotopy class of isomorphisms~\eref{ROisom_e} is obtained 
via~\eref{ROrientCompl_e3}, then this isomorphism is orientation-preserving if and only~if
$$(n\!-\!1)(g_1\!-\!1)(g_2\!-\!1)/2\!+\!
\big(g_1\!-\!1\!+\!\lr{c_1(X,\om),B_1}/2\big)\big(g_2\!-\!1\!+\!\lr{c_1(X,\om),B_2}/2\big)\in2\Z\,. $$
\end{prp}

\begin{proof}
The second claim follows from the first and Lemma~\ref{CvsCanorient_lmm} applied three times.
It can be assumed that the domains $(\Si_1,\si_1,\fJ_1)$ and $(\Si_2,\si_2,\fJ_2)$
of~$[\u_1]$ and~$[\u_2]$, respectively, are smooth and stable with their markings
and thus so is their disjoint union~$(\Si,\si,\fJ)$.
Let $\cC_1,\cC_2,\cC$ be the marked domains of \hbox{$\u_1,\u_2,\u_1\!\sqcup\!\u_2$}, respectively,
and $u_1,u_2,u$ their map components.
We denote the top exterior powers of the relevant tangent spaces 
of~$[\cC_i],[\cC]$ and~$[\u_i],[\u_1\!\sqcup\!\u_2]$ by $\cM_i,\cM$ and $\fM_i,\fM$, respectively, and 
the determinants of the real CR-operators $D_{u_i^*(TX,\nd\phi)},D_{u^*(TX,\nd\phi)}$
and $\dbar_{(\Si_i,\si_i)},\dbar_{(\Si,\si)}$ by $D_i,D$ and $\dbar_i,\dbar$, respectively.\\

\begin{figure}
$$\xymatrix{\fM\ar[dd]\ar[rr]^{\eref{RfMisom_e}}&& 
D\!\otimes\!\cM\ar[d]_{\eref{Runionorient_e0}}\ar[r]^{\eref{reldetdfn_e0}}_{\eref{RcMisom_e}}& 
\dbar^{\otimes n}\!\otimes\!\dbar\ar@{-->}[d]^{\eref{Runionorient_e0}}\\
&& (D_1\!\otimes\!D_2)\!\otimes\!(\cM_1\!\otimes\!\cM_2)\ar[d]
\ar[r]^{\eref{reldetdfn_e0}}_{\eref{RcMisom_e}}& 
\big(\dbar_1^{\otimes n}\!\otimes\!\dbar_2^{\otimes n}\big)
\!\otimes\!\big(\dbar_1\!\otimes\!\dbar_2\big)\ar[d]\\
\fM_1\!\otimes\!\fM_2\ar[rr]^<<<<<<<<{\eref{RfMisom_e}}&& 
(D_1\!\otimes\!\cM_1)\!\otimes\!(D_2\!\otimes\!\cM_2)
\ar[r]^{\eref{reldetdfn_e0}}_{\eref{RcMisom_e}}&
\big(\dbar_1^{\otimes n}\!\otimes\!\dbar_1\big)
\!\otimes\!\big(\dbar_2^{\otimes n}\!\otimes\!\dbar_2\big)}$$
$$\xymatrix{D\ar[d]_{\eref{Runionorient_e0}}\ar[r]^{\eref{reldetdfn_e0}}& 
\dbar^{\otimes n}\ar@{-->}[d]^{\eref{Runionorient_e0}}\\
D_1\!\otimes\!D_2\ar[r]^{\eref{reldetdfn_e0}}&
\dbar_1^{\otimes n}\!\otimes\!\dbar_2^{\otimes n}}$$
\caption{Isomorphisms inducing the orientations of Proposition~\ref{unionorient_prp}} 
\label{unionorient_fig}
\end{figure}

\noindent
The tangent spaces in the statement of the proposition are oriented by the top 
and bottom rows of the top diagram in Figure~\ref{unionorient_fig} from
\begin{enumerate}[label=($\fo\arabic*$),leftmargin=*]

\item the canonical orientations on the $(n\!+\!1)/2$ factors of 
$\dbar^{\otimes2},\dbar_i^{\otimes2}$,

\item the projection (resp.~canonical) orientations on the determinants of real CR-operators
on the real bundle pairs 
$u^*(L^*\!\oplus\!\phi^*\ov{L}^*,\phi_{L^*}^{\oplus}),
u_i^*(L^*\!\oplus\!\phi^*\ov{L}^*,\phi_{L^*}^{\oplus})$ in the first
(resp.~$2u^*(L^*,\wt\phi^*),2u_i^*(L^*,\wt\phi^*)$
in the second) claim of the proposition.

\end{enumerate}
The vertical isomorphisms at the bottom of this diagram interchange the middle factors
in their domains and multiply the result by~$(-1)$ to the power of the index 
of~$D_2$ times the dimension of~$\cM_1$.
Since $n\!\not\in\!2\Z$ and $\lr{c_1(X,\om),B_2}\!\in\!2\Z$,
\begin{equation*}\begin{split}
\big(\ind\,D_2\big)\big(\!\dim\cM_1\big)
&\cong(g_1\!-\!1)(g_2\!-\!1) \\
&\cong\big(\ind\,n\dbar_{(\Si_2,\si_2)}\big)\!\big(\ind\,\dbar_1\big)\mod2
\end{split}\end{equation*}
The two rectangles with solid arrows in this diagram commute.\\

\noindent
The isomorphisms in the upper right rectangle in the top diagram are the tensor products
of the isomorphisms in the bottom diagram and in~\eref{unionorientDM_e0}.
For the purposes of~\eref{Runionorient_e0}, 
$\dbar^{\otimes n}$ is viewed as the determinant of the CR-operator $n\dbar_{(\Si,\si)}$
on~$(\Si\!\times\!\C^n,\si\!\times\!\fc)$ in these diagrams. 
By Corollary~\ref{unionorient_crl}, the bottom diagram would commute up to homotopy 
if the dashed arrow were multiplied by~$(-1)$ to the power~of 
\BE{unionorient_e37}
(g_1\!-\!1)(g_2\!-\!1)  \quad\Big(\hbox{resp.~}
(g_1\!-\!1)(g_2\!-\!1)\!+\!
\big(g_1\!-\!1\!+\!\lr{c_1(L),B_1}\!\big)\!\big(g_2\!-\!1\!+\!\lr{c_1(L),B_2}\!\big)\!\Big)\EE
in the first (resp.~second) claim of the proposition.
Along with Lemma~\ref{unionorientDM_lmm}, this implies that the upper right rectangle in the top diagram 
would commute up to homotopy if the dashed arrow were multiplied by~$(-1)$ 
to the power of~\eref{unionorient_e37}.\\

\noindent
The pullback of the canonical orientation on $\dbar_1^{\otimes(n+1)}\!\otimes\!\dbar_2^{\otimes(n+1)}$
by the bottom right vertical arrow in the top diagram (with the specified sign)
and the top right vertical arrow (without any sign) differs from 
the canonical orientation on $\dbar^{\otimes(n+1)}$ by~$(-1)$ to the power~of
$$\binom{n\!+\!1}{2}\big(\ind\,\dbar_1\big)\big(\ind\,\dbar_2\big)
\cong (n\!+\!1)(g_1\!-\!1)(g_2\!-\!1)/2 \mod2.$$
Combining this with the correcting sign~\eref{unionorient_e37} for 
the top right vertical arrow, we obtain the claim.
\end{proof}

\noindent
For $g\!\in\!\Z$, $\ell\!\in\!\Z^{\ge0}$, \hbox{$B\!\in\!H_2(X;\Z)$}, and $J\!\in\!\cJ_{\om}$,
we denote by $\ov\fM_{g,\ell}^{\bu}(B;J)$
the moduli space of stable $J$-holomorphic degree~$B$ maps 
from closed, possibly nodal and disconnected, Riemann surfaces of arithmetic genus~$g$
with $\ell$~marked points.
If in addition \hbox{$[\ell]\!=\!S^+\!\sqcup\!S^-$}, let
$$\io_{g;S^+,S^-}\!:  
\ov\fM_{g,\ell}^{\bu}(B;J)\lra \ov\fM_{2g-1,\ell}^{\phi;\bu}(B\!-\!\phi_*(B);J\big)$$
be the open immersion sending the equivalence class of a stable map~$\u$ 
with the domain~$(\Si,\fj)$ and map component~$u$ to the equivalence class of 
the map
$$u\!\sqcup\!\{\phi\!\circ\!u\}\!:
\big(\Si\!\sqcup\!\Si,\fj\!\sqcup\!(-\fj)\!\big)\lra(X,J)$$
with the involution $\si$ interchanging the two copies of~$\Si$
and the first copy of~$\Si$ carrying the marked points~$z_i^+$ with $i\!\in\!S^+$
and $z_i^-$ with $i\!\in\!S^-$ at the positions of the corresponding marked points~$z_i$
of~$\u$.

\begin{prp}\label{Dblorient_prp}
Suppose $(X,\om,\phi)$, $J$, $(L,[\psi],\fs)$, $g$, $\ell$, and $B$ are as in Theorem~\ref{main_thm},
\hbox{$[\ell]\!=\!S^+\!\sqcup\!S^-$}, and $[\u]\!\in\!\ov\fM_{g,\ell}^{\bu}(B;J)$.
The isomorphism
$$\nd_{[\u]}\io_{g;S^+,S^-}\!:
T_{[\u]}\big(\ov\fM_{g,\ell}^{\bu}(B;J)\!\big)\lra
T_{\io_{g;S^+,S^-}([\u])}\ov\fM_{2g-1,\ell}^{\phi;\bu}\big(B\!-\!\phi_*(B);J\big)$$
is orientation-preserving with respect to the orientation on the target induced
by $(L,[\psi],\fs)$ via~\eref{ROrientCompl_e2} and the complex orientation
on the domain if and only if $\lr{c_1(L),\phi_*(B)}\!+\!|S^-|\!\in\!2\Z$.
If in addition $\wt\phi$ is a conjugation on~$L$ lifting~$\phi$ and 
the homotopy class of isomorphisms~\eref{ROisom_e} is obtained 
via~\eref{ROrientCompl_e3}, then this isomorphism is orientation-preserving if and only~if
\hbox{$(g\!-\!1)\!+\!|S^-|\!\in\!2\Z$}.
\end{prp}

\begin{proof} The second claim is \cite[Theorem~1.4]{RealGWsII};
it also follows from the first claim and Lemma~\ref{CvsCanorient_lmm}
with $g$, $\rk_{\C}\!L$, and $\deg L$ replaced by $2g\!-\!1$, 1, and $-\lr{c_1(X,\om),B}$,
respectively.
The first claim stated for certain one-dimensional targets $(X,\om,\phi)$ 
is \cite[Lemma~3.1]{GI}, but its proof is a slight modification of 
the proof of \cite[Theorem~1.4]{RealGWsII} summarized below and applies in general.
For $g\!\in\!\Z$ and $\ell\!\in\!\Z^{\ge0}$, we denote by $\ov\cM_{g,\ell}^{\bu}$ 
the Deligne-Mumford moduli space of stable closed, possibly nodal and disconnected, 
Riemann surfaces~$(\Si,\fj)$ of arithmetic genus~$g$
with $\ell$~marked points and~by
$$\det\dbar_{\C}\lra \ov\cM_{g,\ell}^{\bu}$$
the real line bundle formed by the determinants of the standard $\dbar$-operators
on Riemann surfaces.\\

\noindent
As with Proposition~\ref{unionorient_prp}, it is sufficient to assume that
the domain~$\Si$ of~$[\u]$ is smooth and is stable with its markings.
In both cases, the proof then starts with the isomorphism~\eref{RfMisom_e}.
By \cite[Lemma~3.2]{RealGWsII}, the isomorphism
$$\La^{\top}_{\R}\big(T(\R\cM_{2g-1,\ell}^{\bu})\!\big)
\!\otimes\!\big(\!\det\dbar_{\C}\big)\!\Big|_{\io_{g;S^+,S^-}([\u])}
\lra \La^{\top}_{\R}\big(T\cM_{g,\ell}^{\bu}\big)
\!\otimes\!\big(\!\det\dbar_{\C}\big)\!\Big|_{\u}$$
obtained by restricting to the first copy of~$\Si$ similarly to~\eref{Dblorient_e}
respects the natural orientation of the domain provided by \cite[Proposition~5.9]{RealGWsI}
and the complex orientation on the target if and only~if
\hbox{$(g\!-\!1)\!+\!|S^-|$} is even.
The summand~$(g\!-\!1)$ arises from Step~\ref{SD_it} because
the two natural complex orientations of the real dual of 
a finite-dimensional complex vector space are the same if and only if
its complex dimension is even; see the beginning of \cite[Section~3.1]{RealGWsII}.
The summand~$|S^-|$ arises from Steps~\ref{ses_it} and~\ref{SD_it}
because the complex vector space in~\ref{ses_it} is the direct sum of the cotangent spaces
at the marked points~$z_i^+$ (and it is dualized in Step~\ref{SD_it}).
The analogous comparison of the orientation on the first factor on
the right-hand side of~\eref{RfMisom_e}
induced by~$(L,[\psi],\fs)$ and the orientation obtained by restricting to 
the first copy of~$\Si$ is provided by Corollary~\ref{Dblorient_crl}.
\end{proof}

\subsection{Orientations of nodal strata}
\label{MSorient_subs2}

\noindent
Let $(X,\om,\phi)$, $J$, $g,\ell$, and $B$ be as in Theorem~\ref{main_thm}.
We denote by $\cN_XX^{\phi}$ the normal bundle of~$X^{\phi}$ in~$X$.
The symplectic orientation on~$X$ and an orientation on~$X^{\phi}$,
such as one induced by a real orientation~$(L,[\psi],\fs)$ on~$(X,\om,\phi)$,
determine an orientation on~$\cN_XX^{\phi}$ via the short exact sequence
$$0\lra TX^{\phi}\lra TX|_{X^{\phi}}\lra \cN_XX^{\phi}\lra0$$
of real vector bundles over~$X^{\phi}$.
For each $i\!\in\![\ell]$, let 
$$\ev_i\!:\ov\fM_{g,\ell}^{\phi;\bu}(B;J)\lra X \qquad\hbox{and}\qquad
\cL_i,\cL_i^-\lra \ov\fM_{g,\ell}^{\phi;\bu}(B;J)$$
be the natural evaluation map at the marked point~$z_i^+$
and the universal tangent line orbi-bundles at the marked points~$z_i^+$ and~$z_i^-$, 
respectively.
The involutions~$\si$ on the domains of the elements of~$\ov\fM_{g,\ell}^{\phi;\bu}(B;J)$
induce an $\C$-antilinear isomorphism
$$\Phi_i\!:\cL_i\lra\cL_i^-$$
covering the identity of~$\ov\fM_{g,\ell}^{\phi;\bu}(B;J)$. 
Let 
$$\big(\cL_i\!\otimes\!\cL_i^-\big)^{\R}=
\big\{rv\!\otimes_{\C}\!\Phi_i(v)\!:v\!\in\!\cL_i,\,r\!\in\!\R\big\}.$$
This real line bundle over $\ov\fM_{g,\ell}^{\phi;\bu}(B;J)$ is canonically oriented
by the standard orientation of~$\R$.\\

\noindent
Define
\begin{equation*}\begin{split}
\ov\fM_{g-2,\ell+2}'^{\phi;\bu}(B;J)
&=\big\{[\u]\!\in\!\ov\fM_{g-2,\ell+2}^{\phi;\bu}(B;J)\!:\,
\ev_{\ell+1}([\u])\!=\!\ev_{\ell+2}([\u])\!\big\},\\
\ov\fM_{g-2,\ell+2}''^{\phi;\bu}(B;J)
&=\big\{[\u]\!\in\!\ov\fM_{g-1,\ell+1}^{\phi;\bu}(B;J)\!:\,
\ev_{\ell+1}([\u])\!\in\!X^{\phi}\!\big\}.
\end{split}\end{equation*}
The short exact sequences
\begin{equation*}\begin{split}
0&\lra T\ov\fM_{g-2,\ell+2}'^{\phi;\bu}(B;J)\lra 
T\ov\fM_{g-2,\ell+2}^{\phi;\bu}(B;J)\!\big|_{\ov\fM_{g-2,\ell+2}'^{\phi;\bu}(B;J)} 
\lra \ev_{\ell+1}^{\,*}TX\lra 0,\\
0&\lra T\ov\fM_{g-1,\ell+1}''^{\phi;\bu}(B;J)\lra 
T\ov\fM_{g-1,\ell+1}^{\phi;\bu}(B;J)\!\big|_{\ov\fM_{g-1,\ell+1}''^{\phi;\bu}(B;J)} 
\lra \ev_{\ell+1}^{\,*}\cN_XX^{\phi}\lra 0
 \end{split}\end{equation*}
induce isomorphisms
\BE{SubIsom_e}\begin{split}
\La_{\R}^{\top}\big(T\big(\ov\fM_{g-2,\ell+2}^{\phi;\bu}(B;J)\!\big)\!\big)
\!\big|_{\ov\fM_{g-2,\ell+2}'^{\phi;\bu}(B;J)}
&\approx 
\La_{\R}^{\top}\big(T\big(\ov\fM_{g-2,\ell+2}'^{\phi;\bu}(B;J)\!\big)\!\big)
\!\otimes\! \ev_{\ell+1}^*\big(\La_{\R}^{\top}(TX)\!\big),\\
\La_{\R}^{\top}\big(T\big(\ov\fM_{g-1,\ell+1}^{\phi;\bu}(B;J)\!\big)\!\big)
\!\big|_{\ov\fM_{g-1,\ell+1}''^{\phi;\bu}(B;J)}
&\approx 
\La_{\R}^{\top}\big(T\big(\ov\fM_{g-1,\ell+1}''^{\phi;\bu}(B;J)\!\big)\!\big)
\!\otimes\! \ev_{\ell+1}^*\big(\La_{\R}^{\top}(\cN_XX^{\phi})\!\big)
\end{split}\EE
of real line bundles over $\ov\fM_{g-2,\ell+2}'^{\phi;\bu}(B;J)$
and $\ov\fM_{g-1,\ell+1}''^{\phi;\bu}(B;J)$, respectively.\\

\noindent
We denote by 
\BE{Riogldfn_e2}
\io_{g,\ell}^{\C}\!:\ov\fM_{g-2,\ell+2}'^{\phi;\bu}(B;J)\lra\ov\fM_{g,\ell}^{\phi;\bu}(B;J)
\quad\hbox{and}\quad
\io_{g,\ell}^E\!:\ov\fM_{g-1,\ell+1}''^{\phi;\bu}(B;J)\lra\ov\fM_{g,\ell}^{\phi;\bu}(B;J)\EE
the immersion obtained by identifying the marked points~$z_{\ell+1}^+$ and~$z_{\ell+1}^-$ 
of the domain of each map with~$z_{\ell+2}^+$ and~$z_{\ell+2}^-$, respectively,
to form a conjugate pair of nodes
and the immersion obtained by identifying the marked point~$z_{\ell+1}^+$ 
of the domain of each map with~$z_{\ell+1}^-$ to form an $E$-node.
The first immersion is generically $4\!:\!1$ onto its image,
while the second is generically $2\!:\!1$.
There are canonical isomorphisms
\begin{equation*}\begin{split}
\cN\io_{g,\ell}^{\C}&\equiv 
\frac{\io_{g,\ell}^{\C\,*}T(\ov\fM_{g,\ell}^{\phi;\bu}(B;J)\!)}
{\nd\io_{g,\ell}^{\C}(T(\ov\fM_{g-2,\ell+2}'^{\phi;\bu}(B;J)\!)\!)}
\approx \cL_{\ell+1}\!\otimes_{\C}\!\cL_{\ell+2},\\
\cN\io_{g,\ell}^E&\equiv 
\frac{\io_{g,\ell}^{E\,*}T(\ov\fM_{g,\ell}^{\phi;\bu}(B;J)\!)}
{\nd\io_{g,\ell}^E(T(\ov\fM_{g-1,\ell+1}''^{\phi;\bu}(B;J)\!)\!)}
\approx \big(\cL_{\ell+1}\!\otimes\!\cL_{\ell+1}^-\big)^{\R}\,.
\end{split}\end{equation*}
They induce isomorphisms
\BE{CompOrient_e0}\begin{split} 
\io_{g,\ell}^{\C\,*}\big(\La_{\R}^{\top}
\big(T(\ov\fM_{g,\ell}^{\phi;\bu}(B;J)\!)\!\big)\!\big)
&\approx \La_{\R}^{\top}\big(T\big(\ov\fM_{g-2,\ell+2}'^{\phi;\bu}(B;J)\!\big)\!\big)
\!\otimes\! \La_{\R}^2\big(\cL_{\ell+1}\!\otimes_{\C}\!\cL_{\ell+2}\big),\\
\io_{g,\ell}^{E\,*}\big(\La_{\R}^{\top}
\big(T(\ov\fM_{g,\ell}^{\phi;\bu}(B;J)\!)\!\big)\!\big)
&\approx \La_{\R}^{\top}\big(T\big(\ov\fM_{g-1,\ell+1}''^{\phi;\bu}(B;J)\!\big)\!\big)
\!\otimes\! \big(\cL_{\ell+1}\!\otimes\!\cL_{\ell+1}^-\big)^{\R}
\end{split}\EE
of real line bundles over $\ov\fM_{g-2,\ell+2}'^{\phi;\bu}(B;J)$
and $\ov\fM_{g-1,\ell+1}''^{\phi;\bu}(B;J)$, respectively.\\

\noindent
A real orientation~$(L,[\psi],\fs)$ on $(X,\om,\phi)$ and the symplectic orientation on~$X$
induce orientations on $\ov\fM_{g-2,\ell+2}'^{\phi;\bu}(B;J)$
and $\ov\fM_{g-1,\ell+1}''^{\phi;\bu}(B;J)$ via~\eref{ROrientCompl_e2} and~\eref{SubIsom_e}.
We will call these induced orientations of the two moduli spaces 
\sf{the intrinsic orientations induced by~$(L,[\psi],\fs)$ via~\eref{ROrientCompl_e2}}.
A real orientation~$(L,[\psi],\fs)$ on~$(X,\om,\phi)$,
the complex orientation on~$\cL_{\ell+1}\!\otimes_{\C}\!\cL_{\ell+2}$,
and the canonical orientation on~$(\cL_{\ell+1}\!\otimes\!\cL_{\ell+1}^-)^{\R}$
induce orientations on these moduli spaces via~\eref{ROrientCompl_e2} and~\eref{CompOrient_e0}.
We will call the latter induced orientations of the two moduli spaces 
\sf{the pullback orientations induced by~$(L,[\psi],\fs)$ via~\eref{ROrientCompl_e2}}.
If in addition $\wt\phi$ is a conjugation on~$L$ lifting~$\phi$,
we define \sf{the intrinsic} and \sf{pullback orientations} on 
$\ov\fM_{g-2,\ell+2}'^{\phi;\bu}(B;J)$ and $\ov\fM_{g-1,\ell+1}''^{\phi;\bu}(B;J)$
\sf{induced by~$(L,[\psi],\fs)$ and~$\wt\phi$
via~\eref{ROrientCompl_e3}} as above with~\eref{ROrientCompl_e2} replaced by~\eref{ROrientCompl_e3}.

\begin{prp}\label{Rnodisom_prp}
Suppose $(X,\om,\phi)$, $J$, $(L,[\psi],\fs)$, $g$, $\ell$, and~$B$ are as in Theorem~\ref{main_thm}.
The intrinsic and pullback orientations on $\ov\fM_{g-2,\ell+2}'^{\phi;\bu}(B;J)$
induced by~$(L,[\psi],\fs)$ via~\eref{ROrientCompl_e2} are the same.
If in addition $\wt\phi$ is a conjugation on~$L$ lifting~$\phi$,
the intrinsic and pullback orientations on 
$\ov\fM_{g-2,\ell+2}'^{\phi;\bu}(B;J)$ 
induced by~$(L,[\psi],\fs)$ and~$\wt\phi$ via~\eref{ROrientCompl_e3} are opposite.
\end{prp}

\begin{proof}
The second claim is \cite[Theorem~1.2]{RealGWsII};
it also follows from the first claim and Lemma~\ref{CvsCanorient_lmm} applied twice.
The first claim stated for one-dimensional targets $(X,\om,\phi)$ 
is \cite[Proposition~7.3]{GI2}, but its proof is a slight modification of 
the proof of \cite[Theorem~1.2]{RealGWsII} summarized below and
again applies in general.
The full proof is similar to that of Proposition~\ref{RnodisomE_prp},
but without any concern about reordering the orienting factors,
because the degree of at least one of the factors being reordered, 
$\La_{\R}^{\top}(TX)$, $\La_{\R}^2(\cL_{\ell+1}\!\otimes_{\C}\!\cL_{\ell+2})$,
or~$(\La_{\R}^2\C)$, is even.\\

\noindent
Similarly to Propositions~\ref{unionorient_prp} and~\ref{Dblorient_prp}, 
it is sufficient to establish the claims over an element~$[\wh\u]$
of $\ov\fM_{g-2,\ell+2}'^{\phi;\bu}(B;J)$ with smooth domain 
$(\wh\Si,\wh\si,\wh\fj)$ so that the domain~$(\Si,\si,\fj)$ of 
$\io_{g,\ell}^{\C}([\wh\u])$ has one conjugate pair of nodes and no other nodes.
We denote by~$\wh\cC$ and~$\cC$ the marked domains of~$\wh\u$ and~$\io_{g,\ell}^{\C}([\wh\u])$,
respectively, and by $\cL_{\wh\cC}$ the fiber of~$\cL_{\ell+1}\!\otimes_{\C}\!\cL_{\ell+2}$
at~$\wh\u$. 
In both cases, the proof again starts with the isomorphism~\eref{RfMisom_e}.
By \cite[Proposition~4.18]{RealGWsII}, the diagram
\BE{RnodisomDM_e0}\begin{split}
\xymatrix{\La_{\R}^{\top}\big(T_{[\wh\cC]}(\R\ov\cM_{g-1,\ell+1}^{\bu})\!\big)
\!\otimes\!\big(\La_{\R}^2\cL_{\wh{u}}\big)
\ar[rr]^{\eref{RcMisom_e}}_{\approx}  \ar[d]_{\eref{CompOrient_e0}}^{\approx}&& 
\big(\!\det\dbar_{(\wh\Si,\wh\si)}\big)\!\otimes\!\big(\La_{\R}^2\cL_{\wh\cC}\big)
\ar[d]^{\eref{Rnodisom_e}}_{\approx}\\
\La_{\R}^{\top}\big(T_{[\cC]}\big(\R\ov\cM_{g,\ell}^{\bu}\big)\!\big)
\ar[rr]^>>>>>>>>>>>>{\eref{RcMisom_e}}_>>>>>>>>>>>>{\approx} &&
\big(\!\det\dbar_{(\Si,\si)}\big)\!\otimes\!\big(\La_{\R}^2\C\big)
\!\otimes\!\big(\La_{\R}^2\cL_{\wh\cC}\big)}
\end{split}\EE
with the homotopy classes of the horizontal isomorphisms determined 
by the orientation of the line orbi-bundle~\eref{RcMisom_e}
and the complex orientations of~$\C$ and~$\cL_{\wh\cC}$,
does not commute up to homotopy.
The reversal of the orientation occurs due to Step~\ref{SD_it}
because the two natural complex orientations of the real dual of~$\C$ 
are opposite.
The analogous comparison of the orientations on
the first factor on the right-hand side of~\eref{RfMisom_e} induced by~$(L,[\psi],\fs)$
is provided by Corollary~\ref{Rnodisom_crl}.
\end{proof}

\noindent
As explained in Section~\ref{MSorient0_subs}, there is a choice involved
in Step~\ref{SD_it} of the construction of the orientation on the real line orbi-bundle~\eref{RcMisom_e};
this choice corresponds to a choice of orientation on \hbox{$\R\ov\cM_{0,2}\!\approx\!\R\P^1$}.
Reversing the last choice reverses the orientations on
the real line orbi-bundle in~\eref{RcMisom_e} 
(and thus on the moduli space of maps in~\eref{RfMisom_e}) 
if and only if \hbox{$g\!\in\!2\Z$}. 
Thus, the analogues of \cite[Proposition~4.18]{RealGWsII} and Proposition~\ref{unionorient_prp} above
for the second immersion in~\eref{Riogldfn_e2} depend on the choice of orientation on~$\R\ov\cM_{0,2}$.\\

\noindent
For $g\!\in\!\Z$ and $\ell\!\in\!\Z^{\ge0}$, let
$$\io_E\!:\R\ov\cM_{g-1,\ell+1}^{\bu}\lra \R\ov\cM_{g,\ell}^{\bu}$$
be the immersion obtained by identifying the marked point~$z_{\ell+1}^+$ 
of each curve with~$z_{\ell+1}^-$ to form an $E$-node.
We denote by~$\tau'_{\R\ov\cM_{g-1,\ell+1}^{\bu}}$ the real line bundle 
$$\R\ov\cM_{g-1,\ell+1}^{\bu}\!\times\!\big(\C/\R\big)\lra\R\ov\cM_{g-1,\ell+1}^{\bu};$$
it has a canonical orientation~$\fo_{\C/\R}$ inherited from 
the standard orientations of~$\C$ and~$\R$.

\begin{lmm}\label{unionorientEDM_lmm}
Suppose $g\!\in\!\Z$ and $\ell\!\in\!\Z^{\ge0}$.
If $\R\ov\cM_{0,2}$ is oriented via the diffeomorphism~\eref{cMtaudiff_e},
then the diagram 
\BE{RnodisomEDM_e0}\begin{split}
\xymatrix{\La_{\R}^{\top}\big(T(\R\ov\cM_{g-1,\ell+1}^{\bu})\!\big)
\!\otimes\!\big(\cL_{\ell+1}\!\otimes\!\cL_{\ell+1}^-\big)^{\R}
\ar[r]^<<<<<<<<<<<{\eref{RcMisom_e}}_<<<<<<<<<<<{\approx}  \ar[d]_{\eref{CompOrient_e0}}^{\approx}& 
\big(\!\det\dbar_{\C}\big)\!\otimes\!\big(\cL_{\ell+1}\!\otimes\!\cL_{\ell+1}^-\big)^{\R}
\ar[d]^{\eref{RnodisomE_e}}_{\approx}\\
\io_E^*\La_{\R}^{\top}\big(T\big(\R\ov\cM_{g,\ell}^{\bu}\big)\!\big)
\ar[r]^>>>>>>>>>>>>{\eref{RcMisom_e}}_>>>>>>>>>>>>{\approx} &
\io_E^*\big(\!\det\dbar_{\C}\big)\!\otimes\!\tau'_{\R\ov\cM_{g-1,\ell+1}^{\bu}}
\!\otimes\!\big(\cL_{\ell+1}\!\otimes\!\cL_{\ell+1}^-\big)^{\R}}
\end{split}\EE 
with the homotopy classes of the horizontal isomorphisms determined 
by the orientation of the line orbi-bundle~\eref{RcMisom_e}
and the canonical orientations of~$\tau'_{\R\ov\cM_{g-1,\ell+1}^{\bu}}$
and $(\cL_{\ell+1}\!\otimes\!\cL_{\ell+1}^-)^{\R}$, commutes.
\end{lmm}

\begin{proof} By Lemma~\ref{unionorientDM_lmm}, 
it is sufficient to establish the claim over a single element~$[\wh\cC]$ 
of each connected component of~$\R\ov\cM_{g-1,\ell+1}^{\bu}$
parametrizing either connected marked curves or basic marked doublets as in~\eref{Dbldfn_e}
with~$\Si_1$ connected.
By adding more marked points if necessary, we can choose~$\wh\cC$ so that 
the irreducible component~$\Si_1$ of~$\wh\cC$ carrying the marked point~$z_{\ell+1}^+$ is~$\P^1$
and also carries the marked point~$z_1^+$ and either
\begin{enumerate}[label=(\arabic*),leftmargin=*] 

\item the marked point~$z_2^-$ and no other marked or nodal points

\item one non-real node and no other marked or real points.

\end{enumerate}
In the first case, $\wh\cC$ is a genus 0 doublet with 3 conjugate pairs of marked points,
and so~$g\!=\!0$, $\ell\!=\!2$, and $\R\ov\cM_{g-1,\ell+1}^{\bu}$ consists of four points 
in this case.
Since the degrees of~$\La_{\R}^2\C$ and $\La_{\R}^2\cL_{\wh\cC}$ in~\eref{RnodisomDM_e0} 
are even,
the Exact Squares property, \cite[Proposition~4.18]{RealGWsII}, and  
Lemma~\ref{unionorientDM_lmm} reduce the second case above to the first.\\

\noindent
We now assume that $\wh\cC$ is as in the first case above.
Thus, $\io_E([\wh\cC])$ corresponds to \hbox{$0\!\in\!\R\!\subset\!\R\P^1$} 
under the diffeomorphism~\eref{cMtaudiff_e}.
The kernel of the surjective operator~$\dbar_{(\Si,\si)}$ at~$\io_E([\wh\cC])$
consists of constant $\R$-valued functions and thus has a canonical orientation~$\fo_{\dbar}$.
By the assumption in the statement of the lemma, 
the orientation of $T_{\io_E([\wh\cC])}(\R\ov\cM_{0,2})$
corresponding to~$\fo_{\dbar}$ via the orientation of the real line orbi-bundle~\eref{RcMisom_e}
is given by the diffeomorphism~\eref{cMtaudiff_e}.
Since the $E$-node of~$\io_E([\wh\cC])$ is deformed to a real circle 
along the positive direction of~$(\cL_{\ell+1}\!\otimes\!\cL_{\ell+1}^-)^{\R}$,
the canonical orientation of~$(\cL_{\ell+1}\!\otimes\!\cL_{\ell+1}^-)^{\R}$
at~$[\wh\cC]$ corresponds to the orientation of $T_{\io_E([\wh\cC])}(\R\ov\cM_{0,2}^{\bu})$.
Thus, the pullback of the orientation $\fo_{\dbar}\!\otimes\!\fo_{\C/\R}$ 
by the bottom and left isomorphisms in~\eref{RnodisomEDM_e0}
corresponds to the orientation of~$\wh\cC$ as a positive point.\\

\noindent
The kernel of the surjective operator~$\dbar_{(\wh\Si,\wh\si)}$ at~$[\wh\cC]$
consists of $\C$-valued functions that are constant on each of the two components of~$\wh\cC$
and take conjugate values on~them.
The orientation~$\wh\fo_{\dbar}$ on~$\dbar_{(\wh\Si,\wh\si)}$ induced from
the complex orientation on~$\C$ by restricting these functions on~$\Si_1$
corresponds to the orientation $\fo_{\dbar}\!\otimes\!\fo_{\C/\R}$
via the right isomorphism in~\eref{RnodisomEDM_e0}.
By \cite[Lemma~3.2]{RealGWsII}, the top isomorphism in~\eref{RnodisomEDM_e0}
identifies~$\wh\fo_{\dbar}$ with the orientation of $T_{[\wh\cC]}(\R\ov\cM_{-1,3}^{\bu})$
obtained from the complex orientation of~$\ov\cM_{0,3}$ by restricting to the component~$\Si_1$.
Thus, the pullback of $\fo_{\dbar}\!\otimes\!\fo_{\C/\R}$ 
by the right and top isomorphisms in~\eref{RnodisomEDM_e0}
also corresponds to the orientation of~$\wh\cC$ as a positive point.
This establishes the claim.
\end{proof}

\begin{prp}\label{RnodisomE_prp}
Suppose $(X,\om,\phi)$, $n$, $J$, $(L,[\psi],\fs)$, $g$, $\ell$, and~$B$ are as in Theorem~\ref{main_thm}.
The intrinsic and pullback orientations on $\ov\fM_{g-1,\ell+1}''^{\phi;\bu}(B;J)$
induced by~$(L,[\psi],\fs)$ via~\eref{ROrientCompl_e2} are opposite.
If in addition $\wt\phi$ is a conjugation on~$L$ lifting~$\phi$,
the intrinsic and pullback orientations on 
$\ov\fM_{g-1,\ell+1}''^{\phi;\bu}(B;J)$ 
induced by~$(L,[\psi],\fs)$ and~$\wt\phi$ via~\eref{ROrientCompl_e3} are the same 
if and only if $g\!+\!\lr{c_1(X,\om),B}/2\!\in\!2\Z$.
\end{prp}

\begin{proof} It is sufficient to establish the claims over an element~$[\wh\u]$
of $\ov\fM_{g-1,\ell+1}''^{\phi;\bu}(B;J)$ with smooth domain $(\wh\Si,\wh\si,\wh\fj)$.
Let $[\u]\!=\!\io_{g,\ell}^E([\wh\u])$;
the domain~$(\Si,\si,\fj)$ of~$\u$ then has one $E$-node and no other nodes.
Let~$\wh\cC$ and~$\cC$ be the marked domains of~$\wh\u$ and~$\u$,
respectively, $\wh{u},u$ be their map components, and 
$$\wh\fM=T_{[\wh\u]}\ov\fM_{g-1,\ell+1}^{\phi;\bu}(B;J),\quad
\wh\fM''=T_{[\wh\u]}\ov\fM_{g-1,\ell+1}''^{\phi;\bu}(B;J),\quad
\fM=T_{[\u]}\ov\fM_{g,\ell}''^{\phi;\bu}(B;J)\,.$$
We denote the relevant tangent spaces of~$[\wh\cC],[\cC]$ by $\wh\cM,\cM$, respectively,
the real CR-operators $D_{\wh{u}^*(TX,\nd\phi)},D_{u^*(TX,\nd\phi)}$ by~$\wh{D},D$,
and the homomorphisms from the vector spaces~$\cN_XX^{\phi}|_{\ev_{\ell+1}(\wh\u)}$
and~$(\cL_{\ell+1}\!\otimes\!\cL_{\ell+1}^-)^{\R}|_{\wh\u}$ to~$\{0\}$
by~$\0_{X^{\phi}}$ and~$\0_{\cL}$.
Let~$\fo_{\C/\R}$, $\fo_{X^{\phi}}$, and~$\fo_{\cL}$ be 
the orientation on~$\C/\R$ induced by the complex orientation of~$\C$ and
the standard orientation of~$\R\!\subset\!\C$,
the orientation on $\det\0_{X^{\phi}}$ 
determined by the real orientation~$(L,[\psi],\fs)$ on $(X,\om,\phi)$,
and the orientation on~$\det\0_{\cL}$ determined by the canonical 
orientation on~$(\cL_{\ell+1}\!\otimes\!\cL_{\ell+1}^-)^{\R}$, respectively.
For $k\!\not\in\!2\Z$, the standard orientation~$\fo_{\C^k/\R^k}$ on~$\C^k/\R^k$
agrees with $\fo_{\C/\R}^{\otimes k}$ if and only if~$(k\!-\!1)/2$ is even.\\

\noindent
An orientation~$\fo$ on~$\det\dbar_{(\Si,\si)}$ 
determines an orientation~$\wh\fo$ on~$\det\dbar_{(\wh\Si,\wh\si)}$ 
via~\eref{RnodisomE_e} with $V_x\!=\!\C$ from~$\fo_{\C/\R}$,
an orientation~$\fo_{\cM}$ on~$\cM$ from the canonical orientation of 
the line orbi-bundle~\eref{RcMisom_e}, and an orientation~$\fo_D$ 
on~$\det D$ via~\eref{ROrientCompl_e2} from~$(L,[\psi],\fs)$.
The orientation~$\wh\fo$ similarly determines an orientation~$\wh\fo_{\cM}$ on~$\wh\cM$ 
and an orientation~$\wh\fo_D$ on~$\det\wh{D}$.
The orientation~$\wh\fo_{\fM}$ on~$\wh\fM$ (resp.~$\fo_{\fM}$ on~$\fM$) 
determined by~$(L,[\psi],\fs)$ via~\eref{ROrientCompl_e2}
is obtained from the orientations~$\wh\fo_D$ and~$\wh\fo_{\cM}$ (resp.~$\fo_D$ and~$\fo_{\cM}$)
via the isomorphism~\eref{RfMisom_e}, which is associated with the middle row in 
the first (resp.~second) exact square of Figure~\ref{RnodisomE_fig}.
The intrinsic (resp.~pullback) orientation~$\wh\fo_{in}''$ (resp.~$\wh\fo_{pu}''$)
induced by~$(L,[\psi],\fs)$ via~\eref{ROrientCompl_e2} 
is obtained from the orientations~$\wh\fo_{\fM}$ and~$\fo_{X^{\phi}}$ (resp.~$\fo_{\fM}$ and~$\fo_{\cL}$)
via the second isomorphism in~\eref{SubIsom_e} (resp.~\eref{CompOrient_e0}), 
which is associated with the middle column in the first (resp.~second) 
exact square of Figure~\ref{RnodisomE_fig}.\\

\begin{figure}
$$\xymatrix{&0\ar[d]&0\ar[d]&0\ar[d]& &&0\ar[d]&0\ar[d]&0\ar[d]& \\
0\ar[r]&D\ar[r]\ar[d]&\wh\fM''\ar[r]\ar[d]&\wh\cM\ar[r]\ar[d]&0&
0\ar[r]&D\ar[r]\ar[d]&\wh\fM''\ar[r]\ar[d]&\wh\cM\ar[r]\ar[d]&0\\
0\ar[r]&\wh{D}\ar[r]\ar[d]&\wh\fM\ar[r]\ar[d]&\wh\cM\ar[r]\ar[d]&0&
0\ar[r]&D\ar[r]\ar[d]&\fM\ar[r]\ar[d]&\cM\ar[r]\ar[d]&0\\
0\ar[r]&\0_{X^{\phi}}\ar[r]\ar[d]&\0_{X^{\phi}}\ar[r]\ar[d]&0&&
&0\ar[r]&\0_{\cL}\ar[r]\ar[d]&\0_{\cL}\ar[r]\ar[d]&0\\
&0&0&&&&&0&0}$$
\caption{Exact squares of Fredholm operators inducing the orientations of Proposition~\ref{RnodisomE_prp}} 
\label{RnodisomE_fig}
\end{figure}

\noindent
Since $n\!\not\in\!2\Z$, the Exact Squares property implies that 
the right vertical isomorphism in~\eref{RnodisomE_e0a} with $k\!=\!n$ 
respects the orientations~$\wh\fo^{\otimes n}$, $\fo^{\otimes n}$, 
and~$\fo_{\C/\R}^{\otimes n}$ if and only if $(n\!-\!1)/2$ is even.
Thus, the right vertical isomorphism in~\eref{RnodisomE_e0a} with $k\!=\!n$ 
respects the orientations~$\wh\fo^{\otimes n}$, $\fo^{\otimes n}$, and~$\fo_{\C^n/\R^n}$.
Along with the first statement of Corollary~\ref{RnodisomE_crl}, 
this implies that the left vertical isomorphism in~\eref{RnodisomE_e0a},
which is associated with the left column in the first exact square of Figure~\ref{RnodisomE_fig},
respects the orientations~$\wh\fo_D$, $\fo_D$, and~$\fo_{X^{\phi}}$ if and only if 
\hbox{$g\!-\!1\!\in\!2\Z$}.
Since the dimension of~$X^{\phi}$ is odd and the dimension of~$\wh\cM$ is of the same parity
as~$g$,
the Exact Squares property in turn implies that the top row in this square
does not respect the orientations~$\fo_D$, $\fo_{in}''$, and~$\wh\fo_{\cM}$.\\

\noindent
By Lemma~\ref{unionorientEDM_lmm}, 
the left vertical isomorphism in~\eref{RnodisomEDM_e0},
which is associated with the right column in the second exact square of Figure~\ref{RnodisomE_fig},
respects the orientations~$\wh\fo_{\cM}$, $\fo_{\cM}$, and~$\fo_{\cL}$.
The Exact Squares property then implies that the top row in this square
also respects the orientations~$\fo_D$, $\fo_{pu}''$, and~$\wh\fo_{\cM}$.
Comparing the conclusions concerning the top rows in the two diagrams,
we obtain the first claim of the proposition.
The second claim is obtained in the same way using the second statement of 
Corollary~\ref{RnodisomE_crl} instead of the first.
\end{proof}

\noindent
With $(X,\om,\phi)$, $J$, $g$, $\ell$, and $B$ as in Theorem~\ref{main_thm} 
and $i\!\in\![\ell]$, let  
$$D_i^{\pm;\bu} \subset \ov\fM_{g,\ell+1}^{\phi;\bu}(B;J)$$
be the subspace of maps from domains $\Si$ so that one of the irreducible components~$\Si_i$ 
of~$\Si$
is~$\P^1$ which has precisely one node, carries only the marked points~$z_i^+$
and $z_{\ell+1}^{\pm}$, and is contracted by the~map.
The (virtual) normal bundle $\cN D_i^{\pm;\bu}$ of $D_i^{\pm;\bu}$ in 
$\ov\fM_{g,\ell+1}^{\phi;\bu}(B;J)$
is canonically isomorphic to the complex line bundle of the smoothings of 
the above node of~$\Si_i$, as in \cite[Lemma~5.2]{RealEnum}.
The complex orientation on~$\cN D_i^{\pm;\bu}$
and an orientation on~$\ov\fM_{g,\ell+1}^{\phi;\bu}(B;J)$ determine an orientation
on~$D_i^{\pm;\bu}$. 

\begin{crl}\label{Rnodisom_crl2}
Suppose $(X,\om,\phi)$, $J$, $(L,[\psi],\fs)$, $g$, $\ell$, and~$B$ are as in Theorem~\ref{main_thm}
and $i\!\in\![\ell]$.
The sign of the diffeomorphism
\BE{Rnodisom2_e} D_i^{\pm;\bu}\approx \ov\fM_{g,\ell}^{\phi;\bu}(B;J)\EE
induced by dropping the last conjugate pair of marked points is~$\pm1$
with respect  to the orientations on~$\ov\fM_{g,\ell}^{\phi;\bu}(X,B)$
and~$\ov\fM_{g,\ell+1}^{\phi;\bu}(X,B)$
induced by~$(L,[\psi],\fs)$ via~\eref{ROrientCompl_e2}
and the complex orientation on~$\cN D_i^{\pm;\bu}$.
If in addition $\wt\phi$ is a conjugation on~$L$ lifting~$\phi$,
the same property holds with~\eref{ROrientCompl_e2} replaced by~\eref{ROrientCompl_e3}.
\end{crl}

\begin{proof}
We denote by $\fD\ov\fM_{0,3}^{\phi;\pm}(X)\!\subset\!\ov\fM_{-1,3}^{\phi;\bu}(0;J)$
the topological component of maps from domains consisting of two copies of~$\P^1$
one of which carries the marked points $z_1^+,z_2^{\pm},z_3^+$.
Let
\BE{Rnodisom2_e3}\rho\!:\fD\ov\fM_{0,3}^{\phi;\pm}(X)\lra \ov\fM_{0,3}(0;J)\!\approx\!X\EE
be the diffeomorphism obtained by restricting each map to the copy of~$\P^1$
carrying the marked points $z_1^+,z_2^{\pm},z_3^+$ and
$$\vph_i^{\pm}\!:\ov\fM_{g,\ell}^{\phi;\bu}(B;J)\lra \fD\ov\fM_{0,3}^{\phi;\pm}(X)$$
be the map so that $\ev_i\!=\!\ev_3\!\circ\!\vph_i^{\pm}$.\\

\noindent
We denote~by $\wt{D}_i^{\pm;\bu}\!\subset\!\ov\fM_{g-2,\ell+3}^{\phi;\bu}(B;J)$
the image of the open embedding
\BE{Rnodisom2_e4}
\wt\io_{g;i}^{\pm}\!:
\ov\fM_{g,\ell}^{\phi;\bu}(B;J)\!\times\!\fD\ov\fM_{0,3}^{\phi;\pm}(X)
\lra \ov\fM_{g-2,\ell+3}^{\phi;\bu}(B;J)\EE
sending a pair of maps to their disjoint union and re-indexing the conjugate pairs 
of marked points on the two domains~as 
$$(1,\ldots,i\!-\!1,i,i\!+\!1,\ldots,\ell)\lra\big(1,\ldots,i\!-\!1,\ell\!+\!2,i\!+\!1,\ldots,\ell\big), 
\quad (1,2,3)\lra(i,\ell\!+\!1,\ell\!+\!3).$$
Define
\begin{equation*}\begin{split}
\wt{D}_i'^{\pm;\bu}=\wt{D}_i^{\pm;\bu}\!\cap\!
\ov\fM_{g-2,\ell+3}'^{\phi;\bu}(B;J)
&\equiv\big\{[\u]\!\in\!\wt{D}_i^{\pm;\bu}\!:\!\ev_{\ell+2}([\u])\!=\!\ev_{\ell+3}([\u])\!\big\}\\
&=\big\{\wt\io_{g;i}^{\pm}\!\circ\!(\id,\vph_i^{\pm})\!\big\}
\big(\ov\fM_{g,\ell}^{\phi;\bu}(B;J)\!\big).
\end{split}\end{equation*}
An orientation on $\ov\fM_{g-2,\ell+3}^{\phi;\bu}(B;J)$ induces an orientation 
on $\wt{D}_i'^{\pm;\bu}$ via the first isomorphism in~\eref{SubIsom_e}.\\

\noindent
The first node-identifying immersion in~\eref{Riogldfn_e2} restricts to a diffeomorphism
\BE{Rnodisom2_e5}q\!:\wt{D}_i'^{\pm;\bu}\lra D_i^{\pm;\bu}\,.\EE
The composition $q\!\circ\!\wt\io_{g;i}^{\pm}\!\circ\!(\id,\vph_i^{\pm})$ 
is the inverse of~\eref{Rnodisom2_e}.
By Proposition~\ref{unionorient_prp} with \hbox{$g_2\!=\!-1$} and \hbox{$B_2\!=\!0$},
the diffeomorphism~\eref{Rnodisom2_e4} is orientation-preserving with respect
to the orientations induced by~$(L,[\psi],\fs)$ via
either~\eref{ROrientCompl_e2} or~\eref{ROrientCompl_e3}, 
if $L$ admits a conjugation~$\wt\phi$ lifting~$\phi$ in the last case.
By the first claim of Proposition~\ref{Dblorient_prp} with $B\!=\!0$, 
the sign of the diffeomorphism~\eref{Rnodisom2_e3} is $\pm1$ with respect to
the orientation on the domain induced via~\eref{ROrientCompl_e2} and
the complex orientation on the target.
Thus, the sign of the diffeomorphism 
$$\wt\io_{g;i}^{\pm}\!\circ\!(\id,\vph_i^{\pm})\!:
\ov\fM_{g,\ell}^{\phi;\bu}(B;J)\lra \wt{D}_i'^{\pm;\bu}$$
is~$\pm1$ as well.
By the first claim of Proposition~\ref{Rnodisom_prp}, 
the diffeomorphism~\eref{Rnodisom2_e5} is orientation-preserving with respect to the orientations
induced by~$(L,[\psi],\fs)$ via~\eref{ROrientCompl_e2}
as described above.
This establishes the first claim.\\

\noindent
If in addition $\wt\phi$ is a conjugation on~$L$ lifting~$\phi$
and the moduli spaces of real maps are oriented via~\eref{ROrientCompl_e3}, 
then the sign of the diffeomorphism~\eref{Rnodisom2_e3} is $\mp1$
by the second claim of Proposition~\ref{Dblorient_prp} with $g\!=\!0$
and the diffeomorphism~\eref{Rnodisom2_e5} is orientation-reversing
by the second claim of Proposition~\ref{Rnodisom_prp}.
Thus, the sign of~\eref{Rnodisom2_e} is $\pm1$ in this case as well.
\end{proof}

\subsection{Spaces of rational maps}
\label{MSorient0_subs}

\noindent
Let $S^1\!\subset\!\P^1$ be the unit circle in $\C\!\subset\!\P^1$;
it is preserved by the anti-holomorphic involutions~$\tau$ and~$\eta$ defined
in~\eref{tauetadfn_e}.
For $c\!=\!\tau,\eta$, we orient the group~$G_c$ of holomorphic automorphisms 
of $(\P^1,c)$ via the exact sequence 
$$0\lra T_{\id}S^1\lra T_{\id}G_c\lra T_0\C\lra0$$
from the standard orientations of $S^1$ and $\C$.
For $\ell\!\ge\!2$, we denote by $\cM_{0,\ell}^{\tau}$ the uncompactified moduli space of 
equivalence classes of $(\P^1,\tau)$ with $\ell$~pairs of conjugate marked points. 
The Deligne-Mumford compactification~$\ov\cM_{0,2}^{\tau}$ of $\cM_{0,2}^{\tau}$
includes 3 additional stable real two-component nodal curves.
A diffeomorphism of~$\ov\cM_{0,2}^{\tau}$  with a closed interval is given~by
\BE{cMtaudiff_e}\ov\cM_{0,2}^{\tau}\lra [0,\i] , \quad
\big[(z_1^+,z_1^-),(z_2^+,z_2^-)\big]\lra 
\frac{z_2^+\!-\!z_1^-}{z_2^-\!-\!z_1^-}:\frac{z_2^+\!-\!z_1^+}{z_2^-\!-\!z_1^+}
=\frac{|1\!-\!z_1^+/z_2^-|^2}{|z_1^+\!-\!z_2^+|^2}\,.\EE
It takes the two-component curve with $z_1^+$ and $z_2^-$ on the same component to~$0$
and the two-component curve with $z_1^+$ and $z_2^+$ on the same component to~$\i$.\\

\noindent
For a real symplectic manifold $(X,\om,\phi)$,
$J\!\in\!\cJ_{\om}^{\phi}$,  $\ell\!\in\!\Z^{\ge0}$, and $c\!=\!\tau,\eta$, we denote~by 
$$\fM_{0,\ell}^{\phi,c}(B;J)\subset\ov\fM_{0,\ell}^{\phi}(B;J)$$
the subspace of equivalence classes of maps from~$(\P^1,c)$ and by  
$\fP_{\ell}^{\phi,c}(B;J)$ the space of real degree~$B$ $J$-holomorphic 
(parametrized) maps from $(\P^1,c)$ to~$(X,\phi)$ 
with $\ell$~conjugate pairs of (distinct) marked points.
The group~$G_c$ acts on this space~by
\begin{gather*}
G_c\!\times\!\fP_{\ell}^{\phi,c}(B;J)\lra \fP_{\ell}^{\phi,c}(B;J),\\
g\!\cdot\!\big(u,(z_1^+,z_1^-),\ldots,(z_{\ell}^+,z_{\ell}^-)\!\big)
=\big(u\!\circ\!g^{-1},(g(z_1^+),g(z_1^-)\!),\ldots,(g(z_{\ell}^+),g(z_{\ell}^-)\!)\!\big).
\end{gather*}
Thus,
\begin{alignat}{1}\label{fMcP_e}  
\fM_{0,\ell}^{\phi,c}(B;J)&= \fP_{\ell}^{\phi,c}(B;J)\big/G_c\,,\\
\label{fMcP_e2}
\La_{\R}^{\top}\big(T_{\u}\fP_{\ell}^{\phi,c}(B;J)\!\big)
&\approx \La_{\R}^{\top}\big(T_{\id}G_c\big)\!\otimes\!
\La_{\R}^{\top}\big(T_{[\u]}\fM_{0,0}^{\phi,c}(B;J)\big).  
\end{alignat}
Since $G_{\tau}$ has two topological components, an orientation on $\fP_{\ell}^{\phi,\tau}(B;J)$
might not descend to the quotient~\eref{fMcP_e} via~\eref{fMcP_e2}.\\

\noindent
The (virtual) tangent space of $\fP_0^{\phi,c}(B;J)$ at a $(\phi,c)$-real map~$u$ 
is the index (as a K-theory class) of the linearization of the $\dbar_J$-operator at~$u$.
A trivialization of the real bundle pair~$u^*(TX,\tnd\phi)$ over~$(S^1,c|_{S^1})$
determines an orientation on this index, or equivalently on $\det D_{(TX,\tnd\phi)}|_u$,
via a pinching construction as in the proofs of \cite[Proposition~8.1.4]{FOOO}
and \cite[Lemma~2.5]{Teh}.
If $c\!=\!\tau$, a continuously varying collection of such trivializations 
is determined by  a relative spin structure on the real vector bundle~$TX^{\phi}$
over $X^{\phi}\!\subset\!\!X$;
see the proof of \cite[Theorem~8.1.1]{FOOO} or \cite[Theorem~6.36]{Melissa}.\\

\noindent
An orientation on~$\fP_0^{\phi,c}(B;J)$ induces orientations on 
the spaces~$\fP_{\ell}^{\phi,c}(B;J)$ 
with $\ell\!\in\!\Z^+$ by assigning the fibers of the forgetful morphism
\BE{fPeeldfn_e}\fP_{\ell}^{\phi,c}(B;J)\lra\fP_{\ell-1}^{\phi,c}(B;J)\EE
the canonical complex orientation of the tangent space at 
the first marked point in the last conjugate pair of marked points of each marked map.
By the proof of \cite[Theorem~6.6]{Penka2} with~\eref{ROrientCompl_e3} 
replaced by~\eref{ROrientCompl_e2},
the relative spin structure on~$TX^{\phi}$ associated to
a real orientation on~$(X,\om,\phi)$ induces an orientation on~$\fP_{\ell}^{\phi,\tau}(B;J)$
that descends to this quotient and extends to the stable map compactification.
If $\ell\!\ge\!2$, we can take $(X,B)\!=\!(\pt,0)$ in~\eref{fMcP_e} and~\eref{fMcP_e2}
and obtain an orientation~on
\BE{cM02tau_e}\cM_{0,2}^{\tau}=\fM_{0,2}^{\id,\tau}(\pt,0)\,.\EE
This orientation agrees with the orientation on $\ov\cM_{0,2}^{\tau}$ determined by
the diffeomorphism~\eref{cMtaudiff_e}.\\

\noindent
The construction of the orientation on the real line orbi-bundle~\eref{RcMisom_e}
in the proof of \cite[Proposition~5.9]{RealGWsI} involves 
a somewhat arbitrary sign choice for the Serre duality isomorphism
in Step~\ref{SD_it} on page~\pageref{SD_it} of the present paper.
The (real) dimensions of its domain and target are $3(g\!-\!1)\!+\!2\ell$.
If $g\!\not\in\!2\Z$, this  choice thus has no effect on the homotopy class of this isomorphism
or the resulting orientation of the real line bundle~\eref{RcMisom_e}.
If $g\!\in\!2\Z$, changing this choice changes the resulting orientation of~\eref{RcMisom_e}
and  the orientation on the moduli space $\ov\fM_{g,\ell}^{\phi;\bu}(B;J)$ of real~maps
determined by a real orientation via~\eref{RfMisom_e}.
In light of \cite[Proposition~4.18]{RealGWsII}, the above sign choice is determined 
by a choice of orientation
of the real line bundle~\eref{RcMisom_e}  over~$\ov\cM_{0,2}^{\tau}$.
Since the real CR-operator~$\dbar_{(\P^1,\tau)}$ is surjective and its kernel consists
of the constant $\R$-valued functions on~$\P^1$,
an orientation on~\eref{RcMisom_e} over  $\ov\cM_{0,2}^{\tau}$ 
is determined by an orientation on~$\ov\cM_{0,2}^{\tau}$.
We orient~$\ov\cM_{0,2}^{\tau}$ by the diffeomorphism~\eref{cMtaudiff_e}.

\begin{prp}\label{RelSpinOrient_prp}
Suppose $(X,\om,\phi)$, $J$, $(L,[\psi],\fs)$, $\ell$, and $B$ are as in Theorem~\ref{main_thm}.
The orientations on $\fM_{0,\ell}^{\phi,\tau}(B;J)$ induced 
by the associated relative spin structure on the real vector bundle~$TX^{\phi}$
over $X^{\phi}\!\subset\!X$ and by~$(L,[\psi],\fs)$ via~\eref{RfMisom_e} and~\eref{ROrientCompl_e2} 
are the same if and only if $\lr{c_1(X,\om),B}\!\not\in\!4\Z$.
If in addition $\wt\phi$ is a conjugation on~$L$ lifting~$\phi$,
the former orientation and the orientation induced  by~$(L,[\psi],\fs)$ via~\eref{RfMisom_e} 
and~\eref{ROrientCompl_e3} with $\wt\si\!=\!\wt\tau$
are the same if and only if \hbox{$\lr{c_1(X,\om),B}\!\cong\!2,4$} \!\!\!$\mod8$.
If $L^{\wt\phi}\!\lra\!X^{\phi}$ is orientable, then
the orientations on $\fM_{0,\ell}^{\phi,\tau}(B;J)$ induced by 
the associated spin structure on~$TX^{\phi}$ and the latter orientation are opposite.
\end{prp}

\begin{proof}
Since the fibers of the forgetful morphism~\eref{fPeeldfn_e} are canonically oriented,
it is sufficient to establish the claims for $\ell\!=\!2$.
Orienting $\fM_{0,2}^{\phi,\tau}(B;J)$ via~\eref{fMcP_e2} with $\ell\!=\!2$
and~\eref{fPeeldfn_e} with $\ell\!=\!1,2$ from an orientation on~$\det D_{(TX,\tnd\phi)}|_u$
and the complex orientation of~$\P^1$ is equivalent to orienting this moduli space
via the isomorphism
$$\La_{\R}^{\top}\big(T\fM_{0,2}^{\phi,\tau}(B;J)\!\big)
\approx \ff^*\!\big(\La^{\top}_{\R}(T\cM_{0,2}^{\tau})\!\big)
\!\otimes\!\big(\!\det D_{(TX,\nd\phi)}\big)$$
from the chosen orientation of~$\det D_{(TX,\tnd\phi)}|_u$ and 
the orientation of~$\cM_{0,2}^{\tau}$ determined by~\eref{fMcP_e} and~\eref{fMcP_e2}
with $(X,B)\!=\!(\pt,0)$.
Since the dimension of~$\cM_{0,2}^{\tau}$ and the index of~$D_{(TX,\nd\phi)}$ are~odd,
the orientation on~$\fM_{0,2}^{\phi,\tau}(B;J)$
induced in this way is the opposite of the orientation induced
via the canonical isomorphism~\eref{RfMisom_e} with $(g,\ell)\!=\!(0,2)$
from the same orientations on~$\det D_{(TX,\tnd\phi)}|_u$ and~$\cM_{0,2}^{\tau}$.
By the sentence below~\eref{cM02tau_e} and the following paragraph,
the orientation on~$\cM_{0,2}^{\tau}$ used in all of the relevant approaches
to orienting~$\fM_{0,\ell}^{\phi,\tau}(B;J)$ is the same.
The relevant orientations on~$\det D_{(TX,\nd\phi)}$
are compared by Corollary~\ref{RelSpinOrient_crl} with~$(V,\vph)$ replaced 
by~$u^*(TX,\nd\phi)$. 
Taking into account that 
$$\blr{c_1(TX,\om),B}=2\blr{c_1(L),B}$$
if $\fM_{0,\ell}^{\phi,\tau}(B;J)\!\neq\!\eset$, we obtain the claims.
\end{proof}

\begin{crl}\label{RelSpinOrient_crl2}
Suppose $(X,\om,\phi)$, $J$, $(L,[\psi],\fs)$, and $\ell$ are as in Theorem~\ref{main_thm}.
The natural isomorphism
\BE{RelSpinOrient2_e}\ov\fM_{0,\ell}^{\phi}(0;J)\approx \R\ov\cM_{0,\ell}\!\times\!X^{\phi}\EE
is orientation-reversing with respect to the orientation on 
the left-hand side induced by~$(L,[\psi],\fs)$ via~\eref{RfMisom_e} and~\eref{ROrientCompl_e2},
the orientation on~$\R\ov\cM_{0,\ell}$ induced via the diffeomorphism~\eref{cMtaudiff_e},
and the orientation on~$X^{\phi}$ determined by~$(L,[\psi],\fs)$.
If in addition $\wt\phi$ is a conjugation on~$L$ lifting~$\phi$,
the same property holds with~\eref{ROrientCompl_e2} replaced by~\eref{ROrientCompl_e3}.  
\end{crl}

\begin{rmk}\label{RelSpinOrient_rmk}
The orientation on~$\ov\cM_{0,2}^{\tau}$ determined by the diffeomorphism~\eref{cMtaudiff_e}
is as above \cite[Lemma~5.4]{Penka2} and in \cite[Section~12.3]{SpinPin}.
It is the opposite of the orientation taken 
in \cite[Section~3]{RealEnum} and \cite[Section~1.4]{RealGWsII},
which is also used in \cite{RealGWsIII} and~\cite{RealGWvsEnum}.
The order of the two factors on the right-hand side of~\eref{fMcP_e2} 
is opposite from \cite[(1.14)]{RealGWsII}.
This does not impact the orientation on $\fM_{0,\ell}^{\phi,\tau}(B;J)$
induced from an orientation of~$\fP_{\ell}^{\phi}(B;J)$
whenever $(X,\om,\phi)$ and $(L,[\psi],\fs)$ are as in Theorem~\ref{main_thm}
because the dimension~\eref{RfMdim_e} of~$\fM_{0,\ell}^{\phi,\tau}(B;J)$ is then even.
However, this reverses the induced orientation on the moduli spaces
with $(X,B)\!=\!(\pt,0)$.
Thus, the property stated after~\eref{cM02tau_e} with the present conventions
is equivalent to this property with the conventions in~\cite{RealGWsII}.
\end{rmk}
  
\begin{rmk}\label{OrientComp_rmk}	
Let $(X,\om,\phi)$, $J$, $(L,[\psi],\fs)$, $g,\ell$, and $B$ be as in Theorem~\ref{main_thm}.	
If in addition $\wt\phi$ is a conjugation on~$L$ lifting~$\phi$,
the orientations on $\ov\fM_{g,\ell}^{\phi;\bu}(B;J)$ induced via~\eref{RfMisom_e}
along with either~\eref{ROrientCompl_e2} with the projection orientation 
on~$\det D_{(L^*\oplus\si^*\ov{L}^*,\si_{L^*}^{\oplus})}$
or~\eref{ROrientCompl_e3} with the canonical orientation on~$\det D_{2(L^*,\wt\si^*)}$
differ by~$(-1)^{\ve}$ with 
$$\ve=\frac12\big(g\!+\!\lr{c_1(X,\om),B}/2\big)\big(g\!-\!1+\!\lr{c_1(X,\om),B}/2\big);$$
this follows from Lemma~\ref{CvsCanorient_lmm} 
with $\rk_{\C}L\!=\!1$ and $\deg L$ replaced by $-\lr{c_1(X,\om),B}/2$.
The terms $\det(2\dbar_{(\Si,\si)})$ in~\eref{ROrientCompl_e2} and~\eref{ROrientCompl_e3}
and
$$\big(\!\det\dbar_{\C}\big)^{\otimes(n+1)}\approx\big(\!\det(2\dbar_{\C})\!\big)^{\otimes(n+1)/2}$$
in~\eref{RMBvfc_e}
could be oriented either via the canonical orientation (as was done for the purposes
of Propositions~\ref{unionorient_prp}-\ref{Rnodisom_prp}, \ref{RnodisomE_prp}, 
and~\ref{RelSpinOrient_prp} and Corollaries~\ref{Rnodisom_crl2} and~\ref{RelSpinOrient_crl2})
or via the projection orientation.
The resulting orientations on $\ov\fM_{g,\ell}^{\phi;\bu}(B;J)$ would differ
by~$(-1)^{\ve}$ with
$$\ve=\frac{n\!-\!1}{2}\!\cdot\!\frac{g(g\!-\!1)}{2}\,;$$
this follows from Lemma~\ref{CvsCanorient_lmm} with $L\!=\!\!\Si\!\times\!\C$. 
This would make the isomorphism~\eref{unionorient_e} orientation-preserving in the first claim
of Proposition~\ref{unionorient_prp}, 
add $g\!-\!1$ to the parity conditions in Propositions~\ref{Dblorient_prp}
and~\ref{RnodisomE_prp}, and switch the conditions in the two claims of 
Proposition~\ref{Dblorient_prp}, but would not impact on
Corollary~\ref{Rnodisom_crl2} or~\ref{RelSpinOrient_crl2}.
The Step~\ref{ROrientStab_it} on page~\pageref{ROrientStab_it} is 
carried out in~\cite{RealGWsI} via~\eref{ROrientCompl_e3} with $L^*\!=\!T\Si$
and the canonical orientations on~$\det D_{2(L^*,\wt\si^*)}$ and~$\det(2\dbar_{(\Si,\si)})$.
Using the projection orientations on these factors instead would change the induced
orientations on the line bundle~\eref{RcMisom_e} and 
on the moduli space~$\ov\fM_{g,\ell}^{\phi;\bu}(B;J)$ by~$(-1)^{g-1}$.
Since this sign is reversed by changing the preferred choice of the orientation
of~$\cM_{0,2}^{\tau}$, as discussed above, this change is immaterial. 
\end{rmk}

\section{Updates on past work and corrections}
\label{updates_sec}

\subsection{Real-orientable symplectic manifolds, {\cite[Sections 1.1,2.1,2.2]{RealGWsIII}}}
\label{RealGWth_subs}
	
\noindent
Sufficient conditions for a real symplectic manifold $(X,\om,\phi)$ to admit
a real orientation in the sense of \cite[Definition~1.2]{RealGWsI},
i.e.~a triple~$(L,[\psi],\fs)$ as in Definition~\ref{realorient_dfn0} in the present paper
with \hbox{$(V,\vph)\!=\!(TX,\nd\phi)$} and a conjugation~$\wt\phi$ on~$L$ lifting~$\phi$, 
are provided in \cite[Section~1.1]{RealGWsIII}.
As indicated below, some real symplectic manifolds admit real orientations in
the weaker sense of Definition~\ref{realorient_dfn0} without a conjugation~$\wt\phi$ on~$L$.
The next proposition is a refined analogue of \cite[Proposition~1.2]{RealGWsIII} 
for the real orientations of Definition~\ref{realorient_dfn0}.

\begin{prp}\label{CYorient_prp}
Let $(X,\om,\phi)$ be a connected real symplectic manifold such~that
$$H_1(X;\Q)=0 \qquad\hbox{and}\qquad  H^2(X^{\phi};\Z_2)=0.$$
The collection of the real orientations~$(L,[\psi],\fs)$ 
on~$(X,\om,\phi)$ can be identified~with
$$\big\{\mu\!\in\!H^2(X;\Z)\!:
c_1(X,\om)\!=\!\mu\!-\!\phi^*\mu\big\}\!\times\!H^1(X^{\phi};\Z_2)\!\times\!\Z_2
\subset H^2(X;\Z)\!\times\!H^1(X^{\phi};\Z_2)\!\times\!\Z_2.$$
\end{prp}	
	
\begin{proof}
The complex line bundles $L\!\lra\!X$ such that 
\BE{CYorient_e3}\La_{\C}^{\top}(TX)\approx L\!\otimes_{\C}\!\phi^*\ov{L}\EE
correspond via~$c_1$ to the elements $\mu\!\in\!H^2(X;\Z)$ such that 
\hbox{$c_1(X,\om)\!=\!\mu\!-\!\phi^*\mu$}.
Since 
$$w_2(TX^{\phi}\!\oplus\!L^*|_{X^{\phi}})=0\in H^2(X^{\phi};\Z_2) ,$$
the set of spin structures on $TX^{\phi}\!\oplus\!L^*|_{X^{\phi}}$
for each complex line bundle $L\!\lra\!X$ can be identified with~$H^1(X^{\phi};\Z_2)$;
see the SpinPin~1 and~2 properties in \cite[Section~1.2]{SpinPin}.
By \cite[Lemma~2.7]{Teh}, \hbox{$H_1(X;\Q)\!=\!0$} and \eref{CYorient_e3} 
imply that the real line bundle pairs $\La_{\C}^{\top}(TX,\nd\phi)$ and
$(L\!\otimes_{\C}\!\phi^*\ov{L},\phi_L^{\otimes})$ are isomorphic.\\

\noindent
It remains to show that a real line bundle pair $(L,\wt\phi)$ over~$(X,\phi)$
has precisely two homotopy classes of automorphisms if $H_1(X;\Q)\!=\!0$.
Any such homotopy class contains an automorphism
\BE{CYorient_e7}
(L,\wt\phi)\lra(L,\wt\phi) \qquad v\lra \rho(x)v ~~\forall\,v\!\in\!L_x,~x\!\in\!X,\EE
for some continuous map $\rho\!:X\!\lra\!S^1$ such that $\rho(x)\rho(\phi(x)\!)\!=\!1$ 
for all $x\!\in\!X$.
Since \hbox{$H_1(X;\Q)\!=\!0$}, the homomorphism \hbox{$\rho_*\!:\pi_1(X)\!\lra\!\pi_1(S^1)$} is trivial.
Thus, $\rho$ lifts over the covering~map
$$\R\lra S^1, \qquad  \th\lra \ne^{2\pi\fI\th},$$
to a continuous map $\wt\rho\!:X\!\lra\!S^1$ such that 
$\wt\rho(x)\!+\!\wt\rho(\phi(x)\!)\!\in\!\Z$ for all $x\!\in\!X$.
By adding an integer to~$\wt\rho$, we can assume that
$$\wt\rho(x)\!+\!\wt\rho\big(\phi(x)\!\big)=\eps~~\forall\,x\!\in\!X$$
for some $\eps\!\in\!\{0,1\}$.
The family of automorphisms
$$(L,\wt\phi)\lra(L,\wt\phi), \quad \tau\!\times\!v\lra 
\begin{cases}\ne^{2\pi\fI\tau\wt\rho(x)}v,&\hbox{if}~\eps\!=\!0;\\
-\ne^{\pi\fI\tau(2\wt\rho(x)-1)}v,&\hbox{if}~\eps\!=\!1;
\end{cases} \qquad\forall\,v\!\in\!L_x,~x\!\in\!X,$$
is then a homotopy from $\pm\id_L$ to the isomorphism~\eref{CYorient_e7}.\\

\noindent
On the other hand, the automorphisms~$\id_L$ and~$-\id_L$ of~$(L,\wt\phi)$
are not homotopic.
This is immediate if $X^{\phi}\!\neq\!\eset$ because 
$\rho(x)\!=\!\pm1$ for every $x\!\in\!X^{\phi}$.
If $X^{\phi}\!=\!\eset$, we can choose a continuous map 
$$\al\!:S^1\lra X \qquad\hbox{s.t.}\quad \phi\!\circ\!\al=\al\!\circ\!\fa,$$
where $\fa\!:S^1\!\lra\!S^1$ is the antipodal involution.
By the proof of \cite[Lemma~2.4]{Teh}, the automorphisms~$\id_{\al^*L}$ and~$-\id_{\al^*L}$
of the real bundle pair $\al^*(L,\wt\phi)$ over~$(S^1,\fa)$ are not homotopic.
\end{proof}

\noindent
If $k\!\in\!\Z^{\ge0}$, $\a\!\equiv\!(a_1,\ldots,a_k)\!\in\!(\Z^+)^k$,
and $X_{n;\a}\!\subset\!\P^{n-1}$ (resp.~$X_{2m;\a}\!\subset\!\P^{2m-1}$)
is a complete intersection of multi-degree~$\a$ preserved by~$\tau_n$ (resp.~$\eta_{2m}$),  
then $\tau_{n;\a}\!\equiv\!\tau_n|_{X_{n;\a}}$ (resp.~$\eta_{2m;\a}\!\equiv\!\eta_{2m}|_{X_{2m;\a}}$)
is an anti-symplectic involution on $X_{n;\a}$ (resp.~$X_{2m;\a}$)
with respect to the symplectic form $\om_{n;\a}\!=\!\om_n|_{X_{n;\a}}$
(resp.~$\om_{2m;\a}\!=\!\om_{2m}|_{X_{2m;\a}}$).
Sufficient conditions on $(n,\a)$ and $(m,\a)$ 
for $(X_{n;\a},\om_{n;\a},\tau_{n;\a})$ and $(X_{2m;\a},\om_{2m;\a},\eta_{2m;\a})$
to admit a real orientation $(L,[\psi],\fs)$  with a conjugation~$\wt\phi$ on~$L$ lifting
\hbox{$\phi\!=\!\tau_{n;\a},\eta_{2m;\a}$}
are provided by \cite[Proposition~1.4]{RealGWsIII}.
If these conditions are satisfied, the real symplectic manifold under consideration
carries a natural real orientation in the sense of \cite[Definition~1.2]{RealGWsI} 
with
$$L=\big(\cO_{\P^n}\big((n\!-\!a_1\!-\!\ldots\!-\!a_k)/2\big)\!\big)\big|_{X_{n;\a}}\lra
X_{n;\a}\subset\P^{n-1},$$
where $n\!\equiv\!2m$ in the case of $(X_{2m;\a},\om_{2m;\a},\eta_{2m;\a})$.\\

\noindent
Dropping the requirement of the existence of a conjugation~$\wt\phi$ on~$L$ lifting~$\eta_{2m;\a}$
immediately leads to weakening the condition
\BE{etanacond_e}a_1\!+\!\ldots\!+\!a_k\equiv 2m \mod4\EE
in \cite[Proposition~1.4(2)]{RealGWsIII} to the mod~2 equivalence
(a complex line bundle over~$\P^{2m-1}$ needs to be of an even degree to admit 
a conjugation~$\wt\phi$ lifting~$\eta_{2m}$).
However, the condition \eref{etanacond_e} holds for any $\eta_{2m}$-real complete intersection
$X_{2m;\a}\!\subset\!\P^{2m-1}$  of multi-degree~$\a$
because the odd components~$a_i$ of~$\a$ come in pairs
if such a complete intersection exists.
Thus, every $\eta_{2m}$-real complete intersection $X_{2m;\a}\!\subset\!\P^{2m-1}$ 
admits a natural real orientation in the sense of Definition~\ref{realorient_dfn0}.
On the other hand, no difference is made in the case of $(X_{n;\a},\om_{n;\a},\tau_{n;\a})$,
as every complex line bundle over~$\P^{n-1}$ admits a conjugation~$\wt\phi$ lifting~$\tau_n$.\\
	
\noindent
Suppose $(\Si,\si)$ is a smooth compact connected symmetric surface of genus~$g$ and 
$(X,\phi)$ is a real bundle pair over~$(\Si,\si)$ of degree~$d$.
By a standard construction, any precompact neighborhood of $\Si\!\subset\!X$ 
admits a symplectic form~$\om$ so that $\phi^*\om\!=\!-\om$ and 
$\om|_{T\Si}$ is a volume form on~$\Si$; see \cite[(2.9)]{SympDivConf}, for example.
A complex line bundle~$L$ over~$X$ satisfying~\eref{CYorient_e3} exists 
if and only~if \hbox{$d\!\in\!2\Z$}.
By \cite[Propositions~4.1,4.2]{BHH}, $(X,\phi)$ admits a real orientation~$(L,[\psi],\fs)$
as in Definition~\ref{realorient_dfn0} in such a case,
as the first equation in~\eref{realorient_e0} imposes no restriction on~$L$.
For any such real orientation, $L$ admits a conjugation~$\wt\phi$ lifting~$\phi$
if and only~if either $\Si^{\si}\!\neq\!\eset$ or $g\!+\!d/2\!\not\in\!2\Z$.
In particular, every smooth compact connected genus~$g$ symmetric surface $(\Si,\si)$ 
admits a real orientation $(L,[\psi],\fs)$ as in Definition~\ref{realorient_dfn0},
but $(\Si,\si)$ does not admit a real orientation $(L,[\psi],\fs)$ with a conjugation~$\wt\si$
lifting~$\si$ if (and only~if) $\Si^{\si}\!=\!\eset$ and $g\!\in\!2\Z$.
This is the motivation behind the introduction of the weaker notion of real orientation 
in~\cite{GI}.

\subsection{Associated relative spin structures, {\cite[Corollary~3.8]{RealGWsII}}}
\label{ARSS_sub}

\noindent
The second and third claims of Corollary~\ref{RelSpinOrient_crl} are special cases
of the base cases of \cite[Corollary~3.8]{RealGWsII}.
Unfortunately, the derivation of the full statement of \cite[Corollary~3.8]{RealGWsII}
from its base cases is based on the implicit assumption that the isomorphisms as
in~\eref{reldetdfn_e0} induced in a manner similar to~\eref{ROrientCompl_e3}
commute with disjoint unions as in~\eref{unionorient_e0}.
This is indeed the case in the setting relevant to the derivation of the full statement
of \cite[Corollary~3.8(2)]{RealGWsII} from its base cases in~\cite{RealGWsII};
thus, this statement is correct.
However, this commutativity may not hold in the setting relevant to the derivation 
of the full statement of \cite[Corollary~3.8(1)]{RealGWsII};
this statement needs to be corrected as described below.
Fortunately, \cite[Corollary~3.8]{RealGWsII} beyond its base cases has not been elsewhere~yet
(to the best of our knowledge).\\

\noindent
Let $D_{(V,\vph)}$ be a real CR-operator on a rank~$n$ real bundle pair~$(V,\vph)$
over a smooth symmetric surface~$(\Si,\si)$ and $(L,\wt\phi)$ be a rank~1 real 
bundle pair over~$(\Si,\si)$.
We assume that the real vector bundle~$V^{\vph}$ over the fixed locus~$\Si^{\si}$
of~$\si$ is orientable.
The subject of \cite[Corollary~3.8]{RealGWsII} is a comparison between two homotopy classes
of isomorphisms 
\BE{ARSS_e3}\det D_{(V,\vph)}\approx\det\!\big(n\dbar_{(\Si,\si)}\big),\EE
determined by trivializations of $(L,\wt\phi)$ over cutting circles for $(\Si,\si)$;
these homotopy classes correspond to orientations of 
$(\det D_{(V,\vph)})\!\otimes\!\big(\!\det(n\dbar_{(\Si,\si)})\!\big)$.\\

\noindent
In \cite[Corollary~3.8(1)]{RealGWsII}, it is assumed that~$\Si^{\si}$ separates~$\Si$
and has $m$~topological components, $\Si^{\si}_1,\ldots,\Si^{\si}_m$.
The symmetric surface~$(\Si,\si)$ can then be obtained by doubling a bordered surface~$\Si^b$
along its boundary as in~\cite[(1.6)]{GZ1}.
The boundary~$\prt\Si^b$ of~$\Si^b$ must have $m$~topological components as well.
For each $i\!=\!1,\ldots,m$, let
$$\ve_i(L)=\begin{cases}0,&\hbox{if}~w_1(L^{\wt\phi})|_{\Si^{\si}_i}\!=\!0,\\
1,&\hbox{if}~w_1(L^{\wt\phi})|_{\Si^{\si}_i}\!\neq\!0.\end{cases}$$
We define $m_1(L)\!\in\![m]$ to be the sum of these numbers~$\ve_i(L)$. 
The necessary and sufficient condition for the two resulting homotopy classes of
isomorphism~\eref{ARSS_e3} to be the same in \cite[Corollary~3.8(1)]{RealGWsII} 
should be replaced~by
\BE{ARSS_e5}\deg L-mm_1(L)-\binom{m_1(L)}{2}\in4\Z\,,\EE
for the reasons below.\\

\noindent
By the usual pinching construction, as the proof of \cite[Lemma~6.37]{Melissa}, 
a trivialization~$\psi_{V\oplus 2L}$ of the real vector bundle 
$V^{\vph}\!\oplus\!2L^{\wt\phi}$ over~$\Si^{\si}$ and 
a choice of a half-surface~$\Si^b$ for~$(\Si,\si)$
determine a homotopy class of isomorphisms
\BE{ARSS_e7}\det D_{(V\oplus2L,\vph\oplus2\wt\phi)}\approx 
\det\!\big(\!(n\!+\!2)\dbar_{(\Si,\si)}\big)\,.\EE
Similarly to~\eref{ROrientCompl_e3}, this homotopy class determines 
a homotopy class of isomorphisms~\eref{ARSS_e3}, called 
the \sf{stabilization orientation induced by~$\psi_{V\oplus 2L}$} in \cite[Section~3.2]{RealGWsII}.
Along with the isomorphism~$\Phi_L$ in~\eref{GIgen_e2} composed with the projection of the target
to the first factor, $\psi_{V\oplus 2L}$ also determines a relative spin structure on 
the real vector bundle~$V^{\vph}$ over~$\Si^{\si}$
and thus another homotopy class of isomorphisms~\eref{ARSS_e3},
called the \sf{associated relative spin} (or \sf{ARS}) \sf{orientation
induced by~$\psi_{V\oplus 2L}$} in \cite[Section~3.2]{RealGWsII}.
These two homotopy classes are the same if and only if \eref{ARSS_e5} holds.\\

\noindent
The argument in the proof of \cite[Corollary~3.8]{RealGWsII} starts 
by pinching off two circles near each component~$\Si^{\si}_i$ of~$\Si$,
one circle on each side of~$\Si^{\si}_i$, to form a nodal symmetric surface~$(\Si_0,\si_0)$.
The domain of its normalization,
\BE{ARSS_e8}q\!:(\wt\Si_0,\wt\si_0)\lra(\Si_0,\si_0),\EE 
consists of a copy $(\P_i^1,\tau_i)$
of~$(\P^1,\tau)$ for each component~$\Si^{\si}_i$ of~$\Si$ and 
a doublet $(\Si_0^+\!\sqcup\!\Si_0^-,\si')$ as in~\eref{Dbldfn_e}.
We choose~$\Si_0^+$ so that it is contained in the image of~$\Si^b$ under
the pinching procedure.
Let $x_i^+\!\in\!\Si_0$ be the node separating~$\P_i^1$ from~$\Si_0^+$.
We deform the real bundle pairs~$(V,\vph)$ and~$(L,\wt\phi)$ to 
real bundle pairs~$(V_0,\vph_0)$ and~$(L_0,\wt\phi_0)$ over~$(\Si_0,\si_0)$ 
so that~$\deg L_0|_{\Si_0^+}\!\in\!2\Z$ and~thus
\BE{ARSS_e9} \deg L-\sum_{i=1}^m\deg L_0|_{\P_i^1}\in 4\Z.\EE
Let $(V_i,\vph_i)\!\equiv\!(V_0,\vph_0)|_{(\P_i^1,\tau_i)}$ and 
$(L_i,\wt\phi_i)\!\equiv\!(L_0,\wt\phi_0)|_{(\P_i^1,\tau_i)}$.\\

\noindent
A real CR-operator~$D_{(V_0,\vph_0)}$ on~$(V_0,\vph_0)$ lifts to
a real CR-operator~$q^*D_{(V_0,\vph_0)}$ on~$q^*(V_0,\vph_0)$ as above Lemma~\ref{Rnodisom_lmm}.
The latter is isomorphic to the direct sum of real CR-operators~$D_{(V_i,\vph_i)}$
on~$(V_i,\vph_i)$ and a CR-operator~$D_{V_0^+}$ on~$V_0|_{\Si_0^+}$.
There is a canonical homotopy class of isomorphisms
$$ \det D_{(V,\vph)}\approx \det D_{(V_0,\vph_0)}\,.$$
Similarly to~\eref{Runionorient_e0} and~\eref{Rnodisom_e}, there is also an isomorphism
$$\big(\!\det D_{(V_0,\vph_0)}\big)\!\otimes\!\bigotimes_{i=1}^m\La_{\R}^{\top}\big(V_0|_{x_i^+}\big)
\approx \big(\!\det D_{V_0^+}\big)\!\otimes\!\bigotimes_{i=1}^m
\big(\!\det D_{(V_i,\vph_i)}\big)\,.$$
Combining~the last two isomorphisms with the canonical complex orientations on
$\La_{\R}^{\top}(V_0|_{x_i^+})$ and~$\det D_{V_0^+}$, we obtain a homotopy class of isomorphisms
\BE{ARSS_e11} \det D_{(V,\vph)}
\approx \big(\!\det D_{(V_1,\vph_1)}\big)\!\otimes\!\ldots\! 
\otimes\!\big(\!\det D_{(V_m,\vph_m)}\big).\EE

\vspace{.18in}

\noindent 
The trivialization $\psi_{V\oplus 2L}$ determines a relative spin structure on 
the real vector bundle $V_i^{\vph_i}\!\oplus\!2L_i^{\wt\phi_i}$ over 
the fixed locus $\Si^{\si}_i\!\subset\!\P^1_i$ of~$\tau_i$ for each $i\!\in\![m]$ and thus 
a homotopy class of isomorphisms 
\BE{ARSS_e15}  
\det D_{(V_i,\vph_i)}\approx \det\!\big(n\dbar_{(\P^1_i,\tau_i)}\big).\EE
By the assumption that $\deg L_0|_{\Si_0^+}\!\in\!2\Z$, the diagram 
\BE{ARSS_e17}\begin{split}
\xymatrix{\det D_{(V,\vph)} \ar[rr]^>>>>>>>>>>{\eref{ARSS_e11}}_>>>>>>>>>>{\approx}
\ar[d]_{\eref{ARSS_e3}}^{\approx}&&
\big(\!\det D_{(V_1,\vph_1)}\big)\!\otimes\!\ldots\!
\otimes\!\big(\!\det D_{(V_m,\vph_m)}\big)\ar[d]^{\eref{ARSS_e3}}_{\approx}\\
\det\!\big(n\dbar_{(\Si,\si)}\big) \ar[rr]^>>>>>>>>>>{\eref{ARSS_e11}}_>>>>>>>>>>{\approx}&& 
\det\!\big(n\dbar_{(\P^1_1,\tau_1)}\big)\!\otimes\!\ldots\!
\otimes\!\det\!\big(n\dbar_{(\P^1_m,\tau_m)}\big)}
\end{split}\EE 
with the vertical isomorphisms corresponding to
the ARS orientations induced by~$\psi_{V\oplus 2L}$ commutes up to homotopy.\\

\noindent
The analogue of the diagram~\eref{ARSS_e17} for the isomorphisms~\eref{ARSS_e7} 
induced by~$\psi_{V\oplus 2L}$ also commutes.
Similarly to the reasoning in the proof of Corollary~\ref{unionorient_crl}, 
Lemma~\ref{unionorient_lmm} then implies that the diagram~\eref{ARSS_e17}
with the vertical isomorphisms corresponding to
the stabilization orientations induced by~$\psi_{V\oplus 2L}$ commutes up to homotopy
if and only if the number
\BE{ARSS_e18}\sum_{\begin{subarray}{c}i,j\in[m]\\ i<j\end{subarray}}
\!\!\!\!\big(1\!+\!(\deg L_0|_{\P_i^1}\!-\!1)(\deg L_0|_{\P_j^1}\!-\!1)\!\big)
\cong 
\sum_{\begin{subarray}{c}i,j\in[m]\\ i<j\end{subarray}}
\!\!\!\!\big(\ve_i(L)\!+\!\ve_j(L)\!+\!\ve_i(L)\ve_j(L)\!\big)\mod2\EE
is even.
By Lemma~3.4 and Corollary~3.6 in~\cite{RealGWsII}, the homotopy classes of 
isomorphisms~\eref{ARSS_e15}
corresponding to the ARS and stabilization orientations induced by~$\psi_{V\oplus 2L}$
are the same if and only $\deg L_0|_{\P_i^1}\!-\!\ve_i(L)\!\in\!4\Z$.
Combining this with the end of the previous paragraph and~\eref{ARSS_e9},
we conclude that the homotopy classes of left vertical isomorphisms in~\eref{ARSS_e17}
corresponding to the ARS and stabilization orientations induced by~$\psi_{V\oplus 2L}$
are the same if and only if~\eref{ARSS_e5} holds.\\

\noindent
In \cite[Corollary~3.8(2)]{RealGWsII}, the fixed locus~$\Si^{\si}$ of~$\si$ is not assumed to be
separating, but the real line bundle~$L^{\wt\phi}$ is assumed to be orientable.
There is again a choice of a half-surface~$\Si^b$ doubling to~$(\Si,\si)$;
its boundary consists of~$\Si^{\si}$ and so-called \sf{cross-caps},
i.e.~circles with fixed-point free involutions.
In addition to a trivialization~$\psi_{V\oplus 2L}$ of 
the real vector bundle~$V^{\vph}\!\oplus\!2L^{\wt\phi}$ over~$\Si^{\si}$,
there is also a trivialization~$\psi'_{V\oplus 2L}$ of the real bundle pair 
$(V\!\oplus\!2L,\vph\!\oplus\!2\wt\phi)$ over the union~$\prt_1^c\Si^b$ of the crosscaps.
By the pinching construction in the proof of \cite[Theorem~1.1]{GZ1}, 
$\psi_{V\oplus 2L}$, $\psi'_{V\oplus 2L}$, and~$\Si^b$ determine a homotopy class 
of isomorphisms~\eref{ARSS_e7} and thus a homotopy class of isomorphisms~\eref{ARSS_e3}, 
called  the \sf{stabilization orientation induced by~$\psi_{V\oplus 2L}$ and~$\psi'_{V\oplus 2L}$ } 
in \cite[Section~3.2]{RealGWsII}.
Since the real line bundle $2L^{\wt\phi}$ is orientable, 
the trivialization~$\psi_{V\oplus 2L}$ and the canonical homotopy class of trivializations
of~$2L^{\wt\phi}$ determine a homotopy class of trivializations 
of the real vector bundle~$V^{\vph}$ over~$\Si^{\si}$.
The trivialization~$\psi'_{V\oplus 2L}$ and the canonical homotopy class of trivializations
of~$2(L,\wt\phi)$ over~$\prt_1^c\Si^b$ similarly determine a homotopy class of trivializations
of~$(V,\vph)$ over~$\prt_1^c\Si^b$.
The resulting trivializations of~$V^{\vph}$ and~$(V,\vph)|_{\prt_1^c\Si^b}$ yield
another homotopy class of isomorphisms~\eref{ARSS_e3},
called the \sf{associated spin} (or \sf{AS}) \sf{orientation
induced by~$\psi_{V\oplus 2L}$ and~$\psi'_{V\oplus 2L}$} in \cite[Section~3.2]{RealGWsII}.
These two homotopy classes of isomorphisms~\eref{ARSS_e3} are in fact the same.\\

\noindent
The proof of the last claim is similar to the proof of \cite[Corollary~3.8(1)]{RealGWsII}, 
as corrected above.
In this case, we also pinch off a pair of circles near each topological component 
of~$\prt_1^c\Si^b$ that are interchanged by~$\si$ to form a nodal symmetric surface~$(\Si_0,\si_0)$.
The domain of its normalization~\eref{ARSS_e8} now also contains
a copy~$(\P_i^1,\eta_i)$ of~$(\P^1,\eta)$ for each topological component of~$\prt_1^c\Si^b$.
The analogue of~\eref{ARSS_e17} now also includes isomorphisms
\BE{ARSS_e25}  
\det D_{(V_i,\vph_i)}\approx \det\!\big(n\dbar_{(\P^1_i,\eta_i)}\big)\EE
on the right side.
With the vertical isomorphisms corresponding to the AS orientations induced 
by~$\psi_{V\oplus 2L}$ and~$\psi'_{V\oplus 2L}$,
it again commutes up to homotopy, without any condition on~$L_0|_{\Si_0^+}$ in this~case.
With the vertical isomorphisms corresponding to
the stabilization orientations induced by~$\psi_{V\oplus 2L}$ and~$\psi'_{V\oplus 2L}$,
the analogue of the diagram~\eref{ARSS_e17} 
now commutes up to homotopy as well because $\deg L_0|_{\P_i^1}\!\in\!2\Z$ in this case
and thus the left-hand side of~\eref{ARSS_e18} is even.
By Lemmas~3.3 and~3.7 in~\cite{RealGWsII}, the homotopy classes of isomorphisms~\eref{ARSS_e15}
and~\eref{ARSS_e25}
corresponding to the AS and stabilization orientations induced by~$\psi_{V\oplus 2L}$
and~$\psi'_{V\oplus 2L}$ are the~same. 
Thus, the homotopy classes of left vertical isomorphisms in~\eref{ARSS_e17}
corresponding to the AS and stabilization orientations induced by~$\psi_{V\oplus 2L}$
and~$\psi'_{V\oplus 2L}$ are the~same.

\subsection{Equivariant localization, {\cite[Sections~4,5]{RealGWsIII}}}
\label{EquivLocal_sub}

\noindent
Let $n\!\in\!\Z^+$, $k\!\in\!\Z^{\ge0}$, and $\a\!\in\!(\Z^+)^k$ be as in 
Section~\ref{RealGWth_subs} and $m\!\in\!\Z^{\ge0}$ be the integer part of~$n/2$.
We denote by~$\phi$ the involution~$\tau_n$ or~$\eta_n$ on~$\P^{n-1}$,
with $n\!\in\!2\Z$ in the latter case, and 
by $X_{n;\a}\!\subset\!\P^{n-1}$ a $\phi$-real complete intersection of multi-degree~$\a$,
which we assume exists.
Let
$$|\phi|=\begin{cases}0,&\hbox{if}~\phi\!=\!\tau_n;\\
1,&\hbox{if}~\phi\!=\!\eta_n.
\end{cases}$$
We also assume that $n\!-\!k\!\in\!2\Z$ and that the pair $(n,\a)$ satisfies 
the conditions on real orientability of~$X_{n;\a}$ in \cite[Proposition~1.4]{RealGWsIII}
if $\phi\!=\!\tau_n$.
In particular,
\BE{kacong_e}|\a|\equiv a_1\!+\!\ldots\!+\!a_k\cong k \mod2.\EE
The real symplectic manifold $(X_{n;\a},\om_n|_{X_{n;\a}},\phi|_{X_{n;\a}})$
then admits a real orientation $(L,[\psi],\fs)$ with \hbox{$\deg L\!=\!(n\!-\!|\a|)/2$};
see \cite[Sections~2.1/2]{RealGWsIII} and Section~\ref{RealGWth_subs} above.\\

\noindent
We denote by $\cL\!\in\!H_2(X;\Z)$ the homology class of a linearly embedded $\P^1\!\subset\!\P^{n-1}$
and by~$J_{\P^{n-1}}$ the standard complex structure on~$\P^{n-1}$.
The restriction of the standard $\bT^n$-action on~$\P^{n-1}$
to a subtorus $\bT^m\!\subset\!\bT^n$ commutes with the involution~$\phi$ and 
has the same fixed points
$$P_1\equiv[1,0,\ldots,0], \qquad\ldots,\qquad P_n\equiv[0,\ldots,0,1];$$
see \cite[Section~4.1]{RealGWsIII}.
This restriction induces a $\bT^m$-action on the moduli space 
$$\ov\fM_{g,\ell}^{\phi}\big(\P^{n-1},d\big)\equiv 
\ov\fM_{g,\ell}^{\phi}\big(d\cL;J_{\P^{n-1}}\big)$$
of real genus~$g$ degree~$d$ holomorphic maps to~$(\P^{n-1},\phi)$ with 
$\ell$~conjugate pairs of marked points.
Equivariant localization contributions from some $\bT^m$-fixed loci of $\ov\fM_{g,\ell}^{\phi}(\P^{n-1},d)$
to the integral \cite[(4.4)]{RealGWsIII} are provided by \cite[(4.26)]{RealGWsIII}.
This integral computes twisted real GW-invariants 
of~$(\P^{n-1},\om_n,\phi)$ associated with~$\a$.
For $k\!=\!0$, these invariants are the real GW-invariants of~$(\P^{n-1},\om_n,\phi)$
with the real orientation $(L,[\psi],\fs)$ and a conjugation~$\wt\phi$ on~$L$ lifting~$\phi$
specified in \cite[Section~2.2]{RealGWsIII}.
In this case, \cite[Theorem~4.6]{RealGWsIII} specifies the equivariant
localization contributions from all $\bT^m$-fixed loci
and thus provides a way of computing these real invariants.\\

\noindent
The proof of \cite[(4.26)]{RealGWsIII} is based on
the implicit presumption in \cite[(5.25)]{RealGWsIII} that the isomorphisms~\eref{unionorient_e} 
are orientation-preserving in the relevant cases.
Since the moduli spaces in~\cite{RealGWsIII} are oriented via~\eref{ROrientCompl_e3} and 
the degree~$\fd(e)$ of the real map~$[f]$ corresponding to 
an edge $e\!\in\!\nE_{\R}^{\si}(\Ga)$ is odd
if the corresponding factor in \cite[(4.26)]{RealGWsIII}
is not zero,
the second statement of Proposition~\ref{unionorient_prp} with $g_1,g_2\!=\!0$ implies that 
the sign of the isomorphism in \cite[(5.25)]{RealGWsIII}
is $(-1)^{\eps_{\Ga}(\a)}$ with
$$\eps_{\Ga}(\a)\equiv \frac{n\!-\!2\!-\!k}2\binom{|\nE_{\R}^{\si}(\Ga)|}2
\!+\!\frac{n\!-\!2\!-\!|\a|}2\binom{|\nE_{\R}^{\si}(\Ga)|}2
\cong \bigg(n\!+\!\frac{k\!+\!|\a|}{2}\bigg)\binom{|\nE_{\R}^{\si}(\Ga)|}2
\mod2\,.$$
However, this is of relevance for the purposes of \cite[(4.26)]{RealGWsIII}
only if all components~$a_i$ of~$\a$ are odd.
If $\phi\!=\!\tau_n$, $|\a|\!\cong\!k$ mod~4 then
by the conditions on real orientability of~$X_{n;\a}$ in \cite[Proposition~1.4]{RealGWsIII}.
If $\phi\!=\!\eta_n$, then the odd components~$a_i$ come in pairs and 
thus $|\a|\!\cong\!k$ mod~4 in this case as well.
It follows that the implicit presumption in \cite[(5.25)]{RealGWsIII} is valid in
the relevant cases.\\

\noindent
On the other hand, the real edge contributions $\Cntr_{(\Ga,\si);e}$ 
with $e\!\in\!\nE_{\R}^{\si}(\Ga)$ appearing in \cite[(4.26)]{RealGWsIII} 
and defined in \cite[(4.23)]{RealGWsIII} were obtained
in \cite[Section~5.4]{RealGWsIII} under the assumption that 
the moduli spaces~$\R\ov\cM_{0,\ell}$ are oriented as in~\cite[Section~1.4]{RealGWsII}.
If they are oriented as in Section~\ref{MSorient0_subs} of the present paper,
the right-hand side in \cite[(4.23)]{RealGWsIII} should be multiplied by~$(-1)$.
This does not affect the real odd-genus GW-invariants obtained 
in~\cite{RealGWsIII,RealGWvsEnum}, but does change the sign of
the real even-genus GW-invariants.
By \cite[(1.1)]{RealGWsIII}, this makes the count of real lines through 
a conjugate pair of points in~$(\P^{2m-1},\tau_{2m})$,
$$\blr{H^{2m-1}}_{0,1}^{\P^{2m-1},\tau_{2m}}
\equiv\blr{\tau_0(\pt)\!}_{0,1}^{\!\om_{2m},\tau_{2m}},$$
to be $-1$ with respect to the canonical real orientation 
specified in \cite[Section~2.2]{RealGWsIII}.
In light of the last claim of Corollary~\ref{RelSpinOrient_crl}, 
the $2m\!=\!4$ case of this statement agrees with the second equality
in the first equation in \cite[(13.26)]{SpinPin},
as the $(\phi,\tau)$-moduli spaces in~\cite{SpinPin} are oriented 
from a spin structure via the canonical isomorphism~\eref{RfMisom_e};
see \cite[(12.24)]{SpinPin}.
Changing the sign of the right-hand side in \cite[(4.23)]{RealGWsIII} also
reverses the signs of real GW-invariants~$\GW_{g,d}^{\phi}$ of~$(\P^3,\om_4,\tau_4)$ 
and real enumerative invariants~$\nE_{g,d}^{\phi}$ of~$(\P^3,\tau_4)$ 
with $g\!\in\!2\Z$ in \cite[Table~2]{RealGWvsEnum}.
For $(n,(\a))\!=\!(5,(5))$ and $g\!\in\!2\Z$, this change makes 
the overall equivariant contribution~\cite[(4.26)]{RealGWsIII}
the negative of \cite[(3.22)]{Walcher2}.\\

\noindent
Suppose now that the real moduli spaces are oriented via~\eref{ROrientCompl_e2}.
By the first statement of Proposition~\ref{unionorient_prp} with $g_1,g_2\!=\!0$,
the right-hand sides in \cite[(4.26),(5.25)]{RealGWsIII}
should then be multiplied by~$(-1)$ to the power of~$\frac{n-2-k}2\binom{|\nE_{\R}^{\si}(\Ga)|}2$.
The exponent~$\fs_v$ of the vertex sign appearing in \cite[(4.26)]{RealGWsIII}
and defined below \cite[(4.17)]{RealGWsIII} accounts for 
the sign of Proposition~\ref{Dblorient_prp} with $B\!=\!0$ if $\val_v(\Ga)\!\ge\!3$
and for the sign of Proposition~\ref{Rnodisom_prp} 
(applied $|\nE_v(\Ga)|$ times if $\val_v(\Ga)\!\ge\!3$ and 
$|\nE_v(\Ga)|\!-\!1$ times if $\val_v(\Ga)\!=\!1,2$);
see the second half of \cite[Remark~4.3]{RealGWsIII}.
In this case, it thus reduces~to
$$\fs_v=\sum_{i\in S_v^-}\!\!\big(1\!+\!b_i\!+\!p_i\big).$$
The conjugate edge contributions $\Cntr_{(\Ga,\si);e}$ with $e\!\in\!\nE_+^{\si}(\Ga)$
appearing in \cite[(4.26)]{RealGWsIII} and defined in \cite[(4.21/2)]{RealGWsIII}
include the sign of Proposition~\ref{Dblorient_prp};
see the sentence in \cite[Remark~4.4]{RealGWsIII} referencing [13, Theorem~1.4].
In this case, the expressions in \cite[(4.21/2)]{RealGWsIII} should thus be multiplied
by~$(-1)^{(n-|\a|)\fd(e)/2-1}$.\\

\noindent
It remains to determine the sign $(-1)^{\eps}$ of the real edge contribution~$\Cntr_{(\Ga,\si);e}$ 
with $e\!\in\!\nE_{\R}^{\si}(\Ga)$ appearing in \cite[(4.26)]{RealGWsIII} 
and defined in \cite[(4.23)]{RealGWsIII} if the real moduli spaces are 
oriented via~\eref{ROrientCompl_e2}.
This can be done under the assumption that the edge degree~$\fd(e)$ 
and all components~$a_i$ of~$\a$ are odd.
If $\phi\!=\!\tau_n$ or $n\!-\!|\a|\!\in\!4\Z$, we can apply Lemma~\ref{CvsCanorient_lmm}~with 
$$g=0, \qquad  \rk_{\C}L=1, \quad\hbox{and}\quad  
\deg L=-\frac{n\!-\!|\a|}2\fd(e)$$
to compare the orientations induced via~\eref{ROrientCompl_e3} and~\eref{ROrientCompl_e2}.
Taking into consideration the reversal of the orientations of the moduli spaces~$\R\ov\cM_{0,\ell}$
as above, we conclude that the sign exponent in \cite[(4.26)]{RealGWsIII} should be
\BE{Redgeeps_e} \eps=|\phi|\!+\!\frac{\fd(e)\!+\!1}2\!+\!
\flr{\!\frac{n\!-\!|\a|}4\fd(e)\!}\,.\EE
By comparing with the case $\phi\!=\!\tau_n$ and $n\!-\!|\a|\!+\!2\!\in\!4\Z$
established in \cite[Section~5.4]{RealGWsII},
we show in Section~\ref{EquivLocal_sub2} that \eref{Redgeeps_e} applies in the remaining case,
i.e.~$\phi\!=\!\eta_n$ and $n\!-\!|\a|\!+\!2\!\in\!4\Z$, as~well.\\

\noindent
For any graph $(\Ga,\si)$ with a set~$\Ver$ of vertices and a set~$\Edg$ of edges
as in \cite[Theorem~4.6]{RealGWsIII},
\begin{gather*}
%2\big|\nV_+^{\si}(\Ga)\big|=|\Ver|, \quad
\big|\nE_{\R}^{\si}(\Ga)\big|\!+\!2\big|\nE_+^{\si}(\Ga)\big|
=|\Edg|=\sum_{v\in\nV_+^{\si}(\Ga)}\!\!\!\!\!\!\big|\nE_v(\Ga)\big|,\\
|\Edg|\!+2\!\!\!\!\!\sum_{v\in\nV_+^{\si}(\Ga)}\!\!\!\!\!\!\big(\fg(v)\!-\!1\big)
=g\!-\!1,  \quad\hbox{and}\quad
\sum_{e\in\nE_{\R}^{\si}(\Ga)}\!\!\!\!\!\fd(e)\!+
2\!\!\!\!\!\sum_{e\in\nE_+^{\si}(\Ga)}\!\!\!\!\!\fd(e)=d\,.
\end{gather*}
If the product in \cite[(4.26)]{RealGWsIII} is nonzero, then $|\a|\!\cong\!k$ mod~4
and $\fd(e)\!\not\in\!2\Z$ for every $e\!\in\!\nE_{\R}^{\si}(\Ga)$. 
In such a~case, $n\!-\!|\a|,dg\!\in\!2\Z$ and
\begin{equation*}\begin{split}
&\frac{n\!-\!2\!-\!k}2\binom{|\nE_{\R}^{\si}(\Ga)|}2\!+\!
\sum_{e\in\nE_{\R}^{\si}(\Ga)}\!\!\!\bigg(\!1\!+\!\flr{\!\frac{n\!-\!|\a|}4\fd(e)\!}\!\bigg)
\!\!+\!\sum_{e\in\nE_+^{\si}(\Ga)}\!\!\!\bigg(\frac{n\!-\!|\a|}2\fd(e)\!\!-\!1\!\!\bigg)
\!+\!\!\sum_{v\in\nV_+^{\si}(\Ga)}\!\!\!\!\!\!\!\big(\fg(v)\!\!-\!1\!+\!|\nE_v(\Ga)|\big)\\
&\hspace{2in}
\cong \frac{n\!-\!|\a|}4d\bigg(\frac{n\!-\!|\a|}2d\!\!-\!1\!\!\bigg)
\!+\!\frac{g(g\!-\!1)}2\!+\!\frac{n\!-\!|\a|}2dg\!+\!\big|\nE_{\R}^{\si}(\Ga)\big|\\
&\hspace{2in}
\cong \frac12\bigg(g\!+\!\frac{n\!-\!|\a|}2d\!\bigg)
\bigg(g\!+\!\frac{n\!-\!|\a|}2d\!\!-\!1\bigg)\!+\!(g\!-\!1),
\end{split}\end{equation*}
with $\cong$ denoting the mod 2 congruence.
By Lemma~\ref{CvsCanorient_lmm} with $\rk_{\C}L\!=\!1$ and \hbox{$\deg L\!=\!-(n\!-\!|\a|)d/2$},
the first summand on the right-hand side above accounts for the difference 
in orienting the first factor in~\eref{RfMisom_e} 
via~\eref{ROrientCompl_e2} vs.~\eref{ROrientCompl_e3}.
The second summand on the right-hand side corresponds to 
the reversal of the orientation of the second factor in~\eref{RfMisom_e}
for $g\!\in\!2\Z$.

\subsection{Real GW vs.~curve counts, {\cite{RealGWvsEnum}}}
\label{RealGWvsEnum_sub}

\noindent
Let $(X,\om,\phi)$ be a compact real symplectic sixfold and $B\!\in\!H_2(X;\Z)$.
For $g,h\!\in\!\Z^{\ge0}$, define \hbox{$\wt{C}_{h,B}^X(g)\!\in\!\Q$} by
\BE{wtCdfn_e} \sum_{g=0}^{\i}\wt{C}_{h,B}^X(g)t^{2g}
=\bigg(\frac{\sinh(t/2)}{t/2}\bigg)^{h-1+\lr{c_1(X,\om),B}/2}\,.\EE
If $(L,[\psi],\fs)$ is a real orientation on~$(X,\om,\phi)$
and 
\hbox{$\mu_1,\ldots,\mu_{\ell}\!\in\!H^*(X;\Z)$}, define
\hbox{$\nE_{h,B}^{X,\phi}(\mu_1,\ldots,\mu_{\ell})\!\in\!\Q$} by 
\BE{nEdfn_e}
\int_{[\ov\fM_{g,\ell}^{\phi}(X,B)]^{\vir}}\!\! 
\big(\ev_1^*\mu_1\big)\!\ldots\!\big(\ev_{\ell}^*\mu_{\ell}\big)
=\sum_{\begin{subarray}{c}0\le h\le g\\ g-h\in 2\Z\end{subarray}}\!\!\!\!
\wt{C}_{h,B}^X\big(\tfrac{g-h}{2}\big) \nE_{h,B}^{X,\phi}\big(\mu_1,\ldots,\mu_{\ell}\big);\EE
the left-hand side above is the real genus~$g$ degree~$B$ (primary) GW-invariants of $(X,\om,\phi)$
determined the insertions~$\mu_1,\ldots,\mu_{\ell}$.
By \cite[Theorem~1.1]{RealGWvsEnum}, $\nE_{h,B}^{X,\phi}(\mu_1,\ldots,\mu_{\ell})$	
is a signed count of $\phi$-invariant $J$-holomorphic genus~$h$ degree~$B$ curves
in~$X$ for a generic $J\!\in\!\cJ_{\om}^{\phi}$ that pass through generic representatives 
for the Poincare duals of $\mu_1,\ldots,\mu_{\ell}$,
if the complex line bundle~$L$ admits a conjugation~$\wt\phi$ lifting~$\phi$ and
the moduli space~$\ov\fM_{g,\ell}^{\phi}(X,B)$ is oriented via~\eref{ROrientCompl_e3}.  
In particular, \hbox{$\nE_{h,B}^{X,\phi}(\mu_1,\ldots,\mu_{\ell})\!\in\!\Z$}.
As we note below, the proof in~\cite{RealGWvsEnum} contains a subtle omission.
On the other hand, the conclusion holds even if $L$ does not admit 
a conjugation~$\wt\phi$ lifting~$\phi$,
but $\ov\fM_{g,\ell}^{\phi}(X,B)$ is oriented via~\eref{ROrientCompl_e2}
and $\sinh$ in~\eref{wtCdfn_e} is replaced by~$\sin$.\\
	
\noindent
The last paragraph in the proof of \cite[Proposition~2.4]{RealGWvsEnum} presumes
that the isomorphism~\eref{unionorient_e} in the present paper is orientation-preserving.
This is indeed the case in the proof of \cite[Proposition~2.4]{RealGWvsEnum} because 
\begin{enumerate}[label=(\arabic*),leftmargin=*]

\item the moduli spaces of real maps in~\cite{RealGWvsEnum} are oriented 
via~\eref{ROrientCompl_e3};

\item the proof of \cite[Proposition~2.4]{RealGWvsEnum} depends on applications of
the second statement of Proposition~\ref{unionorient_prp} only with the map~$\u_2$ being a doublet 
with the restriction of the map to each connected component of the domain
being constant and thus with $g_2\!\not\in\!2\Z$ and $B_2\!=\!0$.

\end{enumerate}
If the moduli spaces of real maps in~\cite{RealGWvsEnum} were oriented
via~\eref{ROrientCompl_e2},
the isomorphism~\eref{unionorient_e} in the present paper would still be orientation-preserving
in the applications in the proof of \cite[Proposition~2.4]{RealGWvsEnum}
by the first statement of Proposition~\ref{unionorient_prp} with $g_2\!\not\in\!2\Z$.
The first statement of Proposition~\ref{Dblorient_prp} would be used 
with $B\!=\!0$ and $S^-\!=\!\eset$.
The node-identifying immersion \cite[(3.21)]{RealGWvsEnum},
which is a composition of the first node-identifying immersions
in~\eref{Riogldfn_e2} in the present paper,
would then be orientation-preserving by the first statement of Proposition~\ref{Rnodisom_prp}.
The use of the first instead of the second statements of 
Propositions~\ref{Dblorient_prp} and~\ref{Rnodisom_prp}
would drop the signs of~$(-1)^{\fg(e/w_0)}$ and~$(-1)^{g_i}$
from the right-hand sides in the last equation in the proof of \cite[Proposition~2.4]{RealGWvsEnum} 
and in the statement of this proposition, respectively.
The latter would also be dropped from the right-hand sides of \cite[(2.10)]{RealGWvsEnum}
and of the equations in Section~2 in~\cite{RealGWvsEnum}
after Proposition~2.4.
As a consequence, $\fI$ would drop out from the last two expressions in the last equation
in \cite[Section~2]{RealGWvsEnum}.
This would then establish the statement below~\eref{nEdfn_e} with
$\sinh$ in~\eref{wtCdfn_e} replaced by~$\sin$.

\subsection{Real GW-theory of local 3-folds, {\cite{GI}}}
\label{KleiTQFT_sub}

\noindent
A complex vector bundle~$E$ over a compact symplectic manifold~$(X,\om)$ gives rise to
\sf{twisted GW-invariants} of~$(X,\om)$ that have played a prominent role in (complex) GW-theory.
These are the integrals of the Euler classes of the Index bundles of CR-operators 
on the pullbacks of~$E$ over the moduli spaces $\ov\fM_{g,\ell}^{\bu}(X;J)$
of $J$-holomorphic maps.
When $X$ is a Riemann surface and \hbox{$E\!=\!L_1\!\oplus\!L_2$} is the direct sum of 
complex line bundles~$L_1,L_2$,
such invariants give rise to a topological quantum field theory and encode
the local structure of standard GW-invariants of symplectic sixfolds.
Analogous integrals in the real setting, with \hbox{$L_1\!\approx\!L_2$}, 
are set up and studied in~\cite{GI}.
These integrals give rise to a Klein topological quantum field theory and should encode
the local structure of real GW-invariants of real symplectic threefolds with a certain
real orientation, in the sense of \cite[Definition~5.1]{RealGWsI} or \cite[Definition~A.1]{GI}
(i.e.~Definition~\ref{realorient_dfn0} in the present paper).
Contrary to what the abstract and the introduction in~\cite{GI} might suggest,
the associated real orientation on~$E$ is not specified directly,
as it is not relevant to evaluating the integrals \cite[(2.10)]{GI}.
Instead, real orientations on the symmetric Riemann surfaces $X\!=\!\Si$ 
in the sense of \cite[Definition~A.1]{GI} are indicated in terms of 
the properties of orientations they induce on~$\ov\fM_{g,\ell}^{\phi;\bu}(X;J)$;
the associated real orientation on~$E$ can then be deduced from these specifications.\\

\noindent
The explicit real orientations on $\P^1$ with a conjugate pair of marked points
described in the two cases of the proof of \cite[Proposition~5.2]{GI},
which are not really needed for the purposes of~\cite{GI}, 
are off though due to an error in the first statement of \cite[Lemma~5.3]{GI}; 
the correct statement is provided by Lemma~\ref{CvsCanorient_lmm} in the present paper.
With this correction, the homotopy class of isomorphism~$\th_{tw}$ in the last paragraph
of Case~1 of the proof of this proposition is given by the isomorphism~$\th$
in~\cite[(5.9)]{GI}, which is the inverse of the isomorphism~\eref{GIgen_e} in the present paper, 
with $L\!=\!\Si\!\times\!\C$ and $\phi\!=\!\fc$
(not the composition of~$\th$ with $\id_{\Si\times\C}\!\oplus\!(-\id_{\Si\times\C})$,
as indicated by the second sentence of this paragraph).
The isomorphism \cite[(3.14)]{GI} induced by the real orientation specified in this way 
is the isomorphism induced by~$\phi$ in~\cite[(5.3)]{GI}.
The isomorphism~$\th$ in the second and third paragraphs of Case~2 of the proof 
of this proposition should be replaced 
by $\{\id_{\Si\times\C}\!\oplus\!(-\id_{\Si\times\C})\!\}\!\circ\!\th$.
On the other hand, if the orientation on~$\R\ov\cM_{0,2}$ in the proof of 
\cite[Proposition~5.2]{GI} were chosen as in the present paper and in~\cite{Penka2},
instead of as in~\cite{RealEnum,RealGWsI},
the real orientations in the two cases of the proof of this proposition would have
been as indicated in its proof.
The term {\it solutions} in the first and third paragraphs of Case~2 in this proof refers
to real branched covers.
The only other statement in~\cite{GI} which uses its Lemma~5.3 is Lemma~A.3,
where~$(-1)^{\io}$ should analogously be replaced by~$(-1)^{\io(\io-1)/2}$.
The latter is in turn used only at the end of the proof of \cite[Proposition~5.2]{GI},
where {\it any} dependence of the comparison of Lemma~A.3 on the index~$\io$ alone suffices.
In particular, the statement of \cite[Proposition~5.2]{GI} is not impacted in any~way.\\

\noindent
Similarly to the oversights in~\cite{RealGWsIII} and~\cite{RealGWvsEnum},
the claim of the expected exponential relation between connected and disconnected
real GW-invariants below \cite[Theorem~1.1]{GI} and in~\cite[(2.16)]{GI}
implicitly presumes that the isomorphisms~\eref{unionorient_e}  are orientation-preserving 
when $(X,\phi)\!=\!(\Si,c)$ and thus $n\!=\!1$.
Since the moduli spaces of real maps to~$(\Si,c)$ 
are oriented via~\eref{ROrientCompl_e2} in~\cite{GI}, this is indeed the~case.

\subsection{Wrapup of proof of~\eref{Redgeeps_e}}
\label{EquivLocal_sub2}

\noindent
Let $n\!\in\!2\Z$, $m\!=\!n/2$, $\phi\!=\!\tau_n,\eta_n$, and
$$c=\begin{cases}\tau,&\hbox{if}~\phi\!=\!\tau_n\,;\\
\eta,&\hbox{if}~\phi\!=\!\eta_n\,.\end{cases}$$
We denote by $\ga\!\lra\!\P^{n-1}$ the tautological line bundle.
For $a\!\in\!\Z$, the conjugation~$\phi$ on~$\P^{n-1}$ lifts to a $\C$-antilinear automorphism
$$\fc_{\phi}\!:\ga^{\otimes a}\lra\ga^{\otimes a};$$
see \cite[Section~2.1]{RealGWsIII}.
If $\phi\!=\!\tau_n$ or $a\!\in\!2\Z$, $\fc_{\phi}$ is a conjugation;
if $a\!\not\in\!2\Z$, $\fc_{\eta}^2$ is the multiplication by~$-1$.
Define
$$\big(\cL_{n;a},\wt\phi_{n;a}\big)=
\big(\ga^{*\otimes a}\!\oplus\!\phi^*\ov{\ga^{*\otimes a}},\phi_{\ga^{*\otimes a}}^{\oplus}\big)
\lra \big(\P^{n-1},\phi\big).$$
If $\phi\!=\!\eta_n$, this real bundle pair is isomorphic to the real bundle pair
$(2\ga^{*\otimes a},\wt\eta_{n;1,1}^{(a)})$ in \cite[Section~2.1]{RealGWsIII}
via the~map
$$\big(2\ga^{*\otimes a},\wt\eta_{n;1,1}^{(a)}\big)\lra \big(\cL_{n;a},\wt\phi_{n;a}\big),
\quad
(v,w)\lra \big(v,\fc_{\eta}(w)\!\big).$$
If $\phi\!=\!\tau_n$, $(\cL_{n;a},\wt\phi_{n;a})$ is isomorphic to
the real bundle pair $2(\ga^{*\otimes a},\fc_{\tau})$ via the map~$\Phi_L$ in~\eref{GIgen_e}
with $L\!\equiv\!\ga^{*\otimes a}$.\\

\noindent
Let $\a\!\equiv\!(a_i)_{i\in[k]}$ be a $k$-tuple of elements of $\Z^+\!-\!2\Z$, 
$n'\!=\!n\!-\!2k$, \hbox{$\nu_n(\a)\!=\!m\!-\!|\a|$}, and 
$$\big(\cL_{n;\a},\wt\phi_{n;\a}\big)
=\bigoplus_{i=1}^k\big(\cL_{n;a_i},\wt\phi_{n;a_i}\big)
\lra \big(\P^{n-1},\phi\big).$$
We view $\P^{n'-1}\!\subset\!\P^{n-1}$ as being given by the first $n'$~homogeneous coordinates
so that its normal bundle
$$(V_c,\vph_c)\equiv\bigoplus_{i=1}^k\!\big(\cL_{n';1},\wt\phi_{n';1}\big)
\lra \big(\P^{n'-1},\phi|_{\P^{n'-1}}\big)$$
corresponds to the last $2k$ homogeneous coordinates.
A $\phi$-real complete intersection $X_{n;\a\a}\!\subset\!\P^{n-1}$ is 
cut out by a transverse $(\wt\phi_{n;\a},\phi)$-real section~$s_{n;\a\a}$ of~$\cL_{n;\a}$.
Thus, the sequence
\BE{CIprp_e4}0\lra \big(TX_{n;\a\a},\tnd\phi_{n;\a\a}\big)\lra  
\big(T\P^{n-1},\nd\phi\big)\!\big|_{X_{n;\a\a}}\xlra{\na s_{n;\a\a}}  
\big(\cL_{n;\a},\wt\phi_{n;\a}\big)\!\big|_{X_{n;\a\a}}\lra 0\EE
of real bundle pairs over $X_{n;\a\a}$, 
where $\phi_{n;\a\a}\!=\!\phi|_{X_{n;\a\a}}$, is exact.
Along with the Euler exact sequence of \cite[Lemma~2.1]{RealGWsIII}, 
it determines isomorphisms
\BE{CIprp_e5}\begin{split}
\La_{\C}^{\top}\big(TX_{n;\a\a},\tnd\phi_{n;\a\a}\big)\!\otimes\!
\La_{\C}^2\big(\cL_{n;|\a|},\wt\phi_{n;|\a|}\big)\!\big|_{X_{n;\a\a}}
&\approx \La_{\C}^{\top}\big(T\P^{n-1},\tnd\phi\big)\big|_{X_{n;\a\a}},\\
\La_{\C}^{\top}\big(TX_{n;\a\a},\tnd\phi_{n;\a\a}\big) 
&\approx \La_{\C}^2\big(\cL_{n;\nu_n(\a)},\wt\phi_{n;\nu_n(\a)}\big)\!\big|_{X_{n;\a\a}}.
\end{split}\EE
The last isomorphism determines $L\!\equiv\!L_{n;\a}\!\equiv\!\ga^{*\otimes\nu_n(\a)}$
and~$[\psi]$ as in Definition~\ref{realorient_dfn0}.
If $\phi\!=\!\tau_n$ or $n\!-\!|\a|\!\in\!4\Z$, this isomorphism corresponds
to the isomorphisms~\cite[(2.18)]{RealGWsIII} multiplied by~$2^{n'/2-1}$
under identifications as in~\eref{GIgen_e}.\\

\noindent
Let $S^1\!\subset\!\P^1$ be an embedded circle preserved by the involution~$c$ and
\hbox{$f\!:\P^1\!\lra\!\P^{n-1}$} be a $(\phi,c)$-real continuous map.
Let $s_1$ be a nowhere zero section of the complex line bundle~$f^*\ga^*|_{S^1}$.
For each $a\!\in\!\Z$, the bundle homomorphism
\begin{gather*}
\Psi_a\!: \big(S^1\!\times\!\C^2,c|_{S^1}\!\times\!\fc\big)
\lra  f^*\cL_{n;a}\big|_{S^1}\!=\!
\big(f^*\ga^{*\otimes a}\big|_{S^1}\big)\!\oplus\!
\big\{c|_{S^1}\big\}^{\!*}\ov{\big(f^*\ga^{*\otimes a}\big|_{S^1}\big)},\\
\quad \Psi_a(z,c_1,c_2)=\big( (c_1\!+\!\fI c_2)s_1(z)^{\otimes a},
\big(z,\ov{(c_1\!-\!\fI c_2)}s_1(c(z)\!)^{\otimes a}\big)\!\big),
\end{gather*}
is then an isomorphism of real bundle pairs over $(S^1,c|_{S^1})$
and thus determines a trivialization of its target.
Any other nowhere zero section~$s_1'$ of~$f^*\ga^*|_{S^1}$ is 
of the form $s_1'\!=\!\rho s_1$ for some $\C^*$-valued continuous function~$\rho$ on~$S^1$.
In such a case, $\Psi_a^{-1}\!\circ\!\Psi_a'$ corresponds~to the continuous function
$$S^1\lra \GL_2\C, \quad z\lra
\left(\begin{array}{cc}1& 1\\ -\fI& \fI\end{array}\right)
\left(\begin{array}{cc}\rho(z)^a& 0\\ 0&\ov{\rho(c(z)\!)}^a\end{array}\right)
\left(\begin{array}{cc}1& 1\\ -\fI& \fI\end{array}\right)^{-1}.$$
For $a\!\in\!2\Z$, $\Psi_a^{-1}\!\circ\!\Psi_a'$ thus lies in the canonical homotopy class
of trivializations of $(S^1\!\times\!\C^2,c|_{S^1}\!\times\!\fc)$.
It follows that there is a natural correspondence, independent of~the choice of~$s_1$,
between the homotopy classes of trivializations of~$f^*\cL_{n;a_1}\big|_{S^1}$
and~$f^*\cL_{n;a_2}\big|_{S^1}$  whenever~$a_1$ and~$a_2$ are of the same parity.\\

\noindent
Suppose in addition that $f$ is holomorphic and 
\BE{RealOrientCI_e19a}
0\lra (V,\vph) \lra f^*\big(T\P^{n-1},\tnd\phi|_{T\P^{n-1}}\big)
\lra f^*\big(\cL_{n;\a},\wt\phi_{n;\a}\big)\lra 0\EE
is an exact sequence of holomorphic real bundle pairs over $(\P^1,c)$.
Similarly to~\eref{CIprp_e5}, it determines an isomorphism
\BE{CIprp_e5b}
\La_{\C}^{\top}(V,\vph)
\approx f^*\La_{\C}^2\big(\cL_{n;\nu_n(\a)},\wt\phi_{n;\nu_n(\a)}\big).\EE
It also determines a homotopy class of isomorphisms
\BE{CIprp_e7}\begin{split}
&\big(\P^1\!\times\!\C,c\!\times\!\fc\big)
\!\oplus\!\big(\!(V,\vph)\!\oplus\!f^*(\cL_{n;-\nu_n(\a)},\wt\phi_{n;-\nu_n(\a)})\!\big)
\!\oplus\!f^*\big(\cL_{n;\a},\wt\phi_{n;\a}\big)\\
&\hspace{1in}\approx
\frac{n'\!+\!2}2f^*\big(\cL_{n;1},\wt\phi_{n;1}\big)\!\oplus\!
(k\!-\!1)f^*\big(\cL_{n;1},\wt\phi_{n;1}\big)\!\oplus\!
f^*\big(\cL_{n;-\nu_n(\a)},\wt\phi_{n;-\nu_n(\a)}\!\big).
\end{split}\EE
If $\nu_n(\a)\!\not\in\!2\Z$, then $n'\!+\!2\!\in\!4\Z$.
The restriction of the first summand on the right-hand side above to~$S^1$ then has 
a canonical homotopy class of trivializations;
it is obtained by taking the same trivialization on each of the $(n'\!+\!2)/2$ copies 
of~$f^*(\cL_{n;1},\wt\phi_{n;1})|_{S^1}$.
By the previous paragraph, the trivializations of 
the $k$~factors~$f^*(\cL_{n;a_i},\wt\phi_{n;a_i})$ of~$f^*(\cL_{n;\a},\wt\phi_{n;\a})$
over~$S^1$ correspond to the trivializations of the summands of 
the last two terms on the right-hand side of~\eref{CIprp_e7}.
Thus, there is a canonical homotopy class of trivializations
\BE{CIprp_e9}
\big(\!(V,\vph)\!\oplus\!f^*(\cL_{n;-\nu_n(\a)},\wt\phi_{n;-\nu_n(\a)})\!\big)\!\big|_{S^1}
\approx \big(S^1\!\times\!\C^{n'+1},c|_{S^1}\!\times\!\fc\big)\,.\EE
The induced homotopy class of trivializations on the top exterior powers 
corresponds to the restriction of~\eref{CIprp_e5b} to~$S^1$.
If $\phi\!=\!\tau_n$, the restrictions of these trivializations to the real parts of 
the two sides determine a homotopy class of trivializations 
\BE{CIprp_e9b}V^{\vph}\!\oplus\!f^*L_{n;\a}^{\,*}\big|_{S^1}\approx S^1\!\times\!\R^{n'+1}.\EE
The collection of such trivializations obtained by pulling back the exact sequence~\eref{CIprp_e4}
by continuous maps~$f$ with values in~$X_{n;\a\a}$ determines a spin structure~$\fs$ 
on the real vector bundle 
$TX_{n;\a\a}^{\phi_{n;\a\a}}\!\oplus\!L_{n;\a}^{\,*}|_{X_{n;\a\a}^{\phi_{n;\a\a}}}$
over~$X_{n;\a\a}^{\phi_{n;\a\a}}$.
Via the map~$\Phi_{L_{n;\a}^*}$ as in~\eref{GIgen_e}, this spin structure corresponds
to the canonical spin structure provided by \cite[(2.26)]{RealGWsIII}.\\

\noindent
We now also assume that $\nd f(T \P^1)\!\subset\!V$ and the degree of~$f$ is~$d$.
We denote by~$\dbar_V$, $\dbar_{\P^{n-1}}$,  and~$\dbar_{\cL_{n;\a}}$ 
the standard real Cauchy-Riemann operators on~$(V,\vph)$, $f^*(T\P^{n-1},\tnd\phi|_{T\P^{n-1}})$,
and $f^*(\cL_{n;\a},\wt\phi_{n;\a})$, respectively, and~by 
$$G_c(f)\subset \ker\dbar_V\subset\ker \dbar_{\P^{n-1}}
=T_f\big(\fP_0^{\phi,c}(\P^{n-1},d)\!\big)$$
the tangent space to the orbit of the $G_c$-action on the space of 
parametrized real degree~$d$ holomorphic maps to~$\P^{n-1}$.
The orientation on~$G_c$ induces an orientation on~$G_c(f)$.
The operators~$\dbar_V$ and $\dbar_{\P^{n-1}}$ descend
to operators~$\dbar_V'$ and $\dbar_{\P^{n-1}}'$ on the quotients of their domains
by~$G_c(f)$.
The operators $\dbar_{\P^{n-1}}'$ and $\dbar_{\cL_{n;\a}}$ are surjective;
the kernel of $\dbar_{\P^{n-1}}'$ is canonically isomorphic to the tangent space of
$\fM_{0,0}^{\phi}(\P^{n-1},d)$ at~$[f]$.
The exact sequence~\eref{RealOrientCI_e19a}  gives rise to an exact sequence 
$$0\lra \dbar_V' \lra \dbar_{\P^{n-1}}' \lra \dbar_{\cL_{n;\a}} \lra 0$$
of Fredholm operators over~$(\P^1,c)$.
This exact sequence determines an isomorphism
\BE{RealOrientCI_e23a} \big(\!\det\dbar_V'\big)\!\otimes\!\big(\!\det\dbar_{\cL_{n;\a}}\big)
\approx \big(\!\det\dbar_{\P^{n-1}}'\big).\EE
An orientation on the determinant of~$\dbar_V'$ is induced
by the canonical trivialization~\eref{CIprp_e9} via~\eref{ROrientCompl_e2}.
This orientation corresponds to the orientation of
the kernel of~$\dbar_{\cL_{n;\a}}$ relative to $\fM_{0,0}^{\phi}(\P^{n-1},d)$
for the purposes of computing the $\bT^m$-equivariant contributions
to the integral \cite[(4.4)]{RealGWsIII}.\\

\noindent
Let $e\!\in\!\nE_{\R}^{\si}(\Ga)$ be a real edge as in \cite[(4.23)]{RealGWsIII}.
The $(\phi,c)$-real holomorphic map \hbox{$f\!:\P^1\!\lra\!\P^{n-1}$} corresponding to the edge~$e$
is a degree~$\fd(e)$ cover of the line \hbox{$\P^1_{\vt(v_1)\vt(v_2)}\!\subset\!\P^{n-1}$}
through the $\bT^m$-fixed points \hbox{$P_{\vt(v_1)},P_{\vt(v_2)}\!\in\!\P^{n-1}$}
with $\vt(v_2)\!=\!\vt(v_1)\!+\!1$;
it is branched over the points~$P_{\vt(v_1)}$ and~$P_{\vt(v_2)}$ only.
It can be assumed that $\vt(v_2)\!\le\!n'$.
We denote by~$\dbar_{\P^{n'-1}}$ and~$\dbar_{V_c}$ 
the standard real Cauchy-Riemann operators on~$f^*(T\P^{n'-1},\tnd\phi|_{T\P^{n'-1}})$
and $f^*(V_c,\vph_c)$, respectively, and by~$\dbar_{\P^{n'-1}}'$ 
the operator induced by~$\dbar_{\P^{n'-1}}$ on the quotient of the domain of~$\dbar_{\P^{n'-1}}$
by~$G_c(f)$; these operators are surjective.
The exact sequence 
$$0\lra \dbar_{\P^{n'-1}}' \lra \dbar_{\P^{n-1}}' \lra  \dbar_{V_c} \lra 0$$
of Fredholm operators over~$(\P^1,c)$ and the isomorphism~\eref{RealOrientCI_e23a} determine 
an isomorphism
\BE{RealOrientCI_e23} \big(\!\det\dbar_V'\big)\!\otimes\!\big(\!\det\dbar_{\cL_{n;\a}}\big)
\approx \big(\!\det\dbar_{\P^{n'-1}}'\big)\!\otimes\!\big(\!\det\dbar_{V_c}\big).\EE

\vspace{.15in}

\noindent
We now assume that $\nu_n(\a)\!\not\in\!4\Z$ and $\fd(e)\!\not\in\!2\Z$;
the former implies that $n'\!+\!2\!\in\!4\Z$.
Orientations on the kernels (and thus on the determinants) of~$\dbar_{V_c}$ and~$\dbar_{\cL_{n;\a}}$
can be obtained by evaluating sections and their derivatives at $0\!\in\!\P^1$;
see \cite[(5.39)]{RealGWsIII}.
The resulting orientations are called \sf{complex} orientations in~\cite[Section~5.4]{RealGWsIII}.
The $\bT^m$-equivariant Euler classes of the kernels 
of~$\dbar_{V_c}$ and $\dbar_{\cL_{n;\a}}$ with respect to these orientations
are given by~\cite[(5.42)]{RealGWsIII} and~\cite[(5.41)]{RealGWsIII}, respectively.
An orientation on the kernel (and thus on the determinant) of~$\dbar_{\P^{n'-1}}'$
is determined by a trivialization of the real bundle pair 
$f^*(T\P^{n'-1},\tnd\phi|_{T\P^{n'-1}})$ over~$(S^1,c|_{S^1})$.
Such a trivialization is given by \cite[(6.12)]{Teh}.
The $\bT^m$-equivariant Euler class of the kernel of~$\dbar_{\P^{n'-1}}'$
with respect to the resulting orientation is given by~\cite[(5.50)]{RealGWsIII}.
These three orientations induce an orientation on~$\det\dbar_V'$ via~\eref{RealOrientCI_e23}.
The real edge contribution~$\Cntr_{(\Ga,\si);e}$ with respect to the corresponding orientation of
the kernel of~$\dbar_{\cL_{n;\a}}$ relative to $\fM_{0,0}^{\phi}(\P^{n-1},\fd(e)\!)$
at~$[f]$ 
would be \cite[(5.41)]{RealGWsIII} divided by~the product of \cite[(5.50)]{RealGWsIII} 
and~\cite[(5.42)]{RealGWsIII}. 
This ratio is precisely \cite[(4.23)]{RealGWsIII} without the sign
and the order~$\fd(e)$ of the automorphism group of~$f$ in the denominator.\\

\noindent
The sign of the real edge contribution~$\Cntr_{(\Ga,\si);e}$ thus corresponds 
to the comparison of the two orientations on the determinant of~$\dbar_V'$
obtained via~\eref{RealOrientCI_e23a} and~\eref{RealOrientCI_e23} above.
For \hbox{$(\phi,c)\!=\!(\tau_n,\tau)$}, 
the two orientations have been shown to be the same if and only if
the number~$\eps$ given by~\eref{Redgeeps_e} is~even.
The Fredholm setups in the \hbox{$(\phi,c)\!=\!(\tau_n,\tau)$} and 
\hbox{$(\phi,c)\!=\!(\eta_n,\eta)$} cases can be identified by restricting
to the closed unit disk $\ov\D\!\subset\!\C$ as in the proof of \cite[Proposition~3.5]{GZ1}.
Under this identification, the orientations on the determinant of~$\dbar_V'$
obtained via~\eref{RealOrientCI_e23a} 
and~\eref{RealOrientCI_e23} in the $(\tau_n,\tau)$ case
are associated with the corresponding orientations in 
the $(\eta_n,\eta)$ case.
Thus, the two orientations on the determinant of~$\dbar_V'$ in 
the $(\eta_n,\eta)$ case are also the same 
if and only if \hbox{$\eps\!\in\!2\Z$}.
The extra \hbox{$|\phi|\!=\!1$} summand in the $(\eta_n,\eta)$ case is due to the reversal
of the orientation at the end of \cite[Section~3.2]{RealGWsI}.
This establishes the claim made at the end of Section~\ref{EquivLocal_sub}.

\begin{rmk} In~\cite{RealGWsIII}, the conjugation~$\tau_n$ and 
the anti-holomorphic bundle automorphism~$\fc_{\tau}$ are denoted 
by~$\tau_n'$ and~$\fc_{\tau}'$, respectively. 
\end{rmk}

\vspace{.2in}

\noindent
{\it  Institut de Math\'ematiques de Jussieu - Paris Rive Gauche,
Universit\'e Pierre et Marie Curie, 
4~Place Jussieu,
75252 Paris Cedex 5,
France\\
penka.georgieva@imj-prg.fr}\\

\noindent
{\it Department of Mathematics, Stony Brook University, Stony Brook, NY 11794\\
azinger@math.stonybrook.edu}\\

\end{document}